\documentclass[preprint,a4paper,10pt]{article}
\title{Hochschild (co)homology and Koszul duality}
\author{Estanislao Herscovich
\footnote{Departamento de Matem\'atica, FCEyN, Universidad de Buenos Aires, Argentina. 
The author is also a reserch member of CONICET (Argentina). 
On leave of absence from Institut Joseph Fourier, Universit\'e Grenoble I, France.}}
\date{}

\usepackage{amsmath,amsthm,amsfonts,amssymb,stmaryrd}
\usepackage{latexsym}
\usepackage{fullpage}
\usepackage{a4,color,palatino,fancyhdr}
\usepackage{graphicx}
\usepackage{float}  
\usepackage[small, bf, margin=90pt, tableposition=bottom]{caption}
\RequirePackage{amssymb}
\RequirePackage[T1]{fontenc}
\usepackage[breaklinks, naturalnames]{hyperref}
\usepackage{amsrefs}

\input{xy}
\xyoption{all}
\xyoption{poly}
\usepackage[all]{xy}

\newtheorem{theorem}{Theorem}[section]
\newtheorem{proposition}[theorem]{Proposition}

\newtheorem{lemma}[theorem]{Lemma}

\newtheorem{remark}[theorem]{Remark}

\newtheorem{fact}[theorem]{Fact}

\numberwithin{equation}{section}

\def\cl#1{{\langle #1 \rangle}}

\def\place{{-}}

\def\Ker{\mathop{\rm Ker}\nolimits}

\newcommand\ZZ{{\mathbb{Z}}}
\newcommand\NN{{\mathbb{N}}}

\def\place{{-}}

\begin{document}

\maketitle
                                                     
\hrulefill
\begin{abstract}
In this article we discuss two different but related results on Hochschild (co)homology and the theory of Koszul duality. 
On the one hand, we prove essentially that the Tamarkin-Tsygan calculus of an Adams connected augmented dg algebra and of its Koszul dual are dual.  
This uses the fact that Hochschild cohomology and homology may be regarded as a twisted construction of some natural (augmented) dg algebras and dg modules over the former. 
In particular, from these constructions it follows that the computation of the cup product on Hochschild cohomology and cap product on Hochschild homology of a Koszul algebra is directly 
computed from the coalgebra structure of $\operatorname{Tor}_{\bullet}^{A}(k,k)$ (the first of these results is proved differently in \cite{BGSS}). 
We even generalize this situation by studying twisting theory of $A_{\infty}$-algebras to compute the algebra structure of Hochschild (co)homology of more general algebras. 
\end{abstract}



\hrulefill
\section{Introduction}

The aim of our work is to highlight several aspects of the relation between the theory of Hochschild (co)homology and the Koszul theory of Adams connected augmented (dg) algebras. 
We would had supposed that most of what we state was widely known to the experts, but some articles appearing in the literature could imply that this is not exactly the case. 
In particular, we will focus on two different aspects that we shall now describe briefly. 

On the one hand, following B. Keller in \cite{Ke}, given an augmented dg algebra $A$ the \emph{Koszul dual} $E(A)$ is defined as the graded dual of the bar construction $B^{+}(A)$ of $A$. 
If $A = T(V)/\cl{R}$ is a Koszul algebra, with $R \subseteq V^{\otimes 2}$, $E(A)$ is quasi-isomorphic to the usual Koszul dual $A^{!} = T(V^{*})/\cl{R^{\perp}}$, where $R^{\perp} \subseteq (V^{*})^{\otimes 2}$ is the annihilator of $R$ under the identification $(V^{*})^{\otimes 2} \rightarrow (V^{\otimes 2})^{*}$ sending $f \otimes g$ to $(v \otimes w) \mapsto - f(v) g(w)$, 
where $V^{*}$ is considered to be concentrated in cohomological degree $1$. 
In this latter case, it was already observed by B. Feigin and B. Tsygan that there is in fact a duality pair between the Hochschild homology groups of $A$ and $A^{!}$ (only regarded as graded vector spaces). 
In fact, this can be directly deduced from (or following the lines of) the isomorphism between the corresponding cyclic homology groups given in \cite{FT}, Thm. 2.4.1, 
where the authors further suppose that the base field $k$ has characteristic zero, even though this assumption is not strictly necessary if we are interested only in Hochschild homology groups 
(see \cite{Lo} for a more detailed analysis on the corresponding gradings). 
Furthermore, an isomorphism of graded algebras between the Hochschild cohomology groups $HH^{\bullet}(E(A))$ and $HH^{\bullet}(A)$ in case $A$ is also a Koszul algebra was already announced by 
R.-O. Buchweitz in the Conference on Representation Theory held at Canberra on July 2003. 
This result was further generalized by B. Keller in the preprint \cite{Kel}, where he proved that there is in fact a quasi-isomorphism of $B_{\infty}$-algebras between the corresponding Hochschild cohomology cochain complexes. 
On the other hand, Y. F\'elix, L. Menichi and J.-C. Thomas proved in \cite{FMT}, Prop. 5.1 and 5.3, that given a simply connected coaugmented dg coalgebra $C$ over a field, 
there is a isomorphism of Gerstenhaber algebras between the Hochschild cohomology $HH^{\bullet}(C^{\#})$ of the graded dual $C^{\#}$ of $C$ and the Hochschild cohomology 
$HH^{\bullet}(\Omega^{+}(C))$ of the cobar construction of $C$. 
By specializing this result to the case $A = C^{\#}$, and using the obvious result $\Omega^{+}(C) \simeq E(A)$, we get that the Hochschild cohomology of $A$ and $E(A)$ are isomorphic as Gerstenhaber algebras. 
The simply connectedness assumption is however completely unusual in the realm of (generalizations of) Koszul algebras. 
We shall in fact get rid of this assumption and provide a proof (under more typical grading hypotheses for our case) of such an isomorphism. 
Moreover, we shall also prove that the Hochschild homology groups $HH_{\bullet}(A)$ and $HH_{\bullet}(E(A))^{\#}$ are isomorphic Gerstenhaber modules over the corresponding isomorphic 
Gerstenhaber algebras given by Hochschild cohomology (see Theorems \ref{theorem:koszuldual} and \ref{theorem:koszuldual2}). 
In fact, we notice that the Tamarkin-Tsygan calculus of $A$ and of its Koszul dual $E(A)$ are dual (see Remark \ref{rem:tt}). 

On the other hand, given a nonnegatively graded connected algebra $A$ over a field $k$, we know that the homology space $C = \operatorname{Tor}_{\bullet}^{A}(k,k)$ is in fact 
a coaugmented $A_{\infty}$-coalgebra such that its graded dual is quasi-isomorphic to the Koszul dual $E(A)$ of $A$. 
In the case $A$ is (quadratic) Koszul it has been known for a long time that the previous $A_{\infty}$-coalgebra structure reduces to a dg coalgebra with zero differential (this in fact being equivalent to the Koszul property). 
This knowledge suffices to explicitly compute the graded algebraic structure on the Hochschild cohomology $HH^{\bullet}(A)$ by using the minimal projective resolution of $A$ as $A$-bimodules. 
Indeed, by fairly general arguments on twists of dg algebras we see that there is a quasi-isomorphism of augmented dg algebras between the Hochschild cochain complex of $A$ and 
a twist of the dg algebra hom space $\mathcal{H}om(C,A)$. 
The same arguments may be applied to compute the module structure on the Hochschild homology $HH_{\bullet}(A)$ given by the cap product over the graded algebra $HH^{\bullet}(A)$ (see Theorem \ref{theorem:koszul} and the last paragraph of Subsubsection \ref{subsubsec:applkos}). 
The results concerning the computation of the algebra structure of Hochschild cohomology of a Koszul algebra give in fact a direct proof of the main result of \cite{BGSS} 
(stated at the introduction, p. 443, or after as Theorem 2.3), which is just the basis-dependent expression of the result we stated in Theorem \ref{theorem:koszul} for the particular case of Koszul algebras. 
We further generalize this situation by making use of twists of $A_{\infty}$-algebras in Section \ref{sec:ainf}. 
More precisely, we deduce an explicit recipe to compute the graded algebra structure of the Hochschild (co)homology of an augmented dg algebra $A$ over a field provided with a minimal coaugmented $A_{\infty}$-coalgebra quasi-isomorphic to the bar construction of $A$ (see Theorem \ref{theorem:final}). 
The case of cohomology gives in particular another proof of the main result of \cite{XX}. 
This may be particularly useful for multi-Koszul algebras, or any algebra for which there is an explicit description of the coaugmented $A_{\infty}$-coalgebra structure on $\operatorname{Tor}_{\bullet}^{A}(k,k)$ (in the former two cases, for it is just the graded dual of the augmented $A_{\infty}$-algebra structure on $\mathcal{E}xt_{\bullet}^{A}(k,k)$). 

The article is organised as follows. 
In Section \ref{sec:prel} we recall the basic definitions and constructions of graded and dg modules over a base ring $k$, and of graded and dg (co)modules over graded and dg (co)algebras, respectively. 
Even though the contents of this part are completely standard, they are needed for explicitly stating the notation we are going to use, as well as the main sign conventions we shall follow. 
The following section contains the basic rudiments of Hochschild (co)homology theory for (unitary) dg algebras over a field $k$, together with the basic bar and cobar constructions for augmented dg algebras and coaugmented 
dg coalgebras, respectively. 
We explicitly recall the (supposedly well-known) relation between Hochschild (co)homology theory and twists of augmented dg algebras (see Fact \ref{fact:hh}). 
The aim of this long section is to provide a coherent explanation of all the formulas we will use to prove our main results  in a completely explicit manner, where we have worked out all the corresponding signs. 
We believe the experts should be aware of the contents of the second section and they could skip them if they want to. 
At the end of Section \ref{sec:hoch} we provide a direct corollary that allows to compute the multiplicative structure on Hochschild cohomology and the corresponding module structure on Hochschild homology of a Koszul algebra, which in our situation is just a consequence of the way we presented the Hochschild (co)homology complexes (see Theorem \ref{theorem:koszul}). 

In Section \ref{sec:kosdual} we use the tools appearing in the previous sections in order to prove one of the main results of the article: the Tamarkin-Tsygan calculus of an Adams connected augmented dg algebra 
and of its Koszul dual are dual. 
The last section is devoted to apply the theory of twisted (augmented) $A_{\infty}$-algebras to obtain a method to explicitly calculate the graded algebra structure of the Hochschild (co)homology of an augmented dg algebra $A$ over a field provided with a minimal coaugmented $A_{\infty}$-coalgebra quasi-isomorphic to the bar construction of $A$ (see Theorem \ref{theorem:final}). 

We would also like to add a word of warning concerning the reading of this manuscript. 
There are several possible conventions for the objects we shall consider, concerning signs, definitions of morphisms, etc. 
For this reason we have tried to be as careful as possible explaining all the necessary details in order make all our conventions precise, which on the other hand has considerably increased the extension of the text. 
We believe however that this amount of precision is indeed necessary, and (we thus hope that) the article may be also regarded as a possible source of a careful choice of harmonized signs and structures conventions. 
Since the check of the validity of these conventions is often lengthy but straightforward, we shall usually leave the reader to control that the (precise) conventions (\textit{e.g.} signs or commutativity of diagrams) hold. 
We shall however sometimes sketch the main ingredients of the mentioned proof in case we think they are necessary. 

The author would like to express his deep gratitude to Professor Clas L\"ofwall for his careful explanation (by e-mail) of the results by Feigin and Tsygan mentioned before, 
with special care on the sign and grading conventions. 

\section{Preliminaries on basic algebraic structures}
\label{sec:prel}

We recall the following basic facts, which will also establish the notation. 
From now on, $k$ will denote a commutative ring with unit (which we also consider as a unitary graded ring concentrated in degree zero). 
By \emph{module} over $k$ we will always mean a symmetric bimodule over $k$ 
(although several constructions can be clearly performed without this symmetry assumption, we shall suppose it in order to simplify the exposition). 
We fix an abelian group $G$ of the form $\ZZ \times G'$ (or also $\ZZ/2.\ZZ \times G'$ to which all these construction can be adapted straightforward), which we write additively. 
The character map used in the Koszul sign rule will be just given by the projection on the first component of $G$. 
A typical element of $G$ will be denoted by $g$, $h$, etc, and the corresponding first component $i_{g}$, $i_{h}$, etc. 
For an object $M$, we will denote by $\mathrm{id}_{M}$ the identity endomorphism of $M$. 
We also remark that the expression \emph{map} between to graded or dg modules over $k$ (or maybe provided with further structure) will always mean the mapping between 
(say) the underlying modules or even the underlying sets, which comes from forgetting all the extra structure. 
This might be sometimes useful if we want to stress just the values of morphisms at elements of a graded or dg module over $k$. 

\subsection{\texorpdfstring{Graded and differential graded modules over a fixed commutative ring $k$}{Graded and differential graded modules over a fixed commutative ring k}} 
\label{subsec:int}

A \emph{(cohomological) graded module} over $k$ is a module over $k$ provided with a decomposition of $k$-modules of the form $M = \oplus_{g \in G} M^{g}$.
If $m$ is a nonzero homogeneous element of a graded module $M$ over $k$ we define the \emph{degree} $\deg m \in \ZZ$ of $m$ 
and the \emph{weight (internal degree, or Adams degree)} $w(m) \in G'$ of $m$ by $m \in M^{(\deg m, w(m))}$. 
We say in this case that the \emph{complete degree} of $m$ is $(\deg m, w(m)) \in G$. 
The commutative ring $k$ will be considered as a graded module with the grading given by $k^{g} = 0$ if $g \neq 0_{G}$, and $k^{0_{G}} = k$. 
If $M$ is a graded module and $g \in G$, define $M[g]$ to be the graded module over $k$ with the same underlying structure of $k$-module 
but with a complete degree shift given by $M[g]^{g'} = M^{g+g'}$, for all $g' \in G$. 
We shall usually write $M[i]$, for $i \in \ZZ$, instead of the more correct $M[(i,0_{G'})]$, for it causes no confusion. 
For any two graded modules $M$ and $N$ over $k$, $\operatorname{hom}_{k}(M,N)$ is the space of $k$-linear maps of complete degree zero, 
\textit{i.e.} $f(M^{g}) \subseteq N^{g}$ for all $g \in G$. 
The internal space of morphisms if given by $\mathcal{H}om_{k}(M,N) = \oplus_{g \in G} \operatorname{hom}_{k} (M,N[g])$ and it is obviously a graded $k$-module. 
The graded module $\mathcal{H}om_{k} (M,k)$ will be also denoted by $M^{\#}$. 
Note that in this case the $g$-th graded component of $M^{\#}$ is given by $(M^{-g})^{*}$, where $(\place)^{*}$ denotes the usual dual for modules over $k$, and by the previous comments 
we have that $(M[g])^{\#} = (M^{\#})[-g]$, for $g \in G$. 

We remark that by very definition the graded module $\mathcal{H}om_{k}(M[g],N[g'])$ exactly coincides with $\mathcal{H}om_{k}(M,N)[g'-g]$ for $g, g' \in G$.  
In the same manner, the graded modules $M[g] \otimes N[g']$ and $(M \otimes N)[g+g']$, for $g, g' \in G$, are also exactly coincident. 
These ``identities'' are however misleading since they do not (in general) respect the Koszul sign rule, and -in some sense more fundamentally-
the mentioned phenomenon for the homomorphisms spaces is not in accordance with the axioms of category theory. 
We will regard such coincidences only as a consequence of the usual abuse of notation in the definitions of tensor product and morphisms spaces:
since we are interested in considering the Koszul sign rule, we should in fact force them to be noncoincident. 
There is however an identification (and in fact many of them, but in general different from the identity) between the corresponding previous graded modules, which is compatible 
with the Koszul sign rule, and that will be explained in the penultimate paragraph of this subsection.

Given $M$ and $N$ two graded modules over $k$, \emph{a morphism of graded modules of complete degree $g \in G$} is an element 
$f \in \mathcal{H}om_{k}(M,N)$ of complete degree $g$. 
We remark that the map $s_{M,g} : M \rightarrow M[g]$ given by the identity is a morphism of graded modules of degree $-g$, 
and we shall typically denote $s_{M,(1,0_{G'})}$ by $s_{M}$, or simply $s$, if $M$ is clear from the context. 
All along this article, if we do not indicate the complete degree of a morphism (between graded modules, or later on dg modules, etc), it means that it is of complete degree zero. 
Also, for $M$ and $N$ two graded modules over $k$, the (usual) tensor product $M \otimes_{k} N$ has the structure of graded module over $k$ with 
$(M \otimes_{k} N)^{g} = \oplus_{g' \in G} M^{g'} \otimes_{k} N^{g-g'}$. 
From now on, all unadorned tensor products $\otimes$ will mean $\otimes_{k}$. 
For $f : M \rightarrow N$ and $f' : M' \rightarrow N'$ two morphisms of graded modules over $k$ of complete degrees $g$ and $g'$, respectively, the map 
$f \otimes f' : M \otimes M' \rightarrow N \otimes N'$ given by $(f \otimes f')(m \otimes m') = (-1)^{\deg f' \deg m} f(m) \otimes f'(m')$, for $m \in M$ and $m' \in M'$ homogeneous,
 is a morphism of graded modules over $k$ of complete degree $g+g'$. 
Analogously, if $f : M \rightarrow N$ and $f' : N' \rightarrow M'$ are two morphisms of graded modules over $k$ of complete degrees $g$ and $g'$, respectively, the map 
$\mathcal{H}om_{k}(f,f') : \mathcal{H}om_{k} (N',N) \rightarrow \mathcal{H}om_{k} (M',M)$ given by $\phi \mapsto (-1)^{\deg f (\deg \phi + \deg f')} f' \circ \phi \circ f$, for $\phi$ homogeneous, 
is a morphism of complete degree $g+g'$. 
We shall also denote $\mathcal{H}om_{k}(f,N) = \mathcal{H}om_{k}(f,\mathrm{id}_{N})$ and $\mathcal{H}om_{k}(N',f') = \mathcal{H}om_{k}(\mathrm{id}_{N'},f')$. 
Furthermore, we will usually denote $\mathcal{H}om_{k}(f,k)$ by $f^{\#}$, which is of course has the same complete degree as the one of $f$. 
As for the case of the tensor product we shall usually omit the commutative ring $k$ in the notation of the homomorphism groups introduced before, and proceed to write $\mathcal{H}om$ 
instead of $\mathcal{H}om_{k}$. 
The canonical map $\iota_{M} : M \rightarrow (M^{\#})^{\#}$ defined as $\iota_{M}(m)(f) = (-1)^{\deg m \deg f} f(m)$, for $m \in M$ and $f \in M^{\#}$ homogeneous, is a morphism of graded modules. 
Given $M$ and $N$ two graded modules over $k$, we will occasionally consider the morphism $\iota_{M,N} : M^{\#} \otimes N^{\#} \rightarrow (M \otimes N)^{\#}$ of graded modules defined as $\iota_{M,N} (\phi \otimes \psi)(m \otimes n) = (-1)^{\deg \psi \deg m} \phi(m) \psi(n)$. 
We also have the \emph{flip} $\tau_{M,N} : M \otimes N \rightarrow N \otimes M$, which is the morphism of graded modules defined as 
$\tau_{M,N} (m \otimes n) = (-1)^{\deg m \deg n} n \otimes m$, for all $m \in M$ and $n \in N$ homogeneous elements. 

A \emph{differential graded module (or dg module)} over $k$ is a graded $k$-module $M = \oplus_{g \in G} M^{g}$ together with a homogeneous $k$-linear map 
$d_{M} : M \rightarrow M$ of degree $+1$ and zero weight, \textit{i.e.} $d_{M}(M^{g}) \subseteq M^{g+(1,0_{G'})}$ for all $g \in G$, such that it is a differential, \textit{i.e.} $d_{M}^{2} = 0$. 
The graded module structure on $k$ explained before can be extended to a dg module by defining the differential $d_{k} = 0$, and more generally, any graded $k$-module $M$ may be regarded as dg module 
with vanishing differential. 
For a dg module $M$ over $k$, the \emph{cohomology} $H^{\bullet}(M)$ of $M$, given by the quotient $\operatorname{Ker}(d_{M})/\operatorname{Im}(d_{M})$, is in fact a graded module over $k$. 
A dg module $M$ is called \emph{acyclic} if $H^{\bullet}(M)$ vanishes. 
If $M$ is a dg module and $g \in G$, $M[g]$ is the dg module over $k$ with the same graded module structure as before and differential 
$d_{M[g]} = (-1)^{i_{g}} d_{M}$. 
For $M$ and $N$ two dg modules over $k$, the tensor product $M \otimes N$ has the structure of dg module over $k$ with the same underlying graded structure as before 
and with differential $d_{M \otimes N} = d_{M} \otimes \mathrm{id}_{N} + \mathrm{id}_{M} \otimes d_{N}$. 
We endow the graded $k$-module $\mathcal{H}om(M,N)$ with the differential $d_{\mathcal{H}om(M,N)}(f) = d_{N} \circ f - (-1)^{\deg f} f \circ d_{M}$, so it becomes a dg module over $k$. 
In this case, for $M$ a dg module we will still denote by $M^{\#}$ the dg module $\mathcal{H}om(M,k)$. 
Note that $d_{M^{\#}} = - d_{M}^{\#}$. 

Given $M$ and $N$ two dg modules over $k$, $f : M \rightarrow N$ is a \emph{morphism of differential graded modules over $k$ of complete degree $g$} 
if it is a morphism between the underlying graded modules of complete degree $g$ 
and satisfies that $d_{N} \circ f = (-1)^{i_{g}} f \circ d_{M}$, \textit{i.e.} it is cocycle of complete degree $g$ of the dg module $\mathcal{H}om(M,N)$. 
We stress that, as before, if we do not specify the complete degree of a morphism, it will be assumed to be zero. 
Note that $s_{M,g} : M \rightarrow M[g]$ introduced previously is in fact a morphism of dg modules of complete degree $- g$.  
This is in fact tantamount to the definition of dg module structure over $k$ on $M[g]$. 
We stress that if $f : M \rightarrow N$ and $f' : M' \rightarrow N'$ are two morphisms of dg modules over $k$ of complete degree $g$ and $g'$, respectively, then 
$f \otimes f' : M \otimes M' \rightarrow N \otimes N'$ is a morphism of dg modules over $k$ of complete degree $g+g'$. 
Analogously, given $f : M \rightarrow N$ and $f' : N' \rightarrow M'$ two morphisms of dg modules over $k$ of complete degrees $g$ and $g'$, respectively, the map 
$\mathcal{H}om(f,f') : \mathcal{H}om (N',N) \rightarrow \mathcal{H}om (M',M)$ defined for graded modules is moreover a morphism of of dg modules over $k$ of complete degree $g+g'$. 
Note that given two dg $k$-modules $M$ and $N$, the maps $\iota_{M}$, $\iota_{M,N}$ and $\tau_{M,N}$ defined in the third paragraph of this section are further morphisms of dg modules over $k$. 

If $f : M \rightarrow N$ is a morphism of dg modules over $k$ of complete degree $g$, then, for each $g', g'' \in  G$, the may consider the map 
\[     \mathcal{H}om(s_{M,g'}^{-1},s_{N,g''})(f) : M[g'] \rightarrow N[g''],     \]
which will be denoted by $f_{[g']}^{[g'']}$. 
It is trivial to see that $f_{[g']}^{[g'']}$ a morphism of dg modules over $k$ of complete degree $g+g'-g''$, \textit{i.e.} 
\[     d_{N[g']} \circ f = (-1)^{i_{g+g'-g''}} f \circ d_{M[g']}.     \] 
If $f$ is a morphism of complete degree $0_{G}$, it is rather usual to allow the abuse of notation given by denoting the map $(-1)^{\deg g' \deg g''} f_{[g']}^{[g'']}$ also by $f$. 
We shall only follow this convention when we consider it is unambiguous. 
We also remark that (unlike the case for graded modules) the dg modules $\mathcal{H}om(M[g],N[g'])$ and $\mathcal{H}om(M,N)[g'-g]$, for $g, g' \in G$, are not the same, 
for the identity map between the underlying graded modules is not a morphism of dg modules over $k$. 
Indeed, the corresponding isomorphism of dg modules from $\mathcal{H}om(M,N)[g'-g]$ to $\mathcal{H}om(M[g],N[g'])$, which we denote by $H_{M,N,g,g'}$, is given by 
$s_{\mathcal{H}om(M,N),g'-g}(f) \mapsto \mathcal{H}om(s_{M,g}^{-1},s_{N,g'})(f)$, for $f \in \mathcal{H}om(M,N)$. 
The underlying map is thus the identity times a $(-1)^{(\deg f +i_{g'}) i_{g}}$ sign. 
In the same manner, the dg module structure on the tensor product and $M[g'] \otimes N[g']$ and $(M \otimes N)[g+g']$, for $g, g' \in G$, are not same. 
There is though a (not completely canonical) isomorphism of dg modules $M[g'] \otimes N[g'] \rightarrow (M \otimes N)[g+g']$ over $k$, denoted by $T_{M,N,g,g'}$, defined as 
$s_{M,g}(m) \otimes s_{N,g'}(n) \mapsto (-1)^{i_{g'} \deg m} s_{M \otimes N,g+g'}(m \otimes n)$, for $m \in M$ and $n \in N$ homogeneous elements.  

If $f : M \rightarrow N$ is a morphism of dg modules, the \emph{cone} $\operatorname{cone}(f)$ is the dg module whose underlying graded module is $M[1] \oplus N$ and whose differential is given by 
$d_{\operatorname{cone}(f)} (m,n) = (-d_{M}(m)+f(m),d_{N}(n))$.
Given a morphism of dg modules $f : M \rightarrow N$ of complete degree $g$, it directly induces a morphism of graded modules $H^{\bullet}(M) \rightarrow H^{\bullet}(N)$ of the same complete degree, which we will denote by $H^{\bullet}(f)$.  
It is clear that $H^{\bullet}(\mathrm{id}_{M}) = \mathrm{id}_{H^{\bullet}(M)}$ and that $H^{\bullet}(f \circ f') = H^{\bullet}(f) \circ H^{\bullet}(f')$, for any two composable morphisms $f$ and $f'$ of dg modules of complete degrees $g$ and $g'$, resp. 
Furthermore, a morphism of dg modules $f : M \rightarrow N$ of complete degree $0_{G}$ is said to be a \emph{quasi-isomorphism} if $H^{\bullet}(f)$ is an isomorphism of graded modules. 
It is well-known that $f$ is a quasi-isomorphism if and only if $\operatorname{cone}(f)$ is acyclic (see \cite{W}, Cor. 1.5.4). 

\subsection{Graded and differential graded algebras and coalgebras, and modules over the former} 
\label{subsec:int2}

A \emph{(nonunitary) graded algebra} over $k$ is just a (nonunitary) algebra over $k$ together with a decomposition of $k$-modules $A = \oplus_{g \in G} A^{g}$ 
satisfying that $A^{g} A^{g'} \subseteq A^{g+g'}$, for all $g, g' \in G$. 
We will also sometimes denote the product of $A$ by (the morphism of graded modules) $\mu_{A} : A \otimes A \rightarrow A$. 
A \emph{morphism of graded algebras} from a graded algebra $A$ to a graded algebra $B$ is a morphism of graded modules $f : A \rightarrow B$ such that $f(a a') = f(a) f(a')$ for all $a, a' \in A$.
A \emph{unitary graded algebra} over $k$ is a nonunitary one together with an element $1_{A} \in A^{0_{G}}$, called the \emph{unit} of $A$, satisfying the usual axiom $1_{A} a = a 1_{A} = a$ for all $a \in A$. 
We may also consider the unit of $A$ as a morphism of graded modules $\eta_{A} : k \rightarrow A$ which satisfies that $\mu_{A} \circ (\mathrm{id}_{A} \otimes \eta_{A})$ and 
$\mu_{A} \circ (\eta_{A} \otimes \mathrm{id}_{A})$ coincide with the canonical isomorphisms $A \otimes k \rightarrow A$ and $k \otimes A \rightarrow A$, resp.
Given two unitary graded algebras $A$ and $B$, a \emph{morphism of unitary graded algebras} is a morphism of the underlying nonunitary graded algebras $f : A \rightarrow B$ such that $f(1_{A}) = 1_{B}$. 
The \emph{opposite graded algebra} $A^{\mathrm{op}}$ of a nonunitary graded algebra $A$ is given by the same graded module over $k$ but with the product $a \cdot_{\mathrm{op}} b = (-1)^{\deg a \deg b} b a$, for all $a, b \in A$ homogeneous. 
In case $A$ is unitary, $A^{\mathrm{op}}$ also, with the same unit of $A$. 
If $A$ and $B$ are two nonunitary graded algebras, the graded module structure over $k$ of the tensor product $A \otimes B$ is also a graded algebra with the product 
$(a \otimes a') (b \otimes b') = (-1)^{\deg a' \deg b} ab \otimes a' b'$. 
If $A$ and $B$ are unitary with units $1_{A}$ and $1_{B}$, resp., then $A \otimes B$ is also unitary with unit $1_{A} \otimes 1_{B}$. 
We consider the graded algebra $A^{e} = A \otimes A^{\mathrm{op}}$, which is called the \emph{enveloping algebra} of $A$.  

We also have the dual definitions. 
A \emph{(noncounitary) graded coalgebra} over $k$ is a graded module $C = \oplus_{g \in G} C^{g}$ together with a morphism of graded modules $\Delta_{C} : C \rightarrow C \otimes C$ satisfying the coassociativity axiom $(\Delta_{C} \otimes \mathrm{id}_{C}) \circ \Delta_{C} = (\mathrm{id}_{C} \otimes \Delta_{C}) \circ \Delta_{C}$. 
We define $\Delta_{C}^{(n)} : C \rightarrow C^{\otimes n}$ as the composition $(\Delta_{C}^{(n-1)} \otimes \mathrm{id}_{C}) \circ \Delta_{C}$, for $n \in \NN_{\geq 3}$, and $\Delta_{C}^{(2)} = \Delta_{C}$. 
As usual, we may use the \emph{Sweedler notation} $\Delta_{C}^{(n)}(c) = c_{(1)} \otimes \dots \otimes c_{(n)}$ for the iterated coproduct of an element $c \in C$ (by the coassociativity axiom this notation is consistent). 
A \emph{morphism of graded coalgebras} from a graded coalgebra $C$ to a graded algebra $D$ is a morphism of graded modules $f : C \rightarrow D$ such that $\Delta_{D} \circ f = (f \otimes f) \circ \Delta_{C}$.
A graded coalgebra $C$ is called \emph{counitary} if there is a morphism of graded modules $\epsilon_{C} : C \rightarrow k$, called the \emph{counit} of $C$, satisfying that 
$(\epsilon_{C} \otimes \mathrm{id}_{C}) \circ \Delta_{C}$ and $(\mathrm{id}_{C} \otimes \epsilon_{C}) \circ \Delta_{C}$ coincide with the canonical isomorphisms $C \simeq k \otimes C$ and $C \simeq C \otimes k$, resp. 
Given two counitary graded coalgebras $C$ and $D$, a \emph{morphism of counitary graded coalgebras} is a morphism of the underlying noncounitary graded coalgebras $f : C \rightarrow D$ such that $\epsilon_{D} \circ f = \epsilon_{C}$. 
The \emph{coopposite graded coalgebra} $C^{\mathrm{coop}}$ of a noncounitary graded coalgebra $C$ is given by the same graded module over $k$ but with coproduct $\Delta_{C^{\mathrm{coop}}} = \tau_{C,C} \circ \Delta_{C}$. 
If $C$ is counitary, $C^{\mathrm{coop}}$ is also, with the same counit as the one of $C$. 
If $C$ and $D$ are two noncounitary graded coalgebras, the graded module structure over $k$ of the tensor product $C \otimes D$ is also a graded coalgebra with the coproduct 
$\Delta_{C \otimes D} = (\mathrm{id}_{C} \otimes \tau_{C,D} \otimes \mathrm{id}_{D}) \circ (\Delta_{C} \otimes \Delta_{D})$. 
If $C$ and $D$ are counitary with counits $\epsilon_{C}$ and $\epsilon_{D}$, resp., then $C \otimes D$ is also counitary with counit $\epsilon_{C} \otimes \epsilon_{D}$. 
The graded coalgebra $C \otimes C^{\mathrm{coop}}$ is called the \emph{enveloping coalgebra} of $C$, and is denoted by $C^{e}$. 

A left (resp., right) \emph{graded module} over a nonunitary graded algebra $A$ is just a left (resp., right) module over $A$ such that it is a graded module over $k$ for the action of $k$ 
given by restriction (\textit{i.e.} provided with a decomposition of $k$-modules of the form $M = \oplus_{g \in G} M^{g}$) satisfying that $A^{g'} M^{g} \subseteq M^{g'+g}$ 
(resp., $M^{g} A^{g'} \subseteq M^{g+g'}$), for all $g, g' \in G$. 
If $A$ is unitary, we say that $M$ is a left (resp., right) graded module if we further assume that $1_{A} m = m$ (resp., $m 1_{A} = m$) for all $m \in M$. 
A \emph{graded bimodule} over $A$ will be just a left graded module over the enveloping algebra $A^{e}$. 

In the rest of this subsection, unless further explanation is required, we shall usually refer to the term graded algebra (resp., graded coalgebra), differential graded algebra 
(resp., differential graded coalgebra), module over a graded algebra (resp., comodule over a graded coalgebra), etc. without explicitly indicating whether there is a unit (resp., counit) 
or not to indicate that the definitions and constructions apply to each possibility in the sense that either the adjective nonunitary should be applied to them altogether, or else the adjective unitary. 

If $M$ is a left (resp., right) graded module and $g \in G$, define $M[g]$ to be the left (resp., right) graded module over $A$ with the new action of $A$ given by $a \cdot m = (-1)^{i_{g} \deg a} a m$ 
(resp., with the same action) such that the complete degree shifts as $M[g]^{g'} = M^{g+g'}$. 
Note that, if $M$ is a left (resp., right) graded module over $A$, then it is also a right (resp., left) graded module over $A^{\mathrm{op}}$ with the same structure of graded module over $k$ and right 
(resp., left) action $m a = (-1)^{\deg a \deg m} a m$ (resp., $a m = (-1)^{\deg a \deg m} m a$) over $A^{\mathrm{op}}$. 
For any two left (resp., right) graded modules $M$ and $N$ over $A$, $\operatorname{hom}_{A}(M,N)$ is the space of $A$-linear maps of complete degree zero, and 
$\mathcal{H}om_{A}(M,N) = \oplus_{g \in G} \hom_{A}(M,N[g])$, which is obviously a graded $k$-module. 
Note that, if $M$ and $N$ are left (resp., right) graded modules over $A$, this implies that $\mathcal{H}om_{A}(M,N)$ is the subspace of $\mathcal{H}om(M,N)$ given by sums of 
homogeneous maps satisfying that $f (a m) = (-1)^{\deg f \deg a} a f(m)$ (resp., $f(m a) = f(m) a$), for $a \in A$ and $m \in M$ homogeneous elements. 
These latter are called \emph{morphisms of graded left (resp., right) $A$-modules} (of some complete degree). 
Notice that the graded left (resp., right) $A$-module structure on $M[g]$ is tantamount to requiring that the map $s_{M,g} : M \rightarrow M[g]$ is a morphism graded left (resp., right) $A$-modules. 
We may point out that there are similar definitions of graded comodules over graded coalgebras, to which the previous constructions also apply \textit{mutatis mutandi}. 
Since we will not need these, we do not provide such definitions, but we let the interested reader to elaborate on them.  

A \emph{nonunitary (resp., unitary) differential graded algebra (or dg algebra)} over $k$ is a nonunitary (resp., unitary) graded algebra over $k$ together with a homogeneous $k$-linear map 
$d_{A} : A \rightarrow A$ of complete degree $(1,0_{G'})$ satisfying the Leibniz identity, \textit{i.e.} $d_{A}(a b) = d_{A}(a) b + (-1)^{\deg a} a d_{A}(b)$, for all $a, b \in A$ homogeneous, and  $d_{A}^{2} = 0$ 
(resp., $d_{A}(1_{A}) = 0$ and $d_{A}^{2} = 0$).  
As in the case of unitary graded algebras, we may also consider the unit of $A$ as a morphism of dg modules $\eta_{A} : k \rightarrow A$ which satisfies the same axioms as before.
Note that the dg module structure on $k$ stated before is compatible with its structure of unitary algebra, turning $k$ into a unitary dg algebra.
The graded $k$-module given by the cohomology $H^{\bullet}(A)$ of $A$ is in fact a nonunitary (resp., unitary) graded algebra with the product induced by that of $A$ (resp., and the unit of $H^{\bullet}(A)$ is the cohomology class of the unit of $A$). 
Note that if $A$ is a dg algebra over $k$, then the opposite graded algebra together with the same differential $d_{A}$ is also a dg algebra over $k$. 
Analogously, for $A$ and $B$ two dg algebras, the dg module structure over $k$ of the tensor product $A \otimes B$ with the product (and unit if $A$ and $B$ are unitary) described above for graded algebras is also a dg algebra. 
In this case, the enveloping algebra $A^{e}$ of a dg algebra $A$ is also a dg algebra. 

A \emph{noncounitary (resp., counitary) differential graded coalgebra (or dg coalgebra)} over $k$ is a noncounitary (resp., counitary) graded coalgebra $C$ over $k$ provided with a morphism of graded $k$-modules $d_{C} : C \rightarrow C$ of complete degree $(1,0_{G'})$ satisfying that $\Delta_{C} \circ d_{C} = (\mathrm{id}_{C} \otimes d_{C} + d_{C} \otimes \mathrm{id}_{C}) \circ \Delta_{C}$, and $d_{C}^{2} = 0$ (resp., $\epsilon_{C} \circ d_{C} = 0$ and $d_{C}^{2} = 0$). 
Note also that the canonical isomorphism $k \rightarrow k \otimes k$ turns the dg module $k$ into a dg coalgebra, 
which is further counitary by setting $\epsilon_{k} = \mathrm{id}_{k}$. 
If $k$ is Von Neumann regular, the graded $k$-module given by the cohomology $H^{\bullet}(C)$ of $C$ has a coproduct (resp., and a counit) induced by that of $C$, by the K\"unneth formula, 
so it becomes a noncounitary (resp., counitary) graded coalgebra. 
If $C$ is a dg coalgebra over $k$, then the coopposite graded coalgebra together with the same differential $d_{C}$ is also a dg coalgebra over $k$. 
Analogously, for $C$ and $D$ two dg coalgebras, the dg module structure over $k$ of the tensor product $C \otimes D$ with the coproduct (and counit if $C$ and $D$ are counitary) described above for graded coalgebras is also a dg coalgebra. 
As for the case of algebras, the enveloping coalgebra $C^{e}$ of a dg coalgebra $C$ is also a dg coalgebra. 

A left (resp., right) \emph{differential graded module (or dg module)} over a dg algebra $A$ is a left (resp., right) graded $A$-module $M = \oplus_{g \in G} M^{g}$ 
such that it is also a dg module over $k$, 
for the action of $k$ coming from restriction (\textit{i.e.} together with a homogeneous $k$-linear map $d_{M} : M \rightarrow M$ of complete degree $(1,0_{G'})$, such that $d_{M}^{2} = 0$), 
which satisfies the Leibniz identity, \textit{i.e.} $d_{M}(a m) = d_{A}(a) m + (-1)^{\deg a} a d_{M}(m)$ 
(resp., $d_{M}(m a) = d_{M}(m) a + (-1)^{\deg m} m d_{A}(a)$), for all $a \in A$ and $m \in M$ homogeneous. 
If $M$ is a left (resp., right) dg module over $A$ and $g \in G$, $M[g]$ is the left (resp., right) dg module over $A$ with the same graded module structure over $A$ 
as defined previously and differential given by its structure of dg module over $k$, \textit{i.e.} $d_{M[g]} = (-1)^{i_{g}} d_{M}$. 
Note that for any two left (resp., right) dg modules $M$ and $N$ over a dg algebra $A$ the space $\mathcal{H}om_{A}(M,N)$ is obviously a dg $k$-module for 
$d_{\mathcal{H}om_{A}(M,N)}(f) = d_{N} \circ f - (-1)^{\deg f} f \circ d_{M}$. 
A \emph{morphism of differential graded modules over $A$ of complete degree $g$} is an element $f \in \mathcal{H}om_{A}(M,N)$ of degree $d$ satisfying that 
$d_{N} \circ f = (-1)^{i_{g}} f \circ d_{M}$, \textit{i.e.} it is cocycle of complete degree $g$ of the dg $k$-module $\mathcal{H}om_{A}(M,N)$. 
Note that $s_{M,g}$ is a morphism of dg modules over $A$, for any dg $A$-module $M$ and $g \in G$. 
As in the case of dg modules over $k$, if $f : M \rightarrow N$ is a morphism of dg modules over $A$ of complete degree $g$, then, 
for each $g', g'' \in \ZZ$, we may consider  
\[     f_{[g']}^{[g'']} : M[g'] \rightarrow N[g''],     \] 
which is a morphism of dg modules over $A$ of degree $g+g'-g''$.
Also, notice that, if $M$ is a left (resp., right) dg module over $A$, then it is also a right (resp., left) dg module over $A^{\mathrm{op}}$ 
with the same structure of dg module over $k$ and right (resp., left) action as in the case of graded modules over $A^{\mathrm{op}}$. 
Indeed, it is trivial to check that this satisfies the Leibniz identity, so it defines a structure of dg module over $A$. 
A \emph{differential graded bimodule (or dg bimodule)} over $A$ is defined as a left dg module over the enveloping algebra $A^{e}$. 
As before, we endow the graded $k$-module $\mathcal{H}om_{A}(M,N)$ with differential $d(f) = d_{N} \circ f - (-1)^{\deg f} f \circ d_{M}$, so it becomes a dg module over $k$.   
Again, we notice that there are similar definitions of differential graded comodules over dg coalgebras, to which the previous constructions also apply straightforward, 
but we will not give them for they are not going to be required.  

Let $M$ be a dg module over a dg algebra $A$. 
It is called \emph{free} if it is isomorphic to a direct sum of dg modules over $A$ of the form $A[g_{i}]$, for a family $\{ g_{i} : i \in I \}$ of elements of $G$, 
where $I$ is a set of indices. 
We say that $M$ is \emph{semi-free} if there is an increasing filtration $\{ M_{i} \}_{i \in \NN_{0}}$ of dg submodules of $M$ over $A$ such that that $M_{0} = 0$ (\textit{i.e.} the filtration is \emph{Hausdorff}), $\cup_{i \in \NN_{0}} M_{i} = M$ (\textit{i.e.} the filtration is \emph{exhaustive}) and 
$M_{i+1}/M_{i}$ is a free dg module over $A$ for all $i \in \NN_{0}$ (see \cite{AH}, Subsection 1.11, (4)). 
Equivalently, $M$ is semi-free if there exists a set $\mathcal{B} \subseteq M$ of homogeneous elements which gives a basis of the underlying graded module of $M$ over the underlying graded algebra of $A$ with the following property. 
For any $S \subseteq \mathcal{B}$, let $\delta(S) \subseteq \mathcal{B}$ be the smallest subset among all of the subsets $T$ of $\mathcal{B}$ such that $d(S)$ is included in the 
$A$-linear span of $T$. 
Then, the previously mentioned property is that for every $b \in \mathcal{B}$, there is $n \in \NN$ such that $\delta^{n}(\{ b \}) = \emptyset$ (see \cite{AFH}, Prop. 8.2.3). 
It is a very simple exercise to prove that if $M$ is provided with an increasing Hausdorff and exhaustive filtration $\{ M_{i} \}_{i \in \NN_{0}}$ of dg submodules of $M$ over $A$ such that 
$M_{i+1}/M_{i}$ is a semi-free dg module over $A$ for all $i \in \NN_{0}$, then $M$ is also semi-free (see \cite{AFH}, Cor. 8.2.4). 

We say that a dg module $M$ over a dg algebra $A$ is \emph{homotopically projective} if given any acyclic dg module $N$ over $A$ (\textit{i.e.} $H^{\bullet}(N) = 0$) 
and every morphism of dg modules $f : M \rightarrow N$, there is $h \in \operatorname{hom}_{A}(M,N[-1])$ (called a \emph{homotopy} between $f$ and $0$) such that $d(h) = f$. 
As noticed by \cite{AFH}, any semi-free dg module is homotopically projective. 
This follows directly from the easy fact that any homotopy between $f|_{M_{i}}$ and the zero map can be extended to a homotopy between $f|_{M_{i+1}}$ and the corresponding zero map. 
Indeed, this can be easily proved by diagram chasing arguments applied to the following exact sequence of dg modules over $k$ provided with morphisms of complete degree zero
\[     \mathcal{H}om_{A}(M_{i+1}/M_{i},N) \rightarrow \mathcal{H}om_{A}(M_{i+1},N) \rightarrow \mathcal{H}om_{A}(M_{i},N) \rightarrow 0.     \]

A \emph{semi-free resolution} of $M$ is a dg module $F$ over $A$ together with a morphism of dg $A$-modules $f : F \rightarrow M$ of complete degree zero 
such that it is a quasi-isomorphism. 
As noted in \cite{AH}, Subsection 1.11, (6), a semi-free resolution always exists, and the morphism $f$ can be even choosen to be surjective (see \cite{AFH}, Thm. 8.3.2). 
The construction of the pair $(F,f)$ is given by the direct limit of a recursive contruction of pairs $(F_{i},f_{i})_{i \in \NN_{0}}$ satisfying that 
$\{ F_{i} \}_{i \in \NN_{0}}$ is an increasing sequence of dg $A$-modules with $F_{0} = 0$ and $F_{i+1}/F_{i}$ free, $f_{i} : F_{i} \rightarrow M$ is a morphism of dg $A$-modules 
of complete degree zero and $f_{i+1}|_{M_{i}} = f_{i}$ for all $i \in \NN_{0}$. 
The inductive step is given as follows. 
Suppose we have constructed $(F_{j},f_{j})_{j = 0, \dots, i}$ as before, for some $i \in \NN_{0}$, then one takes a free dg $A$-module $P$ together with a morphism $P \rightarrow \operatorname{cone}(f_{i})[-1]$ which induces a surjective morphism between cohomology groups (this can be easily done by taking $P$ the free dg $A$-module generated 
by a set of cocycles, whose cohomology class generate the cohomology of the cone $\operatorname{cone}(f_{i})[-1]$). 
Set 
\[     F_{i+1} = \operatorname{cone}((p_{1})^{[-1]}_{[-1]} \circ \pi),     \] 
where $p_{1} : \operatorname{cone}(f_{i}) \rightarrow F_{i}[1]$ is the morphism of dg modules given by the canonical projection, and 
define $f_{i+1} : F_{i+1} \rightarrow M$ as 
\[     f_{i+1}(p,e) = (p_{2} \circ \pi^{[1]}_{[1]}) (p) + f_{i}(e),      \]
where $(p,e) \in P[1] \oplus F_{i}$, $p_{2} : \operatorname{cone}(f_{i}) \rightarrow M$ is the 
morphism of graded modules given by the canonical projection (it is not a morphism of dg modules!). 
It is easy to check that $f_{i+1}$ is a morphism of dg $A$-modules, there is a canonical inclusion of dg $A$-modules $F_{i} \subseteq F_{i+1}$, $f_{i+1}|_{F_{i}} = f_{i}$ 
and $F_{i+1}/F_{i} \simeq P$ is a free dg $A$-module. 
It is clear that $F$ is semi-free and $f$ is surjective. 
Let us see that it is a quasi-isomorphism. 
It is easy to see that the inclusion $F_{i} \subseteq F_{i+1}$ of dg $A$-modules induce in turn an inclusion $\operatorname{cone}(f_{i}) \rightarrow \operatorname{cone}(f_{i+1})$ of dg $A$-modules, 
thanks to the property $f_{i+1}|_{F_{i}} = f_{i}$. 
We thus obtain an increasing Hausdorff and exhaustive filtration $\{ \operatorname{cone}(f_{i}) \}_{i \in \NN_{0}}$ of dg $A$-modules of $\operatorname{cone}(f)$. 
Since filtered colimits are exact (see \cite{W}, Thm. 2.6.15), they commute with taking cohomology, so the cohomology of $\operatorname{cone}(f)$ is the direct limit of the system given 
by $\{ H^{\bullet}(\operatorname{cone}(f_{i})) \}_{i \in \NN_{0}}$ together with the cohomology classes of the maps $\operatorname{cone}(f_{i}) \rightarrow \operatorname{cone}(f_{i+1})$, for $i \in \NN_{0}$. 
The latter morphisms $H^{\bullet}(\operatorname{cone}(f_{i})) \rightarrow H^{\bullet}(\operatorname{cone}(f_{i+1}))$ vanish by construction, which implies thus that $H^{\bullet}(\operatorname{cone}(f)) = 0$, which in turn implies that $f$ is a quasi-isomorphism. 

\section{Hochschild homology and cohomology of dg algebras} 
\label{sec:hoch}

\subsection{The bar resolution and Hochschild (co)homology of dg algebras} 
\label{subsec:barrhoch}

We recall that, for $A$ and $B$ two unitary dg algebras over $k$, the \emph{free product} $A *_{k} B$ of $A$ and $B$ is given as a unitary graded algebra over $k$ by 
\[     T_{k}(A \oplus B)/\cl{ 1_{A} - 1_{B}, a \otimes a' - a a', b \otimes b' - b b' : \text{ for all } a, a' \in A \text{ and } b, b' \in B},     \] 
where $T_{k}(V)$ is the tensor algebra on a graded module $V$ over $k$. 
Note that the canonical inclusions $i_{A} : A \rightarrow A *_{k} B$ and $i_{B} : B \rightarrow A *_{k} B$ are morphisms of graded $k$-algebras. 
Then $A *_{k} B$ has a natural structure of graded $A$-bimodule via $i_{A}$ and of graded $B$-bimodule via $i_{B}$. 
The differential $d_{A}$ of $A$ can be extended as the unique derivation $d_{A *_{k} B | A}$ of $A *_{k} B$ satisfying that $d_{A *_{k} B | A} \circ i_{A} = d_{A}$ 
and $d_{A *_{k} B | A} \circ i_{B} = 0$. 
The same applies to the differential $d_{B}$, providing a derivation $d_{A *_{k} B | B}$ on $A *_{k} B$. 
Note that $d_{A *_{k} B | A} d_{A *_{k} B | B} = - d_{A *_{k} B | B} d_{A *_{k} B | A}$. 
The differential $d_{A *_{k} B}$ of $A *_{k} B$ is just the derivation $d_{A *_{k} B | A} + d_{A *_{k} B | B}$. 
Hence, we see that $A *_{k} B$ has in fact a natural structure of dg $A$-bimodule via $i_{A}$ and of dg $B$-bimodule via $i_{B}$. 
By abuse of notation, we usually write $d_{A}$ instead of $d_{A *_{k} B | A}$ and $d_{B}$ instead of $d_{A *_{k} B | B}$. 
Note that $A *_{k} B$ is just the coproduct of $A$ and $B$ in the category of the dg algebras over $k$.

Let $k[\epsilon]$ be the differential graded algebra whose underlying $k$-module is the usual polynomial algebra on the indeterminate $\epsilon$, 
where the degree of $\epsilon$ is $-1$ and the weight is zero, provided with the differential of complete degree $(1,0_{G'})$ given by the derivation $\partial/\partial\epsilon$, 
\textit{i.e.} the unique derivation satisfying that $\partial/\partial\epsilon (\epsilon) = 1$. 
Consider the differential graded algebra given by the free product $A *_{k} k[\epsilon]$, and the differential induced by $d_{A}$ and $\partial/\partial\epsilon$. 
Following V. Drinfeld (\textit{cf.} \cite{Gin}, Subsection 4.3), the \emph{augmented (nonreduced or unnormalized) bar complex} of $A$ is just another ``presentation'' 
of the differential graded algebra $A *_{k} k[\epsilon]$ with the differential given by $d_{A} + \partial/\partial \epsilon$. 
We will explain what this means. 
Consider the graded $A$-bimodule given by $\operatorname{Bar}(A) = \oplus_{n \in \NN_{0}} (A \otimes A[1]^{\otimes n} \otimes A)$. 
If $n \in \NN$ we will typically denote an element $a_{0} \otimes s(a_{1}) \otimes \dots \otimes s(a_{n}) \otimes a_{n+1} \in A \otimes A[1]^{\otimes n} \otimes A$ in the form $a_{0} [a_{1} | \dots | a_{n}] a_{n+1}$, where $a_{0}, \dots, a_{n+1} \in A$ and $s : A \rightarrow A[1]$ is the canonical morphism of degree $-1$ recalled in the third paragraph of Subsection \ref{subsec:int}. 
In the same manner, we may usually denote $a_{0} \otimes a_{1}$ by $a_{0} [] a_{1}$. 
There is a canonical identification (as graded $A$-bimodules, so the morphism is of complete degree zero) of $\operatorname{Bar}(A)[1]$ inside $A *_{k} k[\epsilon]$ given by 
$s_{\operatorname{Bar}(A)}(a_{0} [a_{1} | \dots | a_{n}] a_{n+1}) \mapsto (-1)^{\deg a_{0} + \dots + \deg a_{n} + n} a_{0} \epsilon \dots \epsilon a_{n+1}$, for $n \geq 0$, where we have replaced each occurrence of the tensor on the left member by $\epsilon$ and added a sign. 
Under this identification $\operatorname{Bar}(A)$ gets a differential $b'$ of complete degree $(1,0_{G'})$, such that $\operatorname{Bar}(A)$ is a differential graded $A$-bimodule (if we forget about the map 
$\partial/\partial \epsilon$ applied to elements $a_{0}\epsilon a_{1}$). 
Moreover, under the previous identifications, and seeing $A$ inside $A *_{k} k[\epsilon]$ via $i_{A}$, the differential $d_{A *_{k} k[\epsilon]}$ 
of $A *_{k} k[\epsilon]$ induces the map of graded $A$-bimodules (of complete degree zero) from $\operatorname{Bar}(A)$ to $A$ whose restriction to $A \otimes A$ 
is given by the product of $A$, and the restriction to $A \otimes A[1]^{\otimes n} \otimes A$, for $n \neq 0$, vanishes. 
In fact, it is clear that using the previous maps $A *_{k} k[\epsilon]$ is identified (as a graded $A$-bimodule) with the cone of the morphism of dg $A$-bimodules 
$\operatorname{Bar}(A) \rightarrow A$ of complete degree $0_{G}$. 
That this map $\operatorname{Bar}(A) \rightarrow A$ is a quasi-isomorphism is tantamount to the fact that its cone is acyclic, or, under the previous identification, 
that $A *_{k} k[\epsilon]$ is acyclic. 
This last statement follows easily from the fact that the cohomology $H^{\bullet}(A *_{k} k[\epsilon])$ is a unitary algebra whose unit vanishes, since $1_{A *_{k} k[\epsilon]} = d_{A *_{k} k[\epsilon]}(\epsilon)$. 

The \emph{augmented reduced (or normalized) bar complex} is just what becomes identified when we consider $(A *_{k} k[\epsilon])/\cl{\epsilon^{2}}$ instead of $A *_{k} k[\epsilon]$. 
In this case, the underlying graded $A$-bimodule will be given by $\overline{\operatorname{Bar}}(A) = \oplus_{n \in \NN_{0}} A \otimes \bar{A}[1]^{\otimes n} \otimes A$, where $\bar{A} = A/k.1_{A}$ 
(as graded $k$-modules). 
We will still denote an element $a_{0} \otimes s(\bar{a}_{1}) \otimes \dots \otimes s(\bar{a}_{n}) \otimes a_{n+1} \in A \otimes \bar{A}[1]^{\otimes n} \otimes A$ in the form 
$a_{0} [a_{1} | \dots | a_{n}] a_{n+1}$, where $a_{0}, \dots, a_{n+1} \in A$, thus omitting the bars for simplicity. 
This will mean in particular that an element $a_{0} [a_{1} | \dots | a_{n}] a_{n+1}$ in the reduced bar complex of $A$ vanishes if there is some index $i \in \{ 1, \dots, n \}$ such that $a_{i}$ 
is scalar multiple of $1_{A}$. 
The previous isomorphisms now induce an identification of $\overline{\operatorname{Bar}}(A)[1]$ inside $(A *_{k} k[\epsilon])/\cl{\epsilon^{2}}$, and by the same arguments it becomes a dg 
$A$-bimodule with differential $\bar{b}'$, which is a resolution of $A$, also by the map whose restriction to $A \otimes A$ is given by the product of $A$, and whose restriction to $A \otimes \bar{A}[1]^{\otimes n} \otimes A$ vanishes for $n \neq 0$. 
As for the case of the bar complex, the dg $A$-bimodule $(A *_{k} k[\epsilon])/\cl{\epsilon^{2}}$ is identified by the previous map with the cone of the morphism of dg $A$-bimodules $\overline{\operatorname{Bar}}(A) \rightarrow A$ of complete degree $0_{G}$. 
This last map is proved to be a quasi-isomorphism using the same arguments as in the previous paragraph. 
This in turn implies that the morphism of differential graded $A$-bimodules $\operatorname{Bar}(A) \rightarrow \overline{\operatorname{Bar}}(A)$ given by taking quotients is a quasi-isomorphism, for the latter is induced (using the previous identifications) by the canonical quotient morphism $A *_{k} k[\epsilon] \rightarrow (A *_{k} k[\epsilon])/\cl{\epsilon^{2}}$ of differential graded algebras. 
We remark that the previously defined differential $\bar{b}'$ coincides with the differential $d_{0} + d_{1}$ defined in \cite{FTVP}, Subsection 2.2, of the differential graded $A$-bimodule 
$\bar{B}(A;A;A)$ considered there (which, as a graded $A$-bimodule, coincides with $\overline{\operatorname{Bar}}(A)$). 
More explicitly, the previously referred maps are given by
\begin{equation*}
\begin{split}
     d_{0}(a_{0} [a_{1} | \dots | a_{n}] a_{n+1}) =& d_{A}(a_{0}) [a_{1} | \dots | a_{n}] a_{n+1} 
     \\ 
     &- \sum_{i=1}^{n} (-1)^{\epsilon_{i}} a_{0}[a_{1} | \dots | d_{A}(a_{i}) | \dots | a_{n}] a_{n+1} 
     \\  
     &+ (-1)^{\epsilon_{n+1}} a_{0} [a_{1} | \dots | a_{n}] d_{A}(a_{n+1}),     
\end{split}
\end{equation*}
and 
\begin{equation*}
\begin{split}
     d_{1}(a_{0} [a_{1} | \dots | a_{n}] a_{n+1}) =& (-1)^{\deg a_{0}} a_{0} a_{1} [a_{2} | \dots | a_{n}] a_{n+1} 
     \\ 
     &+ \sum_{i=2}^{n} (-1)^{\epsilon_{i}} a_{0}[a_{1} | \dots | a_{i-1} a_{i} | \dots | a_{n}] a_{n+1} 
     \\ 
     &- (-1)^{\epsilon_{n}} a_{0} [a_{1} | \dots | a_{n-1}] a_{n} a_{n+1},     
\end{split}
\end{equation*}
where $\epsilon_{i} = \deg a_{0} + (\sum_{j=1}^{i-1} \deg a_{j}) - i +1$, and where it is assumed that the expression $a_{0} [a_{1} | \dots | a_{n}] a_{n+1}$ vanishes if $a_{i} = \lambda 1_{A}$, 
for some $\lambda \in k$ and $i \in \{ 1, \dots, n \}$. 
The same expression of the differential hold for the nonreduced bar complex. 

The following result justifies the relevance of the bar resolution. 
It was proved for augmented dg algebras and the reduced bar resolution in \cite{FHT}, Lemma 4.3,
though exactly the same proof applies verbatim to this more general case. 
We will provide it for completeness.
\begin{lemma}
Let $A$ be a unitary differential graded algebra such that the underlying dg $k$-module of $A$ is semi-free. 
Then the previously considered morphism of differential graded $A$-bimodules $\operatorname{Bar}(A) \rightarrow A$ (or $\overline{\operatorname{Bar}}(A) \rightarrow A$) is in fact a semi-free resolution of $A$.  
\end{lemma} 
\begin{proof}
The fact that $\operatorname{Bar}(A) \rightarrow A$ (or $\overline{\operatorname{Bar}}(A) \rightarrow A$) is a quasi-isomorphism of dg $A$-bimodules was already shown at the end of the second 
and the beginning of the third paragraphs of this section. 
It remains to prove that $\operatorname{Bar}(A)$ (or $\overline{\operatorname{Bar}}(A)$) is a semi-free dg bimodule over $A$. 
Let us prove it for the nonreduced bar resolution, the case of the reduced one being analogous. 
Since $A$ is a semi-free dg $k$-module, the same applies to the dg $k$-module $A[1]$, and to the tensor products $A[1]^{\otimes n}$. 
This in turn implies that the dg $A$-bimodule $A \otimes A[1]^{\otimes n} \otimes A$ (provided only with the differential induced by $d_{A}$, \textit{i.e.} $d_{0}$ given before) is semi-free. 
The proof ends by using the last property of semi-free modules given in the antepenultimate paragraph of the previous subsection by noting that the previous dg $A$-bimodule $A \otimes A[1]^{\otimes n} \otimes A$ is isomorphic to the quotient $T_{n}/T_{n-1}$, for the increasing Hausdorff and exhaustive filtration $\{ T_{n} \}_{n \in \NN_{0}}$ of $\operatorname{Bar}(A)$ 
given by the dg $A$-bimodules whose underlying graded modules are $T_{n} = \oplus_{i=0}^{n} A \otimes A[1]^{\otimes i} \otimes A$ for all $n \in \NN_{0}$.  
\end{proof}

Let now $M$ be a dg bimodule over $A$. 
The \emph{Hochschild homology} $H_{\bullet}(A,M)$ of $A$ with coefficients on $M$ is just the homology of the complex $M \otimes_{A^{e}} \operatorname{Bar}(A)$, or equivalently, $M \otimes_{A^{e}} \overline{\operatorname{Bar}}(A)$, with differential 
$d_{M} \otimes_{A^{e}} \mathrm{id}_{B_{\bullet}(A)} +  \mathrm{id}_{M} \otimes_{A^{e}} b'$, or $d_{M} \otimes_{A^{e}} \mathrm{id}_{\bar{B}_{\bullet}(A)} + \mathrm{id}_{M} \otimes_{A^{e}} \bar{b}'$, respectively. 
We recall the canonical identification $\Phi_{A,M} : M \otimes_{A^{e}} \overline{\operatorname{Bar}}(A) \rightarrow M \otimes T(s(\bar{A}))$ of the form $m \otimes_{A^{e}} a_{0} [a_{1} | \dots | a_{n}] a_{n+1} \mapsto (-1)^{\deg a_{n+1} (\deg m + (\sum_{i=0}^{n} \deg a_{i}) + n)} a_{n+1} m a_{0} \otimes [a_{1} | \dots | a_{n}]$. 
By means of the former isomorphism we induce a differential of the form $D_{0}' + D_{1}'$ on $M \otimes T(s(\bar{A}))$ given by 
\begin{equation}
\label{eq:dif0homo}
\begin{split}
     D_{0}'(m \otimes [a_{1} | \dots | a_{n}]) = & d_{M}(m) \otimes [a_{1} | \dots | a_{n}]  
     \\ 
     &- \sum_{i=1}^{n} (-1)^{\tilde{\epsilon}_{i}} m \otimes [a_{1} | \dots | d_{A}(a_{i}) | \dots | a_{n}]
     \\ 
     &+ \sum_{i=2}^{n} (-1)^{\tilde{\epsilon}_{i}} m \otimes [a_{1} | \dots | a_{i-1} a_{i} | \dots | a_{n}],       
\end{split}
\end{equation}
and 
\begin{equation}
\label{eq:dif1homo}
\begin{split}
     D_{1}'(m \otimes [a_{1} | \dots | a_{n}]) =& (-1)^{\deg m} m a_{1} \otimes [a_{2} | \dots | a_{n}]  
     \\ 
     &- (-1)^{\tilde{\epsilon}_{n} (\deg a_{n} + 1)} a_{n} m \otimes [a_{1} | \dots | a_{n-1}], 
\end{split}
\end{equation}
where $\tilde{\epsilon}_{i} = \deg m + (\sum_{j=1}^{i-1} \deg a_{j}) - i +1$, and as usual the expression $[a_{1} | \dots | a_{n}]$ vanishes if $a_{i} = \lambda 1_{A}$, 
for some $\lambda \in k$ and $i \in \{ 1, \dots, n \}$. 
The same expression of the differential hold for the nonreduced bar complex. 
Note that our expression of differential coincides with the corresponding one of \cite{T}, Subsection 2.1, if one understands in their equation before (10), and following their notation, 
that the sum is indexed over $i = 0, \dots, p$, and $\sum_{k < i} (|a_{k}| + 1) + 1$ (\textit{i.e.} $\sum_{k=0}^{i-1} (|a_{k}| + 1) + 1$) is in fact $|a_{0}| + (|a_{1} + 1|) + \dots + (|a_{i-1}| + 1)$, 
so it vanishes if $i = 0$, by the principle of summing over the empty set (here $| \phantom{x} |$ is our cohomological degree). 
Note that the expressions written before to interpret $\sum_{k < i} (|a_{k}| + 1) + 1$ coincide for $i \geq 1$. 
Accordingly, if we regard the convention of \cite{NT}, Section 2, (2.2), and also following their notation, $\eta_{j}$ should be understood as $((|a_{0}|-1) + |a_{1}| + \dots + |a_{i-1}|) + 1$ 
(following the interpretation consistent with their identity (2.1) and not with their definition before (1.1)) and not as $\sum_{k=0}^{j-1} (|a_{k}|)$ (here $| \phantom{x} |$ is our cohomological degree plus one). 
The difference between the two expressions is only apparent when $j = 0$, for the latter gives $0$, being a sum indexed over an empty set. 
With this interpretation, our differential would just be $b - \delta$ instead of $b + \delta$ in the notation of that article. 

Analogously, the \emph{Hochschild cohomology} $H^{\bullet}(A,M)$ of $A$ with coefficients on $M$ is given by the cohomology of the complex $\mathcal{H}om_{A^{e}} (\operatorname{Bar}(A),M)$, or equivalently, 
$\mathcal{H}om_{A^{e}} (\overline{\operatorname{Bar}}(A),M)$, with differential $f \mapsto d_{M} \circ f - (-1)^{deg f} f \circ b'$, or $f \mapsto d_{M} \circ f - (-1)^{deg f} f \circ \bar{b}'$, respectively. 
More explicitly, using the canonical identification $\Phi^{A,M} : \mathcal{H}om_{A^{e}} (\overline{\operatorname{Bar}}(A),M) \rightarrow \mathcal{H}om (T(s(\bar{A})),M)$ given by  
$\Phi^{A,M}(f)([a_{1} | \dots | a_{n}]) = f(1_{A} [a_{1} | \dots | a_{n}] 1_{A})$, we get that the latter complex induces a differential given by a sum $D_{0} + D_{1}$, whose value at $f \in \mathcal{H}om (T(s(\bar{A})),M)$ is  
\begin{equation}
\label{eq:dif0cohomo}
\begin{split}
     D_{0}(f)([a_{1} | \dots | a_{n}]) =& d_{M}(f([a_{1} | \dots | a_{n}])) 
     \\ 
     &+ \sum_{i=1}^{n} (-1)^{\bar{\epsilon}_{i}} f([a_{1} | \dots | d_{A}(a_{i}) | \dots | a_{n}]) 
     \\  
     &- \sum_{i=2}^{n} (-1)^{\bar{\epsilon}_{i}} f([a_{1} | \dots | a_{i-1} a_{i} | \dots | a_{n}]),     
\end{split}
\end{equation}
and 
\begin{equation}
\label{eq:dif1cohomo}
\begin{split}
     D_{1}(f)([a_{1} | \dots | a_{n}]) =& - (-1)^{\deg a_{1} \deg f - \deg f} a_{1} f([a_{2} | \dots | a_{n}])  
     \\ 
     &+ (-1)^{\bar{\epsilon}_{n}} f([a_{1} | \dots | a_{n-1}]) a_{n}, 
\end{split}
\end{equation}
where $\bar{\epsilon}_{i} = \deg f + (\sum_{j=1}^{i-1} \deg a_{j}) - i +1$, and as before it is supposed that the expression $[a_{1} | \dots | a_{n}]$ vanishes if $a_{i} = \lambda 1_{A}$, 
for some $\lambda \in k$ and $i \in \{ 1, \dots, n \}$. 
The same expression of the differential holds for the nonreduced bar complex. 
Notice that our expression of differential does coincide with the corresponding one of \cite{NT}, Section 1, Definition 1.1, or \cite{T}, Subsection 2.2, p. 80. 

\subsection{The bar construction and Hochschild (co)homology of augmented dg algebras} 
\label{subsec:barchoch}

We are interested in the Hochschild (co)homology $H^{\bullet}(A,A)$ of $A$ with coefficients on the same dg algebra $A$ (also denoted by $HH^{\bullet}(A)$) 
under the slightly stronger assumption of $A$ being an augmented dg algebra. 
We recall that a unitary dg algebra $A$ is called \emph{augmented} if there is a morphism $\epsilon_{A} : A \rightarrow k$ of unitary dg algebras. 
In this case, $I_{A} = \operatorname{Ker}(\epsilon_{A})$ is called the \emph{augmentation ideal} of $A$. 
A morphism of augmented dg algebras $f : A \rightarrow A'$ is a morphism of unitary dg algebras such that $\epsilon_{A'} \circ f = \epsilon_{A}$. 
Analogously, a counitary dg coalgebra $C$ is said to be \emph{coaugmented} if there exists a morphism of counitary dg coalgebras $\eta_{C} : k \rightarrow C$. 
We will usually denote the coaugmentation cokernel $C/\operatorname{Im}(\eta_{C})$ of $C$ by $J_{C}$, which is a (nonunitary) dg coalgebra. 
As in the case of augmented dg algebras, $J_{C}$ is canonically identified with $\operatorname{Ker}(\epsilon_{C})$, and under that identification the coproduct of $\operatorname{Ker}(\epsilon_{C})$ 
is given by $\Delta_{C}(c) - 1_{C} \otimes c - c \otimes 1_{C}$, for $c \in \operatorname{Ker}(\epsilon_{C})$, where $1_{C} = \eta_{C}(1_{k})$. 
We shall denote such an element by $\Delta_{\operatorname{Ker}(\epsilon_{C})}(c)$, or $c_{(1)}^{-} \otimes c_{(2)}^{-}$.   
A morphism of coaugmented dg coalgebras $f : C' \rightarrow C$ is a morphism of counitary dg coalgebras such that $f \circ \eta_{C'} = \eta_{C}$. 

\subsubsection{Twists of (augmented) dg algebras}
\label{subsubsec:perturb}

One natural question is whether we may ``perturb'' the differential $d_{A}$ of an augmented dg algebra $A$ such that the resulting dg $k$-module is still an augmented dg algebra for the previous product, unit and augmentation. 
We will be concerned with the much simpler issue of finding a homogeneous element 
$a \in A^{(1,0_{G'})}$ such that $d_{A,a} = d_{A} + \mathrm{ad}(a)$, where $\mathrm{ad}(a) : A \rightarrow A$ is the morphism of graded $k$-modules of degree $(1,0_{G'})$ given by $a' \mapsto [a,a'] = a a' - (-1)^{\deg a'} a' a$, is a differential, \textit{i.e.} $d_{A,a}^{2} = 0$, and $A$ is an augmented dg algebra for the same product, unit and augmentation as before but new differential $d_{A,a}$. 
It is clear that the map $d_{A,a}$ is a differential if and only if $d_{A}(a) + a a = d_{A}(a) + [a, a]/2$ lies in the (graded) center of $A$ (the latter identity having meaning if the characteristic of $k$ is different from $2$), and in particular if $d_{A}(a) + a a = 0$, which is called the \emph{Maurer-Cartan equation} for $a$ (this can be done in fact for any dg Lie algebra if the characteristic is different from $2$). 
It is trivially verified that $d_{A,a}$ is always a derivation, that $d_{A,a}(1_{A})=0$ and $\epsilon_{A} \circ d_{A,a} = 0$, so $A$ is an augmented algebra for $d_{A,a}$ if and only if the latter is a differential. 
This is always the case if $a$ satisfies the Maurer-Cartan equation. 
The procedure of obtaining $d_{A,a}$ from $d_{A}$ and an element $a$ satisfying the Maurer-Cartan equation is usually called a \emph{twist} of the differential, and the new augmented dg algebra is called 
the \emph{twisted augmented dg algebra of $(A,d_{A})$ by $a$}. 
Moreover, we see that this latter twisting construction is natural, in the sense that if $f : A \rightarrow A'$ is a morphism of augmented dg algebras and $a \in A^{(1,0_{G'})}$ satisfies the Maurer-Cartan equation, then $f(a) \in (A')^{(1,0_{G'})}$ also satisfies the Maurer-Cartan equation and $f$ can be regarded as a morphism of augmented dg algebras for $A$ and $A'$ provided 
with the new differentials $d_{A,a}$ and $d_{A',f(a)}$, respectively. 

Suppose further that we are given a dg $A$-bimodule $M$ of an augmented dg algebra $A$. 
Let us consider an element $a \in A^{(1,0_{G'})}$ solution to the Maurer-Cartan equation, 
so we may regard the new augmented dg algebra structure on $A$ given by only changing the differential by $d_{A,a}$. 
One may wonder how to twist the differential $d_{M}$ of $M$ in order to still define a dg bimodule with the same action map $A \otimes M \otimes A \rightarrow M$ over the new augmented dg algebra $A$ provided with the differential $d_{A,a}$. 
It is clear that the only condition one needs to verify for the new differential on $M$ is the Leibniz identity, for the others are automatic. 
Furthermore, it is easily verified that the new differential $d_{M,a}$ given by $d_{M} + \mathrm{ad}(a)$, where $\mathrm{ad}(a) : M \rightarrow M$ is the morphism of graded $k$-modules of degree $(1,0_{G'})$ given by $m \mapsto a m - (-1)^{\deg m} m a$, satisfies the Leibniz identity and thus defines on $M$ the structure of a dg bimodule over the augmented dg algebra $A$ provided with the differential $d_{A,a}$. 
If $M$ and $N$ are two dg bimodules over the augmented dg algebra $A$ with differential $d_{A}$, and $g : M \rightarrow N$ is a morphism of dg bimodules, then it is also a morphism of dg bimodules over the augmented dg algebra $A$ provided with the differential $d_{A,a}$, when we regard $M$ and $N$ with differentials $d_{M,a}$ and $d_{N,a}$, respectively. 

\subsubsection{The convolution algebra and the tensor product module}
\label{subsubsec:convo}

The following constructions are well-known (see \cite{Prou}, Lemme 1.35). 
Let $C$ be a coaugmented dg coalgebra a $A$ an augmented dg algebra. 
Consider the dg $k$-module given by $\mathcal{H}om (C,A)$. 
It is in fact an augmented dg algebra with product given by 
\[     \phi * \psi = \mu_{A} \circ (\phi \otimes \psi) \circ \Delta_{C},     \]
unit $\eta_{A} \circ \epsilon_{C}$ and augmentation $\phi \mapsto (\epsilon_{A} \circ \phi \circ \eta_{C})(1_{k})$. 
We remark that, using the Sweedler notation, the coproduct can be written as 
\[     (\phi * \psi) (c) = (-1)^{\deg \psi \deg c_{(1)}} \phi(c_{(1)}) \psi(c_{(2)}),     \]
for $c \in C$. 
Note that the previous construction is natural, \textit{i.e.} if $f' : C' \rightarrow C$ is a morphism of coaugmented dg coalgebras and $f : A \rightarrow A'$ is a morphism of augmented dg algebras, then the morphism of dg $k$-modules $\mathcal{H}om(f',f) : \mathcal{H}om (C,A) \rightarrow \mathcal{H}om (C',A')$ given by $\phi \mapsto f \circ \phi \circ f'$ 
is in fact a morphism of augmented dg algebras. 
If $k$ is semisimple, it is clear that if $f'$ and $f$ are quasi-isomorphisms, then $\mathcal{H}om(f',f)$ is so (see \cite{Bour}, \S 5.2. Cor. 1, and \S 2.5, Ex. 4).

Moreover, given $M$ any dg bimodule over $A$ we see that the dg $k$-module given by the tensor product $M \otimes C$ has a structure of a dg bimodule over $\mathcal{H}om (C,A)$. 
The action is given by 
\[     \phi \cdot (m \otimes c) \cdot \psi = (-1)^{\epsilon} \phi(c_{(3)}).m.\psi(c_{(1)}) \otimes c_{(2)},     \]
where $m \in M$, $c \in C$, $\phi, \psi \in \mathcal{H}om (C,A)$, and 
\[     \epsilon = \deg \psi \deg c + \deg c_{(3)} (\deg m + \deg c_{(1)} + \deg c_{(2)} + \deg \psi).     \] 
If $f' : C' \rightarrow C$ is a morphism of coaugmented dg coalgebras and $g : M' \rightarrow M$ is a morphism of dg bimodules over $A$, then $g \otimes f'$ is a morphism of dg bimodules over $\mathcal{H}om(C,A)$, where $M' \otimes C'$ has the structure of dg bimodule over $\mathcal{H}om(C,A)$ given by the restriction of scalar through $\mathcal{H}om(f',\mathrm{id}_{A}) : \mathcal{H}om(C,A) \rightarrow \mathcal{H}om(C',A)$. 
Provided $k$ is Von Neumann regular (in particular, this holds if $k$ is semisimple), if $f'$ and $g$ are quasi-isomorphisms, then $g \otimes f'$ is also (see \cite{Bour}, \S 4.7. Th\'eo. 3). 

Suppose that $M$ has in fact two graded-commuting dg $A$-bimodule structures, \textit{i.e.} $M$ is a dg $A^{e}$-bimodule (\textit{e.g.} $M = A^{e}$).  
In this case one may use one of the dg bimodule structures over $A$ on $M$ to induce the dg bimodule structure over $\mathcal{H}om (C,A)$ on $M \otimes C$, whereas the second dg bimodule structure over $A$ on $M$ gives in fact a dg $A$-bimodule structure on $M \otimes C$ by the formula
\[     a (m \otimes c) a' = (-1)^{\deg a' \deg c} (a m a') \otimes c,     \]
where we remark that we are using the second dg $A$-bimodule structure on $M$. 
Moreover, both structures are compatible, \textit{i.e.} $M \otimes C$ has in fact a dg bimodule over $\mathcal{H}om (C,A) \otimes A$. 

\subsubsection{The twisted convolution algebra and the twisted tensor product}
\label{subsubsec:twistedconvo}

A solution $\tau \in \mathcal{H}om (C,A)$ to the Maurer-Cartan equation on the augmented dg algebra $\mathcal{H}om (C,A)$ which also satisfies that $\epsilon_{C} \circ \tau = 0$ and $\tau \circ \eta_{A} =0$ is called a \emph{twisting cochain} (some authors call these \emph{admissible} twisting cochains, 
because for them the term twisting cochain is any solution of the Maurer-Cartan equation of $\mathcal{H}om (C,A)$. See for \textit{e.g.} \cite{LH}, D\'ef. 2.2.1.1). 
We will denote the augmented dg algebra $\mathcal{H}om (C,A)$ with the twisted differential $d_{\mathcal{H}om (C,A),\tau}$ by $\mathcal{H}om^{\tau} (C,A)$. 
By the last paragraph of the two previous subsubsections, it is easy to see that the twist construction is natural, \textit{i.e.} given $\tau \in \mathcal{H}om (C,A)$ a twisting cochain, 
a morphism of coaugmented dg coalgebras $f' : C' \rightarrow C$ and a morphism of augmented dg algebras $f : A \rightarrow A'$, 
$\mathcal{H}om(f',f)(\tau) \in \mathcal{H}om (C',A')$ is a twisting cochain and $\mathcal{H}om(f',f)$ induces a morphism of augmented dg algebras from $\mathcal{H}om^{\tau} (C,A)$ to 
$\mathcal{H}om^{f \circ \tau \circ f'} (C',A')$. 
Even for $k$ semisimple and $f$ and $f'$ quasi-isomorphisms, the latter map need not be a quasi-isomorphism. 
A typical example would be as follows. 
Consider the quasi-isomorphism $\mathcal{H}om(B^{+}(A_{+}),A_{+}) \rightarrow \mathcal{H}om(k,A_{+}) \simeq A_{+}$ of augmented dg algebras given by composition with the canonical injection 
$k \subseteq B^{+}(A_{+})$, where $A$ is a unitary algebra and $A_{+}$ is the augmented algebra recalled in the fifth paragraph of Subsubsection \ref{subsubsec:cobar} 
(the fact that it is a quasi-isomorphism follows from the comments in that paragraph). 
It is clear that the image of the twisting cochain $\tau_{A_{+}}$ under the previous mapping is zero, so we get a morphism $\mathcal{H}om^{\tau_{A_{+}}}(B^{+}(A_{+}),A_{+}) \rightarrow A_{+}$ 
of augmented dg algebras. 
Taking cohomology we obtain the morphism $HH^{\bullet}(A_{+}) \rightarrow A_{+}$ given by the composition of the canonical projection $HH^{\bullet}(A_{+}) \rightarrow HH^{0}(A_{+}) \simeq \mathcal{Z}(A_{+})$ 
together with the inclusion of the center $\mathcal{Z}(A_{+})$ of $A_{+}$ inside $A_{+}$. 
If $A$ is noncommutative we see that the previous map in cohomology is not an isomorphism. 

Given $M$ any dg bimodule over $A$, and a twisting cochain $\tau$ in $\mathcal{H}om(C,A)$, we see that we may twist the differential $d_{M \otimes C}$ of the 
tensor product $M \otimes C$, which is a dg bimodule over the augmented dg algebra $\mathcal{H}om(C,A)$, in order to obtain the dg bimodule provided with the differential 
$d_{M \otimes C,\tau}$ over the augmented dg algebra $\mathcal{H}om^{\tau}(C,A)$. 
We shall denote this new dg bimodule by $M \otimes_{\tau} C$. 
If $f'' : C' \rightarrow C$ is a morphism of coaugmented dg coalgebras and $g : M' \rightarrow M'$ is a morphism of dg bimodules over $A$, then $g \otimes f'$ is a morphism of dg bimodules over $\mathcal{H}om^{\tau} (C,A)$ from $M' \otimes_{\tau \circ f'} C'$ to $M \otimes_{\tau} C$, where we regard $M' \otimes_{\tau \circ f'} C'$ as a dg bimodule over $\mathcal{H}om^{\tau} (C,A)$ by means of the morphism of augmented dg algebras $\mathcal{H}om(f',\mathrm{id}_{A}) : \mathcal{H}om^{\tau}(C,A) \rightarrow \mathcal{H}om^{\tau \circ f'}(C',A)$. 
If $M$ has two graded-commuting dg $A$-bimodule structures, so $M \otimes C$ is a dg bimodule over $\mathcal{H}om (C,A) \otimes A$, 
one notices that the previous twisting construction implies that $M \otimes_{\tau} C$ has in fact a dg bimodule structure over $\mathcal{H}om^{\tau} (C,A) \otimes A$. 

\subsubsection{The bar constructions}
\label{subsubsec:bar}

We point out the well-known fact that the dg $k$-module given by $k \otimes_{A^{e}} \overline{\operatorname{Bar}}(A)$ is further a coaugmented dg coalgebra, called the \emph{(reduced or normalized) bar construction} of $A$, and it is denoted by $B^{+}(A)$ (see \cite{FHT01}, Section 19, but also \cite{LH}, Notation 2.2.1.4, which we follow, though we do not use the same sign conventions).
Note that in this case the reduced bar resolution $\overline{\operatorname{Bar}}(A)$ can be equivalently presented as the graded $A$-bimodule $\oplus_{i \in \NN_{0}} A \otimes I_{A}[1]^{\otimes i} \otimes A$ provided with a differential given by the same expression, 
the isomorphism being induced by the obvious map $I_{A} \rightarrow A/k$. 
Moreover, using this identification we may construct an explicit quasi-inverse to the canonical quasi-isomorphism $\operatorname{Bar}(A) \rightarrow \overline{\operatorname{Bar}}(A)$ 
explained in the third paragraph of Subsection \ref{subsec:barrhoch}. 
Indeed, it is easy to check that the map 
\[     \oplus_{i \in \NN_{0}} A \otimes I_{A}[1]^{\otimes i} \otimes A \rightarrow \oplus_{i \in \NN_{0}} A \otimes A[1]^{\otimes i} \otimes A     \]
induced by the inclusion $I_{A}[1] \subseteq A[1]$ is such a quasi-isomorphism. 
Now, using the obvious isomorphism 
\[     k \otimes_{A^{e}} \overline{\operatorname{Bar}}(A) \simeq \bigoplus_{i \in \NN_{0}} I_{A}[1]^{\otimes i},     \]
induced by the identification explained previously, the coproduct is given by the usual deconcatenation 
\[     \Delta ([a_{1}|\dots |a_{n}]) = \sum_{i=0}^{n} [a_{1}|\dots | a_{i}] \otimes [a_{i+1} | \dots | a_{n}],     \]
where we set $[a_{i} | \dots | a_{j}] = 1_{B^{+}(A)}$ if $i > j$, for $1_{B^{+}(A)}$ the image of $1_{k}$ under the canonical inclusion $k = A[1]^{\otimes 0} \subseteq B^{+}(A)$. 
The differential, denoted by $B$, is trivially seen to be of the form 
\begin{equation*}
\begin{split}
     B([a_{1} | \dots | a_{n}]) =&- \sum_{i=1}^{n} (-1)^{\epsilon_{i}} [a_{1} | \dots | d_{A}(a_{i}) | \dots | a_{n}]
     \\
     &+ \sum_{i=2}^{n} (-1)^{\epsilon_{i}} [a_{1} | \dots | a_{i-1} a_{i} | \dots | a_{n}], 
\end{split}
\end{equation*}
where $\epsilon_{i} = (\sum_{j=1}^{i-1} \deg a_{j}) - i +1$. 
One checks easily that it is a coderivation. 
The counit is given by the canonical projection $B^{+}(A) \rightarrow I_{A}[1]^{\otimes 0} = k$, and the coaugmentation is defined as the obvious inclusion $k = I_{A}[1]^{\otimes 0} \subseteq B^{+}(A)$. 
Since $B^{+}(A)$ is a coaugmented tensor graded coalgebra, it is \emph{cocomplete}. 
The image of its differential $B$ lies inside the augmentation kernel $\operatorname{Ker}(\epsilon_{B^{+}(A)})$ 
of $B^{+}(A)$, so $B$ is thus uniquely determined by $\pi_{1} \circ B$, where $\pi_{1} : B^{+}(A) \rightarrow I_{A}[1]$ is the canonical projection 
(see \cite{LH}, Lemme 1.1.2.2, Sections 2.1.1 and 2.1.2, and Notation 2.2.1.4). 
We recall that a coaugmented graded coalgebra $C$ is cocomplete (or sometimes denoted as \emph{conilpotent}) if the (quotient) nonunitary graded coalgebra $J_{C} = C/\operatorname{Im}(\eta_{C})$ satisfies that its \emph{primitive filtration} given by 
\[     \operatorname{Ker}(\Delta_{J_{C}}) \subseteq \operatorname{Ker}(\Delta_{J_{C}}^{(3)}) \subseteq \dots \subseteq \operatorname{Ker}(\Delta_{J_{C}}^{(n)}) \subseteq \dots     \]
is exhausting, \textit{i.e.} its union is $J_{C}$. 
This composition map $\pi_{1} \circ B$ is written as the sum of two terms $b_{1} : I_{A}[1] \rightarrow I_{A}[1]$ and $b_{2} : I_{A}[1]^{\otimes 2} \rightarrow I_{A}[1]$ given by $sa \mapsto -sda$ and $sa \otimes sb \mapsto (-1)^{\deg a + 1} s(a b)$, for $a, b \in A$ homogeneous, resp. 
We note thus that $b_{1} = - s_{I_{A}} \circ d_{A} \circ s_{I_{A}}^{-1}$ and $b_{2} = - s_{I_{A}} \circ \mu_{A} \circ (s_{I_{A}}^{\otimes 2})^{-1}$. 
We warn the reader that even though we have set up our sign conventions for the differential of the bar construction in order to agree with several in the literature (\textit{e.g.} \cite{FHT01}, Section 19, or \cite{LPWZ04}, Section 8), and in particular they coincide with the ``universally'' accepted conventions in case our dg algebra is a plain algebra, 
they differ from others (\textit{e.g.} those in the thesis \cite{LH} of K. Lef\`evre-Hasegawa, Ch. 1 and 2). 
By very definition, the cohomology of $B^{+}(A)$ is the Tor group $\operatorname{Tor}_{\bullet}^{A}(k,k)$, where we recall that we should switch 
from cohomological (upper) grading to homological (lower) grading by the obvious relation $V^{n} = V_{-n}$, for $n \in \ZZ$, and where the Adams degree does not change. 

We remark that there is a \emph{nonreduced (or unnormalized) bar construction} $\tilde{B}^{+}(A)$ of $A$ given as a dg module over $k$ by $k \otimes_{A^{e}} \operatorname{Bar}(A)$, 
which in turn is canonically isomorphic to $\oplus_{i \in \NN_{0}} A[1]^{\otimes i}$ as graded $k$-modules. 
It is also a coaugmented dg coalgebra, and the formulas for the coproduct, the counit and the coaugmentation are the same as for the reduced bar construction. 
Its underlying coaugmented graded coalgebra structure is thus the one of a coaugmented tensor graded coalgebra. 
The explicit form of the differential, which we denote by $\tilde{B}$, is however different from the reduced case, namely, $\tilde{B}([]) = 0$ and for $n \in \NN$ we have that 
\begin{equation*}
\begin{split}
     \tilde{B}([a_{1} | \dots | a_{n}]) =&- \sum_{i=1}^{n} (-1)^{\epsilon_{i}} [a_{1} | \dots | d_{A}(a_{i}) | \dots | a_{n}]
     \\
     &+ \sum_{i=2}^{n} (-1)^{\epsilon_{i}} [a_{1} | \dots | a_{i-1} a_{i} | \dots | a_{n}] 
     \\
      &+ \epsilon_{A}(a_{1}) [a_{2} | \dots | a_{n}] - (-1)^{\epsilon_{n}} \epsilon(a_{n}) [a_{1} | \dots | a_{n-1}], 
\end{split}
\end{equation*}
where $\epsilon_{i} = (\sum_{j=1}^{i-1} \deg a_{j}) - i +1$, and $a_{1}, \dots, a_{n} \in A$. 
A rather long computation shows it is in fact a coderivation. 
It is not difficult to verify that the image of $\tilde{B}$ lies inside the augmentation kernel of $\tilde{B}^{+}(A)$, so it is uniquely determined by its composition with the canonical projection 
$p_{1} : \tilde{B}^{+}(A) \rightarrow A[1]$, which is just the sum of two terms $\tilde{b}_{1} : A[1] \rightarrow A[1]$ and $\tilde{b}_{2} : A[1]^{\otimes 2} \rightarrow A[1]$ given by 
$sa \mapsto -sda$ and $sa \otimes sb \mapsto (-1)^{\deg a + 1} s(a b) + \epsilon_{A}(a) sb - (-1)^{\deg a + 1} \epsilon_{A}(b) sa$, for $a, b \in A$ homogeneous, resp. 
Note that the definition just given before does not coincide with the one given in \cite{FMT}, for our differential is different (due to the last two terms in the explicit expression of $\tilde{B}$). 
In particular, their definition of nonreduced bar construction of a unitary dg algebra is always quasi-isomorphic to $k$, whereas in our case it will be quasi-isomorphic to the reduced bar construction. 

Moreover, the quasi-isomorphism $\operatorname{Bar}(A) \rightarrow \overline{\operatorname{Bar}}(A)$ of dg $A$-bimodules given by taking quotients and recalled in the third paragraph of Subsection \ref{subsec:barrhoch} induces a quasi-isomorphism $\tilde{B}^{+}(A) \rightarrow B^{+}(A)$ of dg modules over $k$ also given by taking quotients. 
By the explicit expressions for the coproduct, the counit and the coaugmentation for both bar constructions the previous map is in fact a quasi-isomorphism of coaugmented dg coalgebras. 
Furthermore, the quasi-isomorphism $\overline{\operatorname{Bar}}(A) \rightarrow \operatorname{Bar}(A)$ of dg $A$-bimodules induced by the inclusion $I_{A}[1] \subseteq A[1]$ 
also induces a quasi-isomorphism of coaugmented dg coalgebras $B^{+}(A) \rightarrow \tilde{B}^{+}(A)$, which is the quasi-inverse to 
the map $\tilde{B}^{+}(A) \rightarrow B^{+}(A)$.

We also want to stress the fact that the bar constructions are functorial. 
Indeed, if $f : A \rightarrow A'$ is a morphism of augmented dg algebras, the unique morphism of coaugmented graded coalgebras $B^{+}(f) : B^{+}(A) \rightarrow B^{+}(A')$ whose composition with the canonical projection $\pi_{1}' : B^{+}(A') \rightarrow I_{A'}[1]$, where $I_{A'}$ is the augmentation kernel of $A'$, is given by the composition of the canonical projection $\pi_{1} : B^{+}(A) \rightarrow I_{A}[1]$, where $I_{A}$ is the augmentation kernel of $A$, together with $- f_{[1]}^{[1]}$ commutes with the differentials, so it gives a morphism of coaugmented dg coalgebras. 
It can be described explicitly as $B^{+}(f)([a_{1} | \dots | a_{n}]) = [f(a_{1}) | \dots | f(a_{n})]$, for $n \in \NN_{0}$ and $a_{1}, \dots, a_{n} \in A$. 
We remark that the minus sign in front of the $f_{[1]}^{[1]}$ was not quite arbitrary: had we chosen the plus sign, then we should have added a $(-1)^{n}$ sign to the previous expression of 
$B^{+}(f)$. 
This would imply in particular that $B^{+}(\mathrm{id}_{A})$ would be different from $\mathrm{id}_{B^{+}(A)}$, so it would not be a functorial choice.  
The explicit expression of the corresponding morphism of coaugmented dg coalgebras $\tilde{B}^{+}(f) : \tilde{B}^{+}(A) \rightarrow \tilde{B}^{+}(A')$ for the nonreduced bar construction 
is the same as for the reduced one. 
Moreover, the quasi-isomorphisms $\tilde{B}^{+}(A) \rightarrow B^{+}(A)$ and $B^{+}(A) \rightarrow \tilde{B}^{+}(A)$ of coaugmented dg coalgebras described in the previous paragraph yield 
in fact natural transformations. 
It is not difficult to show that, given a morphism of augmented dg algebras $f : A \rightarrow A'$ over a Von Neumann regular ring $k$, 
$f$ is a quasi-isomorphism if and only if $B^{+}(f)$ (or $\tilde{B}^{+}(f)$) is also a quasi-isomorphism (see \cite{FHT01}, Section 19, Ex. 1, p. 271, or \cite{LH}, Lemme 1.3.2.2, and Lemme 1.3.2.3, (a) and (b)). 

If $A$ is an augmented dg algebra and we consider its bar construction $B^{+}(A)$, there is in fact a \emph{ (normalized) universal twisting cochain} $\tau_{A} : B^{+}(A) \rightarrow A$ of $A$ whose restriction to 
$I_{A}[1]^{\otimes n}$ vanishes if $n \neq 1$ and such that its restriction to $I_{A}[1]$ is given by minus the composition of canonical inclusion $I_{A}[1] \subseteq A[1]$ with $s_{A}^{-1}$. 
Analogously, we also have a a \emph{unnormalized universal twisting cochain} $\tilde{\tau}_{A} : \tilde{B}^{+}(A) \rightarrow A$ of $A$ 
whose restriction to $A[1]^{\otimes n}$ vanishes if $n \neq 1$ and such that its restriction to $A[1]$ is given by the composition of minus $s_{A}^{-1}$, the projection $A \rightarrow I_{A}$ 
(which can be described by the composition of the canonical projection $A \rightarrow A/k.1_{A}$ with the identification $A/k.1_{A} \simeq I_{A}$) and the inclusion $I_{A} \rightarrow A$. 
Note that, if $f : A \rightarrow A'$ is morphism of augmented dg algebras, then the morphism $B^{+}(f)$ satisfies that 
\begin{equation}
\label{eq:commb}
 f \circ \tau_{A} = \tau_{A'} \circ B^{+}(f)
\end{equation} 
and the same expression holds for the corresponding morphism $\tilde{B}^{+}(f)$ by considering the unnormalized universal twisting cochains. 

\subsubsection{Hochschild (co)homology as a twisted construction} 
\label{subsubsec:hochschildtwisted}

Now, given an augmented dg algebra $A$, we may consider the augmented dg algebra $\mathcal{H}om (B^{+}(A),A)$. 
It is easy to verify that its underlying graded $k$-module coincides with the corresponding one computing the Hochschild cohomology, but the underlying dg $k$-module is different. 
Indeed, the differential of $\mathcal{H}om (B^{+}(A),A)$ is given by $D_{0}$ in \eqref{eq:dif0cohomo}.
However, a twist of the differential of $\mathcal{H}om (B^{+}(A),A)$ will give us precisely the differential of $\mathcal{H}om_{A^{e}} (\overline{\operatorname{Bar}}(A),A)$: it is easy to check that $d_{\mathcal{H}om (B^{+}(A),A),\tau_{A}}$ coincides with the differential $D_{0} + D_{1}$ of the dg $k$-module computing the Hochschild cohomology of $A$ given by \eqref{eq:dif0cohomo} and \eqref{eq:dif1cohomo}, since it is clear that $D_{1}(f) = \mathrm{ad}(\tau_{A})(f)$. 
In other words, the dg $k$-module $\mathcal{H}om_{A^{e}} (\overline{\operatorname{Bar}}(A),A)$ is canonically identified with $\mathcal{H}om^{\tau_{A}} (B^{+}(A),A)$. 
The latter has further the structure of a(n augmented) dg algebra, whose multiplication is usually called \emph{cup product} (see \cite{CE}, Ch. XI, \S 4 and 6), so the Hochschild cohomology $HH^{\bullet}(A)$ thus becomes an augmented graded algebra. 
All the previous comments apply \textit{mutatis mutandi} as well to $\tilde{B}^{+}(A)$ instead of $B^{+}(A)$. 
Moreover, if $k$ is semisimple, 
the canonical quasi-isomorphisms $\tilde{B}^{+}(A) \rightarrow B^{+}(A)$ and $B^{+}(A) \rightarrow \tilde{B}^{+}(A)$ of coaugmented dg coalgebras 
induce quasi-isomorphisms of augmented dg algebras 
\[     \mathcal{H}om (B^{+}(A),A) \rightarrow \mathcal{H}om (\tilde{B}^{+}(A),A)  \hskip 1mm \text{   and   } \hskip 1mm \mathcal{H}om (\tilde{B}^{+}(A),A) \rightarrow \mathcal{H}om (B^{+}(A),A)     \] 
sending $\tau_{A}$ to $\tilde{\tau}_{A}$ and $\tilde{\tau}_{A}$ to $\tau_{A}$, respectively. 
Moreover, they also induce quasi-isomorphisms of augmented dg algebras 
\[     \mathcal{H}om^{\tau_{A}} (B^{+}(A),A) \rightarrow \mathcal{H}om^{\tilde{\tau}_{A}} (\tilde{B}^{+}(A),A) \]
and    
\[     \mathcal{H}om^{\tilde{\tau}_{A}} (\tilde{B}^{+}(A),A) \rightarrow \mathcal{H}om^{\tau_{A}} (B^{+}(A),A).     \]
Indeed, the two maps are obtained  by applying the functor $\mathcal{H}om_{A^{e}} (\place , A)$ to the quasi-isomorphisms $\operatorname{Bar}(A) \rightarrow \overline{\operatorname{Bar}}(A)$ 
and $\overline{\operatorname{Bar}}(A) \rightarrow \operatorname{Bar}(A)$, respectively, and using the canonical identifications explained in the last paragraph of Subsection \ref{subsec:barrhoch}. 

The Hochschild homology of an augmented dg algebra $A$ can be regarded in a similar fashion. 
The underlying graded $k$-module $A \otimes_{A^{e}} \overline{\operatorname{Bar}}(A)$ computing the Hochschild homology of $A$ coincides with $A \otimes B^{+}(A)$, 
though the dg $k$-module structure is different, for the differential of $A \otimes B^{+}(A)$ the former coincides with $D_{0}'$ given in \eqref{eq:dif0homo}. 
However, the twist of the differential of $A \otimes B^{+}(A)$ will give us precisely the differential of $A \otimes_{A^{e}} \overline{\operatorname{Bar}}(A)$ as for the cohomology. 
We see that $d_{A \otimes B^{+}(A),\tau_{A}}$ coincides with the differential $D_{0}'+D_{1}'$ of the dg $k$-module computing the Hochschild homology of $A$ given by \eqref{eq:dif0homo} and \eqref{eq:dif1homo}, since it is easily verified that $D_{1}'$ coincides with the action of $\mathrm{ad}(\tau_{A})$ on $A \otimes B^{+}(A)$. 
Thus, the dg $k$-module $A \otimes_{A^{e}} \overline{\operatorname{Bar}}(A)$ is canonically identified with $A \otimes_{\tau_{A}} B^{+}(A)$. 
By the last paragraph of Subsubsection \ref{subsubsec:twistedconvo}, the latter has further the structure of a dg bimodule over $\mathcal{H}om^{\tau_{A}} (B^{+}(A),A)$. 
Either the left or the right action (or both of them together) may be called \emph{cap product} (see \cite{CE}, Ch. XI, \S 4 and 6), so the Hochschild homology $HH_{\bullet}(A)$ in turn becomes a graded bimodule over the augmented graded algebra given by Hochschild cohomology $HH^{\bullet}(A)$. 
Again, all the previous comments apply \textit{mutatis mutandi} to $\tilde{B}^{+}(A)$ instead of $B^{+}(A)$. 
If $k$ is Von Neumann regular, the canonical quasi-isomorphism $\tilde{B}^{+}(A) \rightarrow B^{+}(A)$ (resp., $B^{+}(A) \rightarrow \tilde{B}^{+}(A)$) of coaugmented dg coalgebras induces a quasi-isomorphism of dg bimodules $A \otimes \tilde{B}^{+}(A) \rightarrow A \otimes B^{+}(A)$ over $\mathcal{H}om (B^{+}(A),A)$ 
(resp., $A \otimes B^{+}(A) \rightarrow A \otimes \tilde{B}^{+}(A)$ over $\mathcal{H}om (\tilde{B}^{+}(A),A)$), 
where the domain has a structure of bimodule over $\mathcal{H}om (B^{+}(A),A)$ (resp., $\mathcal{H}om (\tilde{B}^{+}(A),A)$) by 
$\mathcal{H}om (B^{+}(A),A) \rightarrow \mathcal{H}om (\tilde{B}^{+}(A),A)$ (resp., $\mathcal{H}om (\tilde{B}^{+}(A),A) \rightarrow \mathcal{H}om (B^{+}(A),A)$), 
which is the morphism of augmented dg algebras explained in the previous paragraph.  
Furthermore, we have a quasi-isomorphism of dg bimodules  
$A \otimes_{\tilde{\tau}_{A}} \tilde{B}^{+}(A) \rightarrow A \otimes_{\tau_{A}} B^{+}(A)$ (resp., $A \otimes_{\tau_{A}} B^{+}(A) \rightarrow A \otimes_{\tilde{\tau}_{A}} \tilde{B}^{+}(A)$) over the algebra $\mathcal{H}om^{\tau_{A}} (B^{+}(A),A)$ (resp., $\mathcal{H}om^{\tilde{\tau}_{A}} (\tilde{B}^{+}(A),A)$), where the domain has a structure of bimodule over the algebra  
$\mathcal{H}om^{\tau_{A}} (B^{+}(A),A)$ (resp., $\mathcal{H}om^{\tilde{\tau}_{A}} (\tilde{B}^{+}(A),A)$) given by $\mathcal{H}om^{\tau_{A}} (B^{+}(A),A) \rightarrow \mathcal{H}om^{\tilde{\tau}_{A}} (\tilde{B}^{+}(A),A)$ (resp., given by $\mathcal{H}om^{\tilde{\tau}_{A}} (\tilde{B}^{+}(A),A) \rightarrow \mathcal{H}om^{\tau_{A}} (B^{+}(A),A)$) of augmented dg algebras explained in the previous paragraph.
Indeed, the previous maps follow by applying the functor $A \otimes_{A^{e}} (\place)$ to the quasi-isomorphisms 
$\operatorname{Bar}(A) \rightarrow \overline{\operatorname{Bar}}(A)$ 
and $\overline{\operatorname{Bar}}(A) \rightarrow \operatorname{Bar}(A)$, respectively, and using the canonical identifications explained in the penultimate paragraph of Subsection \ref{subsec:barrhoch}. 

We remark that in fact the graded bimodule $HH_{\bullet}(A)$ over $HH^{\bullet}(A)$ is (graded) symmetric, as one may easily deduce as follows. 
Indeed, as noted in the literature (see for instance \cite{Me}, (10) and the proof of Lemma 16), $H^{\bullet}(A,A^{\#})$ is a symmetric graded bimodule over $HH^{\bullet}(A)$ 
(for the action in fact comes from the cup product on $HH^{\bullet}(A[M])$, where $A[M]$ is the dg algebra with underlying dg module given by $A \oplus M$, the usual product $(a,m) \cdot (a', m') = (a a', a m' + m a')$, 
unit $(1_{A},0_{M})$ and augmentation $(a,m) \mapsto \epsilon_{A}(a)$), which is isomorphic to $HH_{\bullet}(A)^{\#}$. 

We summarize our previous comments in the following result. 
\begin{fact}
\label{fact:hh}
Let $A$ be an augmented dg algebra over $k$, and let $\tau_{A}$ denote the universal twisting cochain of $A$.
Then, 
\begin{itemize}
\item[(i)] the dg $k$-module $\mathcal{H}om_{A^{e}} (\overline{\operatorname{Bar}}(A),A)$ computing Hochschild cohomology is canonically identified with $\mathcal{H}om^{\tau_{A}} (B^{+}(A),A)$. 
Moreover, the cup product on the first complex coincides exactly with the convolution product on the latter. 
\item[(ii)] the dg $k$-module $A \otimes_{A^{e}} \overline{\operatorname{Bar}}(A)$ computing Hochschild homology is canonically identified with the twisted tensor product $A \otimes_{\tau_{A}} B^{+}(A)$. 
Furthermore, the symmetric bimodule structure of the first complex over $\mathcal{H}om_{A^{e}} (\overline{\operatorname{Bar}}(A),A)$ given by the cap product coincides exactly with the symmetric bimodule structure of the latter complex over $\mathcal{H}om^{\tau_{A}} (B^{+}(A),A)$. 
\end{itemize}
\end{fact}

We see that the previous description of Hochschild homology and cohomology groups is by no mean accidental. 
Indeed, it is a direct consequence of the definitions once one notes that the reduced (resp., nonreduced) bar resolution of $A$ is canonically identified (as dg $k$-modules) 
with $A^{e} \otimes_{\tau_{A}} B^{+}(A)$ (resp., $A^{e} \otimes_{\tilde{\tau}_{A}} \tilde{B}^{+}(A)$), 
where $A^{e}$ is a dg $A$-bimodule with the \emph{outer structure} of $A^{e}$ given by $a (a' \otimes b') b = (a a') \otimes (b' b)$, for $a, a', b, b' \in A$. 
The identification isomorphism is given by $(a_{n+1} \otimes a_{0}) \otimes [a_{1} | \dots | a_{n}] \mapsto (-1)^{\deg a_{n+1} (\deg a_{0} + \epsilon)} a_{0} [a_{1} | \dots | a_{n}] a_{n+1}$, 
where $\epsilon = (\sum_{i=1}^{n} \deg a_{i}) - n$. 
Consider the dg $A$-bimodule structure of $A^{e} \otimes_{\tau_{A}} B^{+}(A)$ (resp., $A^{e} \otimes_{\tilde{\tau}_{A}} \tilde{B}^{+}(A)$) coming from the \emph{inner structure} 
of $A^{e}$ given by 
\[     a (a' \otimes b') b = (-1)^{(\deg a' \deg a + \deg b \deg b' + \deg a \deg b)} (a' b) \otimes (a b'),     \] 
for $a, a', b, b' \in A$. 
By the last paragraphs of the Subsubsections \ref{subsubsec:convo} and \ref{subsubsec:twistedconvo}, it induces a structure of dg bimodule over the algebra $\mathcal{H}om^{\tau_{A}}(B^{+}(A), A) \otimes A$ 
(resp., $\mathcal{H}om^{\tilde{\tau}_{A}}(\tilde{B}^{+}(A), A) \otimes A$) on the twisted tensor product $A^{e} \otimes_{\tau_{A}} B^{+}(A)$ (resp., $A^{e} \otimes_{\tilde{\tau}_{A}} \tilde{B}^{+}(A)$).  
By means of this structure the previous identification gives in fact an isomorphism of dg bimodules over $\mathcal{H}om^{\tau_{A}}(B^{+}(A), A) \otimes A$ (resp., $\mathcal{H}om^{\tilde{\tau}_{A}}(\tilde{B}^{+}(A), A) \otimes A$).  
If we apply the functors $\mathcal{H}om_{A^{e}}(\place,A)$ and $A \otimes_{A^{e}} (\place)$ to the previous identifications of the bar constructions we get precisely 
the description of the dg $k$-modules computing Hochschild cohomology and homology as described in the previous two paragraphs. 

\subsubsection{The cobar constructions} 
\label{subsubsec:cobar}

Even though an analogous definition of the (un)reduced bar resolution for coaugmented dg coalgebras would be possible (\textit{e.g.} following the lines of Drinfeld's idea), 
we would recall a more \textit{ad-hoc} definition of the (un)reduced cobar constructions, without passing through the corresponding cobar resolutions.  
Given a coaugmented dg coalgebra $C$, $\Omega^{+}(C)$ denotes the augmented dg algebra called the \emph{(reduced or normalized) cobar construction} of $C$ 
(see \textit{e.g.} \cite{FHT01}, Section 19, or \cite{LPWZ04}, Def. 8.4). 
Its underlying augmented graded algebra structure is given by the tensor algebra on the graded vector space $J_{C}[-1] = C/\operatorname{Im}(\eta_{C})[-1]$, with the product defined by concatenation, unit given by the inclusion of $k$ and the obvious augmentation given by the canonical projection onto $k$. 
Since the algebra is free, a differential $D$ is uniquely determined by its restriction to $J_{C}[-1]$, which we take to be the sum of two terms 
\[     - s_{J_{C}[-1]}^{-1} \circ d_{C} \circ s_{J_{C}[-1]}     \] 
and 
\[     (s_{J_{C}[-1]}^{\otimes 2})^{-1} \circ \Delta_{J_{C}} \circ s_{J_{C}[-1]}.     \] 
If $n \in \NN$, an element of the form $s_{J_{C}[-1]}^{-1}(\bar{c}_{1}) \otimes \dots \otimes s_{J_{C}[-1]}^{-1} (\bar{c}_{n})$ will be typically denoted by $\cl{c_{1} | \dots | c_{n}}$, where $c_{1}, \dots, c_{n} \in C$. 
Analogously, we may denote by $\cl{}$ the unit $1_{\Omega^{+}(C)}$ of the algebra $\Omega^{+}(C)$. 
We may now write the differential $D$ more explicitly as 
\begin{equation*}
\begin{split}
     D(\cl{c_{1} | \dots | c_{n}}) =& - \sum_{i=1}^{n} (-1)^{\epsilon_{i}} \cl{c_{1} | \dots | c_{i-1} | d_{C}(c_{i}) | c_{i+1} | \dots | c_{n}} 
     \\
     &+ \sum_{i=1}^{n} (-1)^{\epsilon_{i}+ \deg c_{i,(1)}^{-} + 1} \cl{c_{1} | \dots | c_{i-1} | c_{i,(1)}^{-} | c_{i,(2)}^{-} | c_{i+1} | \dots | c_{n}},    
\end{split} 
\end{equation*}
where $\Delta_{J_{C}}(c_{i}) = c_{i,(1)}^{-} \otimes c_{i,(2)}^{-}$ and $\epsilon_{i} = (\sum_{j=1}^{i-1} \deg c_{j}) - i +1$. 
Note the coincidence with the sign in \cite{LPWZ04}, Def. 8.4, but the difference with the one given in \cite{FMT}, Section 2.7. 
Notice that, using the identification between $J_{C}$ and $\operatorname{Ker}(\epsilon_{C})$ given in the first paragraph of Subsection \ref{subsec:barchoch}, we may have 
equivalently presented the reduced cobar construction of $C$ as the tensor algebra on the graded vector space $\operatorname{Ker}(\epsilon_{C})$, with the product defined by concatenation, unit given by the inclusion of $k$ and the obvious augmentation given by the canonical projection onto $k$. 
The differential given by the same expression as before, but one should use $\Delta_{\operatorname{Ker}(\epsilon_{C})}$, defined in the first paragraph of Subsection \ref{subsec:barchoch}, instead of 
$\Delta_{J_{C}}$. 

There is also a \emph{nonreduced (or unnormalized) cobar construction} $\tilde{\Omega}^{+}(C)$ of $C$, whose underlying augmented graded algebra structure is given by the tensor algebra on the graded vector space $C[-1]$, with the product defined by concatenation, unit given by the inclusion of $k$ and the obvious augmentation given by the canonical projection onto $k$. 
As in the reduced case, an element of the form $s_{C[-1]}^{-1}(c_{1}) \otimes \dots \otimes s_{C[-1]}^{-1} (c_{n})$ will be usually denoted by $\cl{c_{1} | \dots | c_{n}}$, 
where $c_{1}, \dots, c_{n} \in C$. 
The explicit expression of the differential is given by 
\begin{equation*}
\begin{split}
     \tilde{D}(\cl{c_{1} | \dots | c_{n}}) =& - \sum_{i=1}^{n} (-1)^{\epsilon_{i}} \cl{c_{1} | \dots | c_{i-1} | d_{C}(c_{i}) | c_{i+1} | \dots | c_{n}} 
     \\
     &+ \sum_{i=1}^{n} (-1)^{\epsilon_{i}+ \deg c_{i,(1)} + 1} \cl{c_{1} | \dots | c_{i-1} | c_{i,(1)} | c_{i,(2)} | c_{i+1} | \dots | c_{n}}
     \\
     &+ \cl{1_{C} | c_{1} | \dots | c_{n}} - (-1)^{\epsilon_{n}} \cl{c_{1} | \dots | c_{n} | 1_{C}},    
\end{split} 
\end{equation*}
where $\Delta_{C}(c_{i}) = c_{i,(1)} \otimes c_{i,(2)}$. 
One may easily check that it is a derivation of square zero. 
As for the bar construction, there is a quasi-isomorphism of augmented dg algebras $\tilde{\Omega}^{+}(C) \rightarrow \Omega^{+}(C)$ given by the direct sum of the maps 
$(-q_{[1]}^{[1]})^{\otimes n}$, for $n \in \NN_{0}$, where $q$ denotes the canonical projection $C \rightarrow C/\operatorname{Im}(\eta_{C})$. 
Note that by convention, the underlying maps of $-q_{[1]}^{[1]}$ and of $q$ coincide. 
Furthermore, the previous quasi-isomorphism has a quasi-inverse given by the morphism of augmented dg algebras $\Omega^{+}(C) \rightarrow \tilde{\Omega}^{+}(C)$ induced 
by the inclusion $\operatorname{Ker}(\epsilon_{C})[-1] \subseteq C[-1]$, where we have also used the identification between $J_{C}$ and $\operatorname{Ker}(\epsilon_{C})$ given in the first paragraph of this subsection. 
Notice again that our definition of nonreduced cobar construction differs from the one given for instance in \cite{FMT}, Subsection 2.7. 

If $C$ is a coaugmented dg coalgebra and we consider its cobar construction $\Omega^{+}(C)$, there is in fact a \emph{universal (normalized) twisting cochain} $\tau^{C} : C \rightarrow \Omega^{+}(C)$ of $C$, given by the composition of the canonical projection $C \rightarrow C/\operatorname{Im}(\eta_{C})$, $s_{C/\operatorname{Im}(\eta_{C})[-1]}^{-1}$ and the canonical inclusion 
of $C/\operatorname{Im}(\eta_{C})[-1]$ inside $\Omega^{+}(C)$. 
Analogously, we also have a a \emph{universal (unnormalized) twisting cochain} $\tau^{C} : C \rightarrow \tilde{\Omega}^{+}(C)$ of $C$ 
given by the composition of the projection $C \rightarrow \operatorname{Ker}(\epsilon_{C})$ (given by the composition of the canonical projection $C \rightarrow C/\operatorname{Im}(\eta_{C})$ and the identification $ C/\operatorname{Im}(\eta_{C}) \simeq \operatorname{Ker}(\epsilon_{C})$), the canonical inclusion $\operatorname{Ker}(\epsilon_{C}) \subseteq C$, $s_{C[-1]}^{-1}$ and 
the canonical inclusion of $C[-1]$ inside $\tilde{\Omega}^{+}(C)$. 
The adjective for the twisting cochains $\tau_{A}$ and $\tau^{C}$ is justified. 
Indeed, given an augmented dg algebra $A$ and a coaugmented dg coalgebra $C$, if $\operatorname{Tw}(C,A)$ denote the set of twisting cochains from $C$ to $A$, 
we have canonical maps 
\begin{equation}
\label{eq:isostw}
     \operatorname{Hom}_{\text{aug-dg-alg}} (\Omega^{+}(C),A) \rightarrow \operatorname{Tw}(C,A) \leftarrow \operatorname{Hom}_{\text{coaug-dg-coalg}} (C,B^{+}(A)),     
\end{equation}
where the left map is given by $g \mapsto g \circ \tau^{C}$ and the right one by $f \mapsto \tau_{A} \circ f$, the first space of homomorphism denotes the set of morphisms of augmented dg algebras and the last one the set of morphisms of coaugmented dg coalgebras. 
It is clear that left map is a bijection, and the same holds for the right one provided $C$ is cocomplete (see \cite{LH}, Lemme 2.2.1.5). 
Moreover these maps are clearly seen to be natural. 

We also remark that the cobar constructions are functorial. 
Given $f : C \rightarrow C'$ a morphism of coaugmented dg coalgebras, the unique morphism of augmented graded algebras $\Omega^{+}(f) : \Omega^{+}(C) \rightarrow \Omega^{+}(C')$ 
whose restriction to $J_{C}[-1]$, where $J_{C}$ is the coaugmentation cokernel of $C$, is given by the composition of $- f_{[-1]}^{[-1]}$ with the canonical inclusion 
$J_{C'}[-1] \rightarrow \Omega^{+}(C')$, where $J_{C'}$ is the coaugmentation cokernel of $C'$, commutes with the differentials, so it gives a morphism of augmented dg algebras. 
It is given explicitly by $\Omega^{+}(f)(\cl{ c_{1} | \dots | c_{n}}) = \cl{ f(c_{1}) | \dots | f(c_{n})}$, for $n \in \NN_{0}$ and $c_{1}, \dots, c_{n} \in C$. 
As for the bar construction we stress that the minus sign in front of the $f_{[-1]}^{[-1]}$ was not really arbitrary, since a plus sign would have added a $(-1)^{n}$ sign to the previous expression of 
$\Omega^{+}(f)$, so in particular $\Omega^{+}(\mathrm{id}_{C})$ would be different from $\mathrm{id}_{\Omega^{+}(C)}$ and it would not be a functorial choice.  
The explicit expression of the corresponding morphism of augmented dg algebras $\tilde{\Omega}^{+}(f) : \tilde{\Omega}^{+}(C) \rightarrow \tilde{\Omega}^{+}(C')$ 
for the nonreduced cobar construction is the same as for the reduced one. 
Moreover, the quasi-isomorphisms $\tilde{\Omega}^{+}(C) \rightarrow \Omega^{+}(C)$ and $\Omega^{+}(C) \rightarrow \tilde{\Omega}^{+}(C)$ of augmented dg algebras described in the previous paragraph yield in fact natural transformations. 
Note that, if $f : C \rightarrow C'$ is a morphism of coaugmented dg coalgebras, then 
\begin{equation}
\label{eq:commcob}
 \tau^{C} \circ f = \Omega^{+}(f) \circ \tau^{C'}
\end{equation} 
and the same expression holds for the corresponding morphism $\tilde{\Omega}^{+}(f)$ by considering the unnormalized universal twisting cochains. 

We would like to note that, in contrast with the property of preservation of quasi-isomorphisms of the bar construction(s) for augmented dg algebras, it may occur that 
a morphism of coaugmented dg coalgebras $f : C \rightarrow C'$ is a quasi-isomorphism such that $\Omega^{+}(f)$ (or $\tilde{\Omega}^{+}(f)$) is not. 
A typical example of such a situation may be constructed as follows. 
We will also assume in the rest of this subsubsection that $k$ is a semisimple. 
Take $A$ a unitary dg algebra and consider the augmented dg algebra $A_{+}$ whose underlying dg $k$-module is given by $A \oplus k$, and 
with product $(a,\lambda) \cdot (a',\lambda') = (a a' + a \lambda' + \lambda a, \lambda \lambda')$, unit $\eta_{A_{+}}$ given the canonical inclusion $k \subseteq A_{+}$ of the form $\lambda \mapsto (0_{A}, \lambda)$ and augmentation $\epsilon_{A_{+}}$ given by the canonical projection $A_{+} \rightarrow k$. 
It is easy to check that $B^{+}(\eta_{A_{+}})$ and $B^{+}(\epsilon_{A_{+}})$ are in fact quasi-isomorphisms of coaugmented dg coalgebras, for the underlying dg $k$-module of $B^{+}(A_{+})$ 
is just $k \oplus \operatorname{Bar}(A)[2]$, but $\Omega^{+}(B^{+}(\eta_{A_{+}}))$ and $\Omega^{+}(B^{+}(\epsilon_{A_{+}}))$ are not, since $\Omega^{+}(B^{+}(A_{+}))$ is quasi-isomorphic to 
$A_{+}$. 
We remark however that if $C$ and $C'$ are cocomplete, and either $\Omega^{+}(f)$ or $\tilde{\Omega}^{+}(f)$ is a quasi-isomorphism, 
then $f$ is also (see for instance \cite{LH}, Lemme 1.3.2.2, and Lemme 1.3.2.3, (a) and (c). 
For this last item see also the corresponding errata). 
Following the terminology of K. Lef\`evre-Hasegawa (see \cite{LH}, the definitions before Thm. 1.2.1.2), we will say that morphism of coaugmented dg coalgebras $f : C \rightarrow C'$ is a 
\emph{weak equivalence} if $\Omega^{+}(f)$ (or, equivalently, $\tilde{\Omega}^{+}(f)$) is a quasi-isomorphism. 
There is a standard criterion for a morphism $f$ of cocomplete coaugmented dg coalgebras to be a weak equivalence, which we now recall. 
In order to do so, we need to introduce the following definitions. 
We say that a coaugmented dg coalgebra $C$ has an \emph{admissible filtration} (or that $C$ is \emph{admissibly filtered}), 
if there is an exhaustive increasing sequence $\{ C_{i} \}_{i \in \NN_{0}}$ of coaugmented dg subcoalgebras of $C$, such that $C_{0} = k.1_{C}$. 
The primitive filtration of a cocomplete coaugmented dg coalgebra $C$, given by $C_{i} = \operatorname{Ker}(\Delta_{J_{C}}^{(i+1)}) \oplus k.1_{C}$, if $i \in \NN$ and $C_{0} = k.1_{C}$, is admissible, 
by its very definition. 
It is clear in this case that it is in fact a filtration of $C$ by coaugmented dg coalgebras, which induces a filtration of augmented dg algebras on $\Omega^{+}(C)$, 
that induces in turn an admissible filtration on $B^{+}(\Omega^{+}(C))$, which is called the \emph{$C$-primitive filtration}. 
Given a coaugmented dg coalgebra $C$ with an admissible filtration, we may construct the \emph{associated graded object}, 
\[     \mathrm{Gr}_{C_{\bullet}}(C) = \bigoplus_{i \in \NN_{0}} C_{i}/C_{i-1},     \]
that may be in principle regarded as a graded $k$-module over the grading group $G \times \ZZ$, 
where the last factor comes from the index $i$ of the filtration, that will be called the \emph{filtration grading}. 
Now, we further provide with the unique differential induced by that of $C$ which preserves the filtration grading, and it thus becomes a dg module over $k$. 
Given two coaugmented dg coalgebras $C$ and $C'$ provided with admissible filtrations $\{ C_{i} \}_{i \in \NN_{0}}$ and $\{ C_{i}' \}_{i \in \NN_{0}}$, resp., 
a \emph{filtered morphism} is a morphism $f : C \rightarrow C'$ of coaugmented dg coalgebras such that $f(C_{i}) \subseteq C_{i}'$, for all $i \in \NN_{0}$. 
It further induces a unique morphism $\mathrm{Gr}(f) : \mathrm{Gr}_{C_{\bullet}}(C) \rightarrow \mathrm{Gr}_{C_{\bullet}'}(C')$ of dg $k$-modules preserving the filtration grading, 
called the \emph{associated graded morphism}.  
We say that $f$ is a \emph{filtered quasi-isomorphism} if the associated graded morphism $\mathrm{Gr}(f)$ is a quasi-isomorphism. 
Given a filtered morphism $f : C \rightarrow C'$ of coaugmented dg coalgebras provided with admissible filtrations, 
it can be proved that $f$ is a weak equivalence if it is a filtered quasi-isomorphism (see \cite{LH}, Lemme 1.3.2.2). 
Moreover, given a morphism $f : A \rightarrow A'$ of augmented dg algebras, $B^{+}(f)$ is a filtered morphism of coaugmented dg coalgebras, where $B^{+}(A)$ and $B^{+}(A')$ 
are provided with the primitive filtration (see \cite{LH}, Lemme 1.3.2.3, (a)). 
Note that these filtrations are admissible, for both $B^{+}(A)$ and $B^{+}(A')$ are cocomplete. 
The same comments apply to the nonreduced bar construction, and moreover the quasi-isomorphisms $\tilde{B}^{+}(A) \rightarrow B^{+}(A)$ and $B^{+}(A) \rightarrow \tilde{B}^{+}(A)$ 
described in the previous subsection are filtered for the primitive filtrations of the cocomplete dg coalgebras $\tilde{B}^{+}(A)$ and $B^{+}(A)$, so the former quasi-isomorphisms 
are in fact weak equivalences. 

Finally, we recall that the canonical map $\beta_{A} : \Omega^{+}(B^{+}(A)) \rightarrow A$ given by $\cl{} \mapsto 1_{A}$ and 
$\cl{\omega_{1} | \dots | \omega_{n}} \mapsto (-1)^{n} s^{-1}_{I_{A}}(\pi_{1}(\omega_{1})) \dots s^{-1}_{I_{A}}(\pi_{1}(\omega_{n}))$ if $n \in \NN$, where $\pi_{1} : B^{+}(A) \rightarrow I_{A}[1]$ is the canonical projection and $\omega_{1}, \dots, \omega_{n}$ are elements in the coaugmentation cokernel of $B^{+}(A)$, 
is a quasi-isomorphism of augmented dg algebras (see \cite{FHT01}, Section 19, Ex. 2, or \cite{HMS}, Section II.4, Thm. II.4.4, or \cite{Prou}, Th. 2.28). 
It is the unique morphism of augmented dg algebras satisfying that the composition of $\tau^{B^{+}(A)}$ with it is $\tau_{A}$. 
We have also the morphism of coaugmented dg coalgebras $\beta^{C} : C \rightarrow B^{+}(\Omega^{+}(C))$ given by the unique such morphism that its composition with 
$\tau_{\Omega^{+}(C)} : B^{+}(\Omega^{+}(C)) \rightarrow \Omega^{+}(C)$ is $\tau^{C}$. 
Hence $\beta^{C}$ sends $1_{C}$ to $1_{B^{+}(\Omega^{+}(C))}$, and for $c \in \operatorname{Ker}(\epsilon_{C})$, it satisfies that 
\[     \beta^{C}(c) = -[\cl{c}] + \sum_{n \in \NN_{\geq 2}} (-1)^{n} [\cl{c_{(1)}^{-}}| \dots | \cl{c_{(n)}^{-}}],     \]
where $\Delta_{\operatorname{Ker}(\epsilon_{C})}^{(n)} (c) = c_{(1)}^{-} \otimes \dots \otimes c_{(n)}^{-}$ is the iterated coproduct of the comultiplication indicated in the first paragraph of Subsection \ref{subsec:barchoch}. 
If $C$ is cocomplete, $\beta^{C}$ is a filtered quasi-isomorphism, where $C$ is provided with the primitive filtration and $B^{+}(\Omega^{+}(C))$ with the $C$-primitive one (see \cite{LH}, Lemme 1.3.2.3, (c), and the corresponding errata). 

The previous comments may be used in order to provide a simpler resolution than the bar resolution of an augmented dg algebra $A$, when $A = \Omega^{+}(C)$, 
where $C$ is a coaugmented dg coalgebra. 
In this case, we already know that the reduced bar resolution of $\Omega^{+}(C)$ is isomorphic to $\Omega^{+}(C)^{e} \otimes_{\tau_{\Omega^{+}(C)}} B^{+}(\Omega^{+}(C))$. 
However, since $\beta^{C} : C \rightarrow B^{+}(\Omega^{+}(C))$ is a quasi-isomorphism of coaugmented dg coalgebras which satisfies by definition that 
$\tau_{\Omega^{+}(C)} \circ \beta^{C} = \tau^{C}$, then 
\[     \mathrm{id}_{\Omega^{+}(C)^{e}} \otimes \beta^{C} : \Omega^{+}(C)^{e} \otimes_{\tau^{C}} C \rightarrow \Omega^{+}(C)^{e} \otimes_{\tau_{\Omega^{+}(C)}} B^{+}(\Omega^{+}(C))     \]
is a quasi-isomorphism (see \cite{LH}, Prop. 2.2.4.1). 
Furthermore, the map is clearly a morphism of dg bimodules over $\Omega^{+}(C)$, where the action by the last tensor factor is given by the inner structure of $\Omega^{+}(C)^{e}$. 
Using the identification $\Omega^{+}(C)^{e} \otimes_{\tau^{C}} C \rightarrow \Omega^{+}(C) \otimes C \otimes \Omega^{+}(C)$ of graded bimodules over $\Omega^{+}(C)$ given by 
$\omega' \otimes \omega \otimes c \mapsto (-1)^{\deg \omega' (\deg \omega + \deg c)} \omega \otimes c \otimes \omega'$, the differential of the latter space becomes 
\begin{equation*}
\begin{split}
     \omega \otimes c \otimes \omega' \mapsto &D(\omega) \otimes c \otimes \omega' + (-1)^{\deg \omega} \omega \otimes d_{C}(c) \otimes \omega' 
     \\ 
      &+ (-1)^{\deg \omega + \deg c} \omega \otimes c \otimes D(\omega') 
     \\
     &+ (-1)^{\deg \omega + \deg c_{(1)}} \omega \otimes c_{(1)} \otimes \tau^{C}(c_{(2)}) \omega'      
     \\
     &- (-1)^{\deg \omega} \omega \tau^{C}(c_{(1)}) \otimes c_{(2)} \otimes \omega',
\end{split}
\end{equation*}
where $\Delta_{C}(c) = c_{(1)} \otimes c_{(2)}$. 
Interestingly, a quasi-inverse for the previous map $\mathrm{id}_{\Omega^{+}(C)^{e}} \otimes \beta^{C}$ can be easily constructed (see \cite{VP}, Th\'eo. 1.4). 
Indeed, the morphism 
\[     \gamma^{C} : \overline{\operatorname{Bar}}(\Omega^{+}(C)) \rightarrow \Omega^{+}(C) \otimes C \otimes \Omega^{+}(C)     \]
of graded $\Omega^{+}(C)$-bimodules satisfying that $\gamma^{C}(\omega_{0} [] \omega_{1}) = \omega_{0} \otimes 1_{C} \otimes \omega_{1}$ and 
\[     \gamma^{C}(\omega_{0} [\omega_{1} | \dots | \omega_{m}] \omega_{m+1}) 
              = \delta_{m,1} \sum_{j=1}^{n} (-1)^{\varepsilon_{j}+1} \omega_{0} \cl{c_{1} | \dots | c_{j-1}} \otimes c_{j} \otimes \cl{c_{j+1} | \dots | c_{n}} \omega_{2}     \]
if $m \in \NN$, where $\omega_{1} = \cl{c_{1} | \dots | c_{n}}$ and $\varepsilon_{j} = (\sum_{l=1}^{j-1} \deg c_{l}) + j - 1$, 
is in fact a morphism of dg bimodules over $\Omega^{+}(C)$, and it satisfies that 
it is a left inverse of $\mathrm{id}_{\Omega^{+}(C)^{e}} \otimes \beta^{C}$ (after using the canonical identifications explained before). 

\subsubsection{Application to the Hochschild (co)homology of Koszul algebras}
\label{subsubsec:applkos}

We want to stress that the description of Hochschild (co)homology in Subsubsection \ref{subsubsec:bar} together with Fact \ref{fact:hh} in Subsubsection \ref{subsubsec:hochschildtwisted} 
allow us to compute the augmented graded algebra structure on the Hochschild cohomology $HH^{\bullet}(A)$ and the corresponding graded bimodule structure 
on the Hochschild homology $HH_{\bullet}(A)$ using different coaugmented dg coalgebras. 
More precisely, assuming that $k$ is semisimple, if $C$ is a coaugmented dg coalgebra provided with a weak equivalence of coaugmented dg coalgebras $f' : C \rightarrow B^{+}(A)$, 
then the map $\mathcal{H}om(f',A)$ gives a quasi-isomorphism of augmented dg algebras from $\mathcal{H}om^{\tau_{A}} (B^{+}(A),A)$ to $\mathcal{H}om^{\tau_{A} \circ f'} (C,A)$, 
which in turn induces an isomorphism of augmented graded algebras from $HH^{\bullet}(A)$ to $H^{\bullet}(\mathcal{H}om^{\tau_{A} \circ f'} (C,A))$. 
We remark that the map $\mathcal{H}om(f',A)$ is a quasi-isomorphism, for it coincides with the map given by applying the functor $\mathcal{H}om_{A^{e}}(\place,A)$ to the quasi-isomorphism 
$\mathrm{id}_{A^{e}} \otimes f' : A^{e} \otimes_{\tau_{A} \circ f'} C \rightarrow A^{e} \otimes_{\tau_{A}} B^{+}(A)$ of dg $A$-bimodules, and using the obvious identifications explained in the last paragraph of 
Subsection \ref{subsec:barrhoch}. 
The quasi-isomorphism property of the stated map of dg $A$-bimodules follows from the equivalence between items (c) and (f) in \cite{LH}, Prop. 2.2.4.1, taking into account that item (f) holds trivially for the map $g_{\tau}$ in that statement is just our $f'$.
Analogously, under the same assumptions we have a quasi-isomorphism of dg bimodules over $\mathcal{H}om^{\tau_{A}} (B^{+}(A),A)$ from $A \otimes_{\tau_{A} \circ f'} C$ to $A \otimes_{\tau_{A}} B^{+}(A)$ defined by $\mathrm{id}_{A} \otimes f'$, where $A \otimes_{\tau_{A} \circ f'} C$ has a structure of dg bimodule over $\mathcal{H}om^{\tau_{A}} (B^{+}(A),A)$ 
by means of the morphism of dg algebras $\mathcal{H}om(f',A)$. 
It is indeed a quasi-isomorphism, because it coincides with the map given by applying the functor $A \otimes_{A^{e}} (\place)$ to the quasi-isomorphism 
$\mathrm{id}_{A^{e}} \otimes f' : A^{e} \otimes_{\tau_{A} \circ f'} C \rightarrow A^{e} \otimes_{\tau_{A}} B^{+}(A)$ of dg $A$-bimodules, and using the obvious identifications explained in the penultimate paragraph of 
Subsection \ref{subsec:barrhoch}. 
Hence, we get an isomorphism of graded bimodules over $HH^{\bullet}(A)$ from $H_{\bullet}(A \otimes_{\tau_{A} \circ f'} C)$ to $HH_{\bullet}(A)$, 
where the structure of graded bimodule over $HH^{\bullet}(A)$ of $H_{\bullet}(A \otimes_{\tau_{A} \circ f'} C)$ is given by the previous isomorphism 
$HH^{\bullet}(A) \rightarrow H^{\bullet}(\mathcal{H}om^{\tau_{A} \circ f'} (C,A))$. 

We have thus proved the following result.
\begin{theorem}
\label{theorem:koszul}
Let $A$ be an augmented dg algebra over a field $k$ and let $C$ be a coaugmented dg coalgebra 
provided with a weak equivalence of coaugmented dg coalgebras $f' : C \rightarrow B^{+}(A)$. 
Then, the map $\mathcal{H}om(f',A)$ gives a quasi-isomorphism of augmented dg algebras from $\mathcal{H}om^{\tau_{A}} (B^{+}(A),A)$ to $\mathcal{H}om^{\tau_{A} \circ f'} (C,A)$. 
Moreover, $\mathrm{id}_{A} \otimes f'$ is a quasi-isomorphism of dg bimodules over $\mathcal{H}om^{\tau_{A}} (B^{+}(A),A)$ from $A \otimes_{\tau_{A} \circ f'} C$ to $A \otimes_{\tau_{A}} B^{+}(A)$, where $A \otimes_{\tau_{A} \circ f'} C$ has a structure of dg bimodule over $\mathcal{H}om^{\tau_{A}} (B^{+}(A),A)$ by means of the morphism of dg algebras $\mathcal{H}om(f',A)$. 
\end{theorem}

This can be applied to the particular case when $A$ is for instance a (quadratic) Koszul algebra over a field $k$, and taking $C = \operatorname{Tor}^{A}_{\bullet}(k,k)$ a coaugmented dg coalgebra with zero differential, where the cohomological degree of the component $\operatorname{Tor}^{A}_{i}(k,k)$ is $-i$ (see for instance \cite{PP}, Ch. 1, Section 1, pp. 4--5).
The structure of this coaugmented dg coalgebra $C$ and its quasi-isomorphism to the bar construction $B^{+}(A)$ is well-known and lies in the heart of the Koszul property.   
We recall that if $A= TV/\cl{R}$ is a (quadratic) Koszul algebra, then we have the isomorphisms of Adams graded vector spaces
$\operatorname{Tor}^{A}_{0}(k,k) \simeq k$, $\operatorname{Tor}^{A}_{1}(k,k) \simeq V$, and 
\[     \operatorname{Tor}^{A}_{i}(k,k) \simeq \bigcap_{j=0}^{i-2} V^{\otimes j} \otimes R \otimes V^{\otimes (i-j)},     \]
for $i \in \NN_{\geq 2}$. 
Let us denote the latter intersection space by $J_{i}$, if $i \in \NN_{\geq 2}$, and set $J_{1} =V$ and $J_{0} = k$. 
Under this identification the coproduct $\Delta$ is defined as follows. 
The composition of the restriction of $\Delta$ to $J_{i}$ with the canonical projection onto $J_{i'} \otimes J_{i''}$, where $i = i' + i''$, is given by the canonical inclusion $J_{i} \subseteq J_{i'} \otimes J_{i''}$ of Adams graded vector subspaces of $TV$. 
The counit is given the canonical projection of $\operatorname{Tor}^{A}_{\bullet}(k,k)$ onto $\operatorname{Tor}^{A}_{0}(k,k) \simeq k$, and the cougmentation by the usual inclusion 
of $k \simeq \operatorname{Tor}^{A}_{0}(k,k)$ inside $\operatorname{Tor}^{A}_{\bullet}(k,k)$. 
We recall that the canonical map $f' : C \rightarrow B^{+}(A)$ induced by the inclusions 
\[     J_{i} \subseteq V^{\otimes i} \rightarrow I_{A}[1]^{\otimes i},     \]
for $i \in \NN_{0}$, and the previous identifications, is a quasi-isomorphism of coaugmented dg algebras (this is in fact equivalent to the Koszul property on $A$). 
In this case it is trivial to verify that the twisting cochain $\tau_{A} \circ f'$ is given by the map whose restriction to $V$ is minus the canonical inclusion $V \rightarrow A$, and the restriction to $V^{\otimes i}$ vanishes for $i \in \NN_{0} \setminus \{ 1 \}$. 
As a consequence, we see that it is possible to compute the algebra structure of the Hochschild cohomology $HH^{\bullet}(A)$ for a Koszul algebra $A$ by calculating $H^{\bullet}(\mathcal{H}om^{\tau_{A} \circ f'} (C,A))$ and the graded bimodule of the Hochschild homology $HH_{\bullet}(A)$ over the cohomology $HH^{\bullet}(A)$ 
by means of $H_{\bullet}(A \otimes_{\tau_{A} \circ f'} C)$. 
The former result on Hochschild cohomology gives precisely the algebraic description of the cup product for the Hochschild cohomology of a Koszul algebra given by \cite{BGSS}. 
Indeed, the main result of the mentioned article (stated at the introduction, p. 443, or after as Theorem 2.3) is just the (basis-dependent expression of the) direct consequence of the 
complete knowledge (which was already well-known) of the coaugmented graded coalgebra structure on the $\operatorname{Tor}_{\bullet}^{A}(k,k)$ of a Koszul algebra 
and the comments of the previous paragraph. 
We will later provide a generalization of the results explained here, since in general the requirement of having a (manageable) coaugmented dg coalgebra $C$ weak equivalent to $B^{+}(A)$ is somehow narrow (see Theorem \ref{theorem:final}). 

\subsubsection{Usual duality between the bar and cobar construction}
\label{subsubsec:barcobar}

We will assume in this subsubsection that $k$ is a field. 
A graded or dg vector space $M$ is \emph{locally finite dimensional} if each of the homogeneous components of $M$ is finite dimensional. 
Note that in this case the graded dual $M^{\#}$ is also locally finite dimensional and the canonical map $\iota_{M} : M \rightarrow (M^{\#})^{\#}$ given by 
$\iota_{M}(m)(f) = (-1)^{\deg m \deg f} f(m)$, for $m \in M$ and $f \in M^{\#}$ homogeneous, is an isomorphism of graded or dg vector spaces.
Furthermore, if $M$ and $N$ are locally finite dimensional, then $\iota_{M,N}$ is an isomorphism in the corresponding category.

We now recall the well-known fact that if $C$ is a coaugmented dg coalgebra, the graded dual $C^{\#}$ has a structure of augmented dg algebra, where the product is given by $\Delta_{C}^{\#} \circ \iota_{C,C}$, unit $\epsilon_{C}$ and augmentation given by $\omega \mapsto \omega(\eta_{C}(1_{k}))$, for $\omega \in C^{\#}$. 
Conversely if $A$ is a locally finite dimensional augmented dg algebra, then the graded dual $A^{\#}$ has a structure of (locally finite dimensional) coaugmented dg coalgebra, 
where the product is given by $\iota_{A,A}^{-1} \circ \mu_{A}^{\#}$, counit given by $\omega \mapsto \omega(\eta_{A}(1_{k}))$, for $\omega \in A^{\#}$, 
and coaugmentation given by $1_{k} \mapsto \epsilon_{A}$.  
Note that in this latter case the morphism $\iota_{A} : A \rightarrow (A^{\#})^{\#}$ is in fact an isomorphism of augmented dg algebras. 
Analogously, if $C$ is a locally finite dimensional coaugmented dg coalgebra, then $\iota_{C} : C \rightarrow (C^{\#})^{\#}$ is also an isomorphism of coaugmented dg coalgebras. 

The main duality properties we shall use between the bar and cobar constructions are the following ones. 
We shall state them for the reduced bar and cobar construction, though the exact same results \textit{mutatis mutandi} are obtained for the nonreduced ones. 
In particular, all the analogous morphisms for the nonreduced cases will be denoted just adding a tilde over the corresponding nonreduced ones that will be explicitly stated 
in the rest of the subsubsection. 
If $\Lambda$ is locally finite dimensional augmented dg algebra such that $B^{+}(\Lambda)$ is also locally finite dimensional, 
then $B^{+}(\Lambda)^{\#}$ is canonically isomorphic to $\Omega^{+}(\Lambda^{\#})$, as augmented dg algebras (see \cite{LPWZ04}, Lemma 8.6, (c), 
where, using the notation of that article, one should further impose that $\Omega C$ and $BA$ are locally finite dimensional. 
The same correction would apply to \cite{LPWZ08}, Lemma 1.15. 
They are more or less a consequence of \cite{FHT01}, Section 19, Ex. 3, p. 272). 
The isomorphism $j_{\Lambda} : \Omega^{+}(\Lambda^{\#}) \rightarrow B^{+}(\Lambda)^{\#}$ is the unique one satisfying that its restriction to the coaugmentation cokernel 
$J_{\lambda^{\#}} \simeq I_{\Lambda}^{\#}$ of $\Lambda^{\#}$ is minus the graded dual of the canonical projection $B^{+}(\Lambda) \rightarrow I_{\Lambda}[1]$, 
where $I_{\Lambda}$ is the augmentation ideal of $\Lambda$ (using the identification $I_{\Lambda}^{\#} \simeq J_{\lambda^{\#}}$ induced by the graded dual of the inclusion 
$I_{\Lambda} \subseteq \Lambda$, 
and the isomorphism $H_{\Lambda,k,(1,0_{G'}),0_{G}} : (\Lambda^{\#})[-1] \rightarrow (\Lambda[1])^{\#}$ of dg modules explained in the sixth paragraph of Subsection \ref{subsec:int}), 
and it may be explicitly given as follows. 
We remark that the choice of signs is exactly the one in order to make our map commute with the differentials. 
We will provide an explicit expression of this isomorphism. 
For $n \in \NN$ and $\omega_{1}, \dots, \omega_{n} \in I_{\Lambda}^{\#}$ the morphism sends $\cl{\omega_{1} | \dots | \omega_{n}}$ to the linear functional 
\[     [\lambda_{1} | \dots | \lambda_{m}] \mapsto (-1)^{\epsilon} \delta_{n,m} \omega_{1}(\lambda_{1}) \dots \omega_{n} (\lambda_{n}),     \]
where $\lambda_{1}, \dots, \lambda_{m} \in \Lambda$, $\delta_{n,m}$ is the Dirac delta sign, 
\begin{align*}     
\epsilon = &\deg \omega_{1} + \dots + \deg \omega_{n} + n + (\deg \omega_{2} + 1) (\deg \lambda_{1} + 1)  
\\
&+ \dots + (\deg \omega_{n} + 1 ) (\deg \lambda_{1} + \dots + \deg \lambda_{n-1} + n - 1),     
\end{align*} 
and it sends $1 \in \Omega^{+}(\Lambda^{\#})$ to the canonical projection $B^{+}(\Lambda) \rightarrow k$. 
We remark that the previous morphism is in fact a natural isomorphism between the functors $\Omega^{+}((\place)^{\#})$ and $B^{+}(\place)^{\#}$, 
\textit{i.e.} if $f : \Lambda \rightarrow \Lambda'$ is a morphism of augmented dg algebras, 
then it can be directly verified from the explicit expression of the morphisms involved that we have the commutative diagram 
\[ 
\xymatrix
{
\Omega^{+}((\Lambda')^{\#})
\ar[r]^{j_{\Lambda'}}
\ar[d]^{\Omega^{+}(f^{\#})}
&
B^{+}(\Lambda')^{\#}
\ar[d]^{B^{+}(f)^{\#}}
\\
\Omega^{+}(\Lambda^{\#})
\ar[r]^{j_{\Lambda}}
&
B^{+}(\Lambda)^{\#}
}
\] 
Analogously, if $D$ is locally finite dimensional coaugmented dg coalgebra such that $\Omega^{+}(D)$ is also locally finite dimensional, 
then $B^{+}(D^{\#})$ is canonically isomorphic to $\Omega^{+}(D)^{\#}$, as coaugmented dg coalgebras (see \cite{LPWZ04}, Lemma 8.6, (c), 
where the analogous corrections to the ones indicated before apply). 
Indeed, the isomorphism $j^{D} : B^{+}(D^{\#}) \rightarrow \Omega^{+}(D)^{\#}$ may be explicitly given as follows. 
For $n \in \NN$ and $\rho_{1}, \dots, \rho_{n} \in I_{D^{\#}}$ the morphism sends $[ \rho_{1} | \dots | \rho_{n} ]$ to the linear functional 
\[     \cl{\theta_{1} | \dots | \theta_{m}} \mapsto (-1)^{\epsilon} \delta_{n,m} \rho_{1}(\theta_{1}) \dots \rho_{n} (\theta_{n}),     \]
where $\theta_{1}, \dots, \theta_{m} \in D$, $\delta_{n,m}$ is the Dirac delta sign, 
\begin{align*}  
     \epsilon = &\deg \rho_{1} + \dots + \deg \rho_{n} + (\deg \rho_{2} + 1) (\deg \theta_{1} + 1)  
     \\
     &+ \dots + (\deg \rho_{n} + 1) (\deg \theta_{1} + \dots + \deg \theta_{n-1} + n - 1),
\end{align*}      
and it sends $1 \in B^{+}(D^{\#})$ to the canonical projection $\Omega^{+}(D) \rightarrow k$. 
Again, if $f : D \rightarrow D'$ is a morphism of coaugmented dg coalgebras, 
then it can be directly verified from the explicit expression of the morphisms involved that we have the commutative diagram 
\[ 
\xymatrix
{
B^{+}((D')^{\#})
\ar[r]^{j^{D'}}
\ar[d]^{B^{+}(f^{\#})}
&
\Omega^{+}(D')^{\#}
\ar[d]^{\Omega^{+}(f)^{\#}}
\\
B^{+}(D^{\#})
\ar[r]^{j^{D}}
&
\Omega^{+}(D)^{\#}
}
\] 
Moreover, a straighforward computation shows that 
\begin{equation}
\label{eq:impdual1}
\begin{split}
   j^{D} &= \Omega^{+}(\iota_{D})^{\#} \circ (j_{D^{\#}})^{\#} \circ \iota_{B^{+}(D^{\#})},
  \\
  j_{\Lambda} &= B^{+}(\iota_{\Lambda})^{\#} \circ (j^{\Lambda^{\#}})^{\#} \circ \iota_{\Omega^{+}(\Lambda^{\#})}.
\end{split}
\end{equation} 
Using the easy fact that $\iota_{M}^{\#} \circ \iota_{M^{\#}} = \mathrm{id}_{M^{\#}}$, for any dg $k$-module $M$, these identities imply that 
\begin{equation}
\label{eq:impdual2}
\begin{split}
   (j^{D})^{\#} \circ \iota_{\Omega^{+}(D)} &= j_{D^{\#}} \circ \Omega^{+}(\iota_{D}),
  \\
  (j_{\Lambda})^{\#} \circ \iota_{B^{+}(\Lambda)} &= j^{\Lambda^{\#}} \circ B^{+}(\iota_{D}).
\end{split}
\end{equation} 

On the other hand, it is trivial to verify from the expressions of the morphisms involved that 
\begin{equation}
\label{eq:candual1}
\begin{split}
   \tau_{D^{\#}} &= (\tau^{D})^{\#} \circ j^{D},
  \\
   \tau_{\Lambda}^{\#} &= j_{\Lambda} \circ \tau^{\Lambda^{\#}}.
\end{split}
\end{equation} 
These equations in turn imply the following identities 
\begin{equation}
\label{eq:candual2}
\begin{split}
   \beta_{D^{\#}} &= (\beta^{D})^{\#} \circ j_{\Omega^{+}(D)} \circ \Omega^{+}(j^{D}),
   \\
   \beta_{\Lambda}^{\#} &= j^{B^{+}(\Lambda)} \circ B^{+}(j_{\Lambda}) \circ \beta^{\Lambda^{\#}}.
\end{split}
\end{equation} 
Indeed, let us show how to prove the second one, for the first one is analogous.
We first note that, since $\beta_{\Lambda}$ is the unique morphism of augmented dg algebras such that the composition of $\tau^{B^{+}(\Lambda)}$ with it is $\tau_{\Lambda}$, by taking duals 
$\beta_{\Lambda}^{\#}$ is the unique morphism of coaugmented dg coalgebras such that its composition with $(\tau^{B^{+}(\Lambda)})^{\#}$ is $\tau_{\Lambda}^{\#}$. 
It thus suffices to prove that the composition of the right member with $(\tau^{B^{+}(\Lambda)})^{\#}$ is $\tau_{\Lambda}^{\#}$. 
By the first identity of \eqref{eq:candual1} for $D = B^{+}(\Lambda)$ we get that this composition is $\tau_{B^{+}(\Lambda)^{\#}} \circ B^{+}(j_{\Lambda}) \circ \beta^{\Lambda^{\#}}$. 
Now, \eqref{eq:commb} for the morphism $f = j_{\Lambda}$ tells us that the latter composition coincides with $j_{\Lambda} \circ \tau_{\Omega^{+}(\Lambda^{\#})} \circ \beta^{\Lambda^{\#}}$, 
which is equal to $j_{\Lambda} \circ \tau^{\Lambda^{\#}}$. 
The second identity of \eqref{eq:candual1} gives the claim. 

We remark that under these assumptions a quasi-isomorphism of coaugmented dg coalgebras is a weak equivalence (the converse is always true), 
for a quasi-isomorphism $C \rightarrow D$ induces a quasi-isomorphism of augmented dg algebras $D^{\#} \rightarrow C^{\#}$, which induces a quasi-isomorphisms between the bar constructions 
$B^{+}(D^{\#}) \rightarrow B^{+}(C^{\#})$, and by the previously recalled isomorphism we get a quasi-isomorphism $\Omega^{+}(D)^{\#} \rightarrow \Omega^{+}(C)^{\#}$ of coaugmented dg coalgebras. 
Taking duals again we obtain a quasi-isomorphism $\Omega^{+}(C) \rightarrow \Omega^{+}(D)$ of augmented dg algebras, so a weak equivalence $C \rightarrow D$. 

\subsubsection{Gerstenhaber brackets}
\label{subsubsec:ger}

We recall that given an augmented dg algebra $A$, the graded module $\operatorname{Der}(A)$ over $k$ of \emph{derivations} of $A$ is the graded submodule of $\mathcal{H}om(A,A)$ 
given by sums of homogeneous maps $d : A \rightarrow A$ satisfying that $\mu_{A} \circ (d \otimes \mathrm{id}_{A} + \mathrm{id}_{A} \otimes d) = d \circ \mu_{A}$ and $d(1_{A}) = 0$.  
Analogously, given a coaugmented dg coalgebra $C$, the graded module $\operatorname{Coder}(C)$ over $k$ of \emph{coderivations} of $C$ is the graded submodule of $\mathcal{H}om(C,C)$ 
given by sums of homogeneous maps $d : C \rightarrow C$ satisfying that $(d \otimes \mathrm{id}_{C} + \mathrm{id}_{C} \otimes d) \circ \Delta_{C} = \Delta_{C} \circ d$ 
and $\epsilon_{C} \circ d = 0$. 
Note that, by composing the first of the previous identities with $\eta_{C}$, we get that $d \circ \eta_{C} (1_{k})$ is a \textit{primitive} element of $C$, 
\textit{i.e.} an element $c \in C$ satisfying that $\Delta_{C}(C) = 1_{C} \otimes c + c \otimes 1_{C}$, where we recall that $1_{C} = \eta_{C}(1_{k})$. 
In the particular case of a coaugmented dg coalgebra $C$ whose underlying coaugmented graded coalgebra structure is a coaugmented tensor coalgebra 
on a graded module $V$, which is usually denoted by $T^{c}V$, we note that the primitive elements are given by $V \subseteq T^{c}V$. 
Note that both $\operatorname{Der}(A)$ and $\operatorname{Coder}(C)$ are graded Lie algebras with the bracket given by the graded commutator.

As noted by E. Getzler in \cite{Ge}, Prop. 1.3., for an augmented dg algebra $A$, the graded module $\mathcal{H}om(\tilde{B}^{+}(A),A)[1] \simeq \mathcal{H}om(\tilde{B}^{+}(A),A[1])$ over $k$ is isomorphic to the graded module $\operatorname{Coder}(\tilde{B}^{+}(A))$ of coderivations of the coaugmented dg coalgebra $\tilde{B}^{+}(A)$. 
Indeed, the isomorphism $\delta_{A}$ sends $s \phi$, where $\phi \in \mathcal{H}om(\tilde{B}^{+}(A),A)$ to the coderivation $\delta_{A}(s \phi)$ satisfying that 
$\delta_{A}(s \phi)(1_{\tilde{B}^{+}(A)}) = s_{A} (\phi(1_{\tilde{B}^{+}(A)}))$ and, for $n \in \NN$, $\pi_{j} (\delta_{A}(s \phi)([a_{1} | \dots | a_{n}]))$ is given by 
\begin{align*}
&\sum\limits_{i=1}^{j} (-1)^{(\deg \phi - 1)\epsilon_{i}} [a_{1} | \dots | a_{i-1} | \phi([a_{i}| \dots | a_{i+n-j}]) | a_{i+n-j+1} | \dots | a_{n}], &\text{if $1 \leq j \leq n$,} 
\\
&\sum\limits_{i=1}^{n+1} (-1)^{(\deg \phi - 1)\epsilon_{i}} [a_{1} | \dots | a_{i-1} | \phi(1_{B^{+}(A)}) | a_{i} | \dots | a_{n}],  &\text{if $j = n+1$,} 
\\
&0, &\text{if $j > n+1$,}                                                                                          
\end{align*}
where $\epsilon_{i} = (\sum_{j=1}^{i-1} \deg a_{j}) - i +1$ and $\pi_{j} : \tilde{B}^{+}(A) \rightarrow A[1]^{\otimes j}$ is the canonical projection. 
It is rather usual to provide the explicit expression of $\pi_{j} (\delta_{A}(\phi)([a_{1} | \dots | a_{n}]))$ only by the first of the previous lines, for the others are easily obtained 
as a typical abuse of (or ``extended'') notation. 
The inverse of $\delta_{A}$ is given by sending a coderivation $d \in \operatorname{Coder}(\tilde{B}^{+}(A))$ to the morphism $s_{A}^{-1} \circ \pi_{1} \circ d$. 

Also, given a coaugmented dg coalgebra $C$ we have a canonical isomorphism between the graded module 
$\mathcal{H}om(C,\tilde{\Omega}^{+}(C))[1] \simeq \mathcal{H}om(C[-1],\tilde{\Omega}^{+}(C))$ over $k$ and the graded module $\operatorname{Der}(\tilde{\Omega}^{+}(C))$ of derivations 
of the augmented dg algebra $\tilde{\Omega}^{+}(C)$. 
Indeed, the isomorphism $\delta^{C}$ sends $s \psi$, for $\psi \in \mathcal{H}om^{\tilde{\tau}_{A}}(C, \tilde{\Omega}^{+}(C))$ to the derivation $\delta^{C}(s \psi)$ satisfying that 
$\delta^{C}(s\psi)(1_{\tilde{\Omega}^{+}(C)}) = \psi (1_{A})$ and, for $n \in \NN$,  is given by 
\[     \delta^{C}(s \psi)(\cl{c_{1} | \dots | c_{n}}) = \sum\limits_{i=1}^{n+1} (-1)^{(\deg \psi - 1) (\bar{\epsilon}_{i} + 1))} \cl{c_{1} | \dots | c_{i-1}} \psi(c_{i}) \cl{c_{i+1} | \dots | c_{n}},     \]
where $\bar{\epsilon}_{i} = (\sum_{j=1}^{i-1} \deg c_{j}) - i +1$. 
The inverse is given by sending a derivation $d \in \operatorname{Der}(\tilde{\Omega}^{+}(C))$ to the morphism $(-1)^{\deg d} d|_{C[-1]} \circ s_{C[-1]}^{-1}$. 

From the previous identifications one obtains graded Lie algebra structures on both $\mathcal{H}om(\tilde{B}^{+}(A),A)[1]$ and $\mathcal{H}om(C,\tilde{\Omega}^{+}(C))[1]$. 
These bracket will be called the \emph{Gerstenhaber brackets}, for the first of these was introduce by M. Gerstenhaber in the seminal work \cite{Ger}. 
More explicitly, for $\phi \in \mathcal{H}om(A[1]^{\otimes n},A)$ and $\phi' \in \mathcal{H}om(A[1]^{\otimes m},A)$, where $n, m \in \NN_{0}$, the Gerstenhaber bracket is given by 
the element $[\phi , \phi'] \in \mathcal{H}om(A[1]^{\otimes (n+m-1)},A)$ sending $[a_{1} | \dots | a_{n+m-1}]$ to  
\begin{equation}
\label{eq:gerbra}
\begin{split}
    &\sum_{i=0}^{n-1} (-1)^{(\deg \phi' -1) \epsilon_{i+1}} \phi([a_{1} | \dots | a_{i} | \phi'([a_{i+1} | \dots | a_{i+m}]) | a_{i+m+1} | \dots | a_{n+m-1}]) 
    \\
    &- \sum_{i=0}^{m-1} (-1)^{\epsilon_{i}'} \phi'([a_{1} | \dots | a_{i} | \phi([a_{i+1} | \dots | a_{i+n}]) | a_{i+n+1} | \dots | a_{n+m-1}]),     
\end{split}
\end{equation}
where $\epsilon_{i+1} = (\sum_{j=1}^{i} \deg a_{j}) - i$ and $\epsilon_{i}' = (\deg \phi -1) (\epsilon_{i+1} + \deg \phi' -1)$ (see \cite{T}, Subsection 2.2). 
Note that the notation implies that if $n = 0$ the left sum of the right member vanishes and we should consider $\phi(1_{\tilde{B}^{+}(A)})$ in the right one, 
whereas if $m = 0$ the right sum is zero and we should have $\phi'(1_{\tilde{B}^{+}(A)})$ in the left one. 
It is straightforward to prove that the differential of $\mathcal{H}om^{\tilde{\tau}_{A}}(\tilde{B}^{+}(A),A)$ is in fact given by $\phi \mapsto [s_{A}^{-1} \circ \pi_{1} \circ \tilde{B} , \phi]$. 
This in turn implies that $\mathcal{H}om^{\tilde{\tau}_{A}}(\tilde{B}^{+}(A),A)[1]$ is in fact a dg Lie algebra. 
Moreover, the expression for the Gerstenhaber bracket \eqref{eq:gerbra} may be also applied to elements of $\mathcal{H}om(B^{+}(A),A)[1]$, which also becomes a dg Lie algebra such 
that the quasi-isomorphisms $\mathcal{H}om^{\tilde{\tau}_{A}}(\tilde{B}^{+}(A),A) \rightarrow \mathcal{H}om^{\tau_{A}}(B^{+}(A),A)$  and 
$\mathcal{H}om^{\tau_{A}}(B^{+}(A),A) \rightarrow \mathcal{H}om^{\tilde{\tau}_{A}}(\tilde{B}^{+}(A),A)$ recalled in the first paragraph of Subsubsection \ref{subsubsec:hochschildtwisted} 
induce in fact morphisms of dg Lie algebras. 

Analogously, let $\psi \in \mathcal{H}om(C, C[-1]^{\otimes n})$ and $\psi' \in \mathcal{H}om(C, C[-1]^{\otimes m})$, where $n, m \in \NN_{0}$. 
We will use the (Sweedler-alike) notation $\psi(c) = \cl{c^{\psi}_{(1)} | \dots | c^{\psi}_{(n)}}$ and $\psi'(c) = \cl{c^{\psi'}_{(1)} | \dots | c^{\psi'}_{(m)}}$, for $c \in C$, 
where the sum is omitted for simplicity. 
Then, the Gerstenhaber bracket is given by the element $[\psi , \psi'] \in \mathcal{H}om(C, C[-1]^{\otimes (n+m-1)})$ sending $c \in C$ to
\begin{equation}
\label{eq:gercobra}
\begin{split}
    &\sum_{i=0}^{m-1} (-1)^{(\deg \psi -1) \epsilon'_{i+1}} \cl{c_{(1)}^{\psi'} | \dots | c_{(i)}^{\psi'}} \psi(c_{(i+1)}^{\psi'}) \cl{c_{(i+2)}^{\psi'} | \dots | c_{(m)}^{\psi'}}
   \\
    &- \sum_{i=0}^{n-1} (-1)^{(\deg \psi' -1) (\epsilon_{i+1} + \deg \psi -1)} \cl{c_{(1)}^{\psi} | \dots | c_{(i)}^{\psi}} \psi'(c_{(i+1)}^{\psi}) \cl{c_{(i+2)}^{\psi} | \dots | c_{(n)}^{\psi}},     
\end{split}
\end{equation}
where $\epsilon_{i+1} = (\sum_{j=1}^{i} \deg c_{(j)}^{\psi}) - i$ and $\epsilon'_{i+1} = (\sum_{j=1}^{i} \deg c_{(j)}^{\psi'}) - i$.  
It can be proved along the same lines as for the case of algebras that this gives a dg Lie algebra structure on $\mathcal{H}om(C,\tilde{\Omega}^{+}(C))[1]$. 
Moreover, the expression for the Gerstenhaber bracket \eqref{eq:gercobra} may be also applied to elements of $\mathcal{H}om(C,\Omega^{+}(C))[1]$, which also becomes a dg Lie algebra such 
that the quasi-isomorphisms $\mathcal{H}om^{\tilde{\tau}^{A}}(C,\tilde{\Omega}^{+}(C)) \rightarrow \mathcal{H}om^{\tau^{C}}(C,\Omega^{+}(C))$  and 
$\mathcal{H}om^{\tau^{C}}(C,\Omega^{+}(C)) \rightarrow \mathcal{H}om^{\tilde{\tau}^{A}}(C,\tilde{\Omega}^{+}(C))$ induced by those recalled in the second paragraph of Subsubsection \ref{subsubsec:cobar} 
are in fact morphisms of dg Lie algebras.

Let us assume that $k$ is a field and 
$C$ is a locally finite dimensional coaugmented dg coalgebra. 
By the comments of the previous subsubsection we know that $C^{\#}$ is a locally finite dimensional augmented dg algebra. 
Moreover, the map $\operatorname{Coder}(C) \rightarrow \operatorname{Der}(C^{\#})^{\mathrm{op}}$ given by $\phi \mapsto (-1)^{\deg \phi} \phi^{\#}$ is an isomorphism of graded Lie algebras, 
where we recall that the \emph{opposite} graded Lie algebra $\mathfrak{g}^{\mathrm{op}}$ of a graded Lie algebra $\mathfrak{g}$ with bracket $[ \hskip 0.6mm , ]$ 
has the same underlying graded $k$-module structure and \emph{opposite bracket} $[ \hskip 0.6mm , ]_{\mathrm{op}}$ given by $[x,y]_{\mathrm{op}} = (-1)^{\deg x \deg y} [y, x]$($= - [x,y]$), 
for $x, y \in \mathfrak{g}$ homogeneous. 
In such situations we may usually say that the previous map is an \emph{anti-isomorphism of dg Lie algebras}. 
Analogously, if $A$ is a locally finite dimensional augmented dg algebra, then $A^{\#}$ is a locally finite dimensional coaugmented dg coalgebra, and 
$\operatorname{Der}(A) \rightarrow \operatorname{Coder}(A^{\#})^{\mathrm{op}}$ given by $\phi \mapsto \phi^{\#}$ is an isomorphism of graded Lie algebras. 
Since the map $\mathfrak{g} \rightarrow \mathfrak{g}^{\mathrm{op}}$ given by $x \mapsto -x$ is an isomorphism of graded Lie algebras, we get in fact  isomorphisms of graded Lie algebras $\operatorname{Coder}(C) \rightarrow \operatorname{Der}(C^{\#})$ and $\operatorname{Der}(A) \rightarrow \operatorname{Coder}(A^{\#})$ given by $\phi \mapsto - \phi^{\#}$, 
for $\phi \in \operatorname{Coder}(C)$ or $\phi \in \operatorname{Der}(A)$, respectively. 

\section{Koszul duality for Hochschild (co)homology} 
\label{sec:kosdual}

We will assume for the rest of the article that $k$ is a field. 
As we have seen in the previous subsection, the Hochschild cohomology of an augmented dg algebra $A$ can be realized as a twist of the augmented dg algebra $\mathcal{H}om(C,A)$, for a certain coaugmented dg coalgebra $C$, which is quasi-isomorphic to the bar construction $B^{+}(A)$ of $A$. 
If $f' : C \rightarrow B^{+}(A)$ is the quasi-isomorphism, the twist is given by a twisting cochain $\tau = \tau_{A} \circ f'$. 
Moreover the Hochschild homology of $A$ can be regarded in a similar fashion. 

Before stating the following result we recall that given an augmented dg algebra $\Lambda$ and a dg bimodule $M$ over $\Lambda$, the graded dual $M^{\#}$ is also a dg bimodule over $\Lambda$ provided with the action 
\[     (\lambda \cdot f \cdot \lambda') (m) = (-1)^{\deg \lambda (\deg f + \deg \lambda' + \deg m)} f (\lambda' m \lambda),     \]
where $m \in M$, $f \in M^{\#}$ and $\lambda, \lambda' \in \Lambda$ are homogeneous. 

The following result is straightforward but somehow tedious. 
\begin{proposition}
\label{prop:preprehoch}
Given a locally finite dimensional coaugmented dg coalgebra $C$ and a locally finite dimensional augmented dg algebra $A$, there is an isomorphism of augmented dg algebras 
\[     \mathcal{H}om (C,A) \rightarrow \mathcal{H}om (A^{\#},C^{\#})     \]
given by taking dual $\phi \mapsto \phi^{\#}$, the inverse being defined as $\psi \mapsto \iota_{A}^{-1} \circ \psi^{\#} \circ \iota_{C}$. 
Via this isomorphism, the dg $\mathcal{H}om (A^{\#},C^{\#})$-bimodule $C^{\#} \otimes A^{\#}$ becomes a dg $\mathcal{H}om (C,A)$-bimodule. 
Moreover, the canonical nondegenerate pairing 
\[     \beta : (C^{\#} \otimes A^{\#}) \otimes (A \otimes C) \rightarrow k     \]
given by $(g \otimes f) \otimes (a \otimes c) \rightarrow g(c) f(a)$, for $a \in A$, $c \in C$, $g \in C^{\#}$ and $f \in A^{\#}$ homogeneous, 
is a morphism of dg vector spaces which is also $\mathcal{H}om (C,A)^{e}$-balanced, \textit{i.e.} 
\[     \beta((g \otimes f).(\phi \otimes \psi) \otimes (a \otimes c)) = \beta((g \otimes f) \otimes (\phi \otimes \psi).(a \otimes c)),     \]
where $\phi \otimes \psi \in \mathcal{H}om (C,A)^{e}$, and we are using the obvious equivalence between dg bimodules and either left or right dg modules over the enveloping algebra. 
This further implies that there is an isomorphism of  dg $\mathcal{H}om (C,A)$-bimodules between $C^{\#} \otimes A^{\#}$ and $(A \otimes C)^{\#}$, 
given explicitly by $(g \otimes f) \mapsto ((a \otimes c)\mapsto f(a) g(c))$.  
\end{proposition}

We may also twist the previous result by a twisting cochain $\tau \in \mathcal{H}om (C,A)$. 
This implies that $\tau^{\#}$ is a twisting cochain of $\mathcal{H}om (A^{\#},C^{\#})$, and moreover the dual map still provides an isomorphism of augmented dg algebras 
\[     \mathcal{H}om^{\tau} (C,A) \rightarrow \mathcal{H}om^{\tau^{\#}} (A^{\#},C^{\#}).     \]
Moreover, the canonical pairing $\beta$ considered before can be regarded as pairing on the underlying graded vector spaces 
\[     \beta : (C^{\#} \otimes_{\tau^{\#}} A^{\#}) \otimes (A \otimes_{\tau} C) \rightarrow k,     \]
which is still $\mathcal{H}om (C,A)^{e}$-balanced, for the graded bimodule structures have not changed. 
Moreover, it is easily verified that it commutes with the new differentials, taking into account the previous isomorphism of augmented dg algebras and for the differential twists are given in terms of the bimodule structure. 
Hence $\beta$ is in fact a morphism of dg bimodules, which in turn implies that the previous isomorphism of dg $\mathcal{H}om (C,A)$-bimodules between $C^{\#} \otimes A^{\#}$ and $(A \otimes C)^{\#}$ induces an isomorphism of dg $\mathcal{H}om^{\tau} (C,A)$-bimodules between $C^{\#} \otimes_{\tau^{\#}} A^{\#}$ and $(A \otimes_{\tau} C)^{\#}$. 

We collect the previous remark in the following statement. 
\begin{proposition}
\label{prop:prehoch}
Given a locally finite dimensional coaugmented dg coalgebra $C$, a locally finite dimensional augmented dg algebra $A$ and a twisting cochain $\tau \in \mathcal{H}om (C,A)$, 
then $\tau^{\#}$ is also a twisting cochain and there is an isomorphism of augmented dg algebras 
\[     \mathcal{H}om^{\tau} (C,A) \rightarrow \mathcal{H}om^{\tau^{\#}} (A^{\#},C^{\#})     \]
given by taking dual $\phi \mapsto \phi^{\#}$, with inverse defined as $\psi \mapsto \iota_{A}^{-1} \circ \psi^{\#} \circ \iota_{C}$.  
Via this last isomorphism, the dg $\mathcal{H}om^{\tau^{\#}} (A^{\#},C^{\#})$-bimodule $C^{\#} \otimes_{\tau^{\#}} A^{\#}$ becomes a dg $\mathcal{H}om^{\tau} (C,A)$-bimodule. 
Moreover, the canonical nondegenerate pairing 
\[     \beta : (C^{\#} \otimes_{\tau^{\#}} A^{\#}) \otimes (A \otimes_{\tau} C) \rightarrow k     \]
given by $(g \otimes f) \otimes (a \otimes c) \rightarrow g(c) f(a)$, for $a \in A$, $c \in C$, $g \in C^{\#}$ and $f \in A^{\#}$ homogeneous, 
is $\mathcal{H}om^{\tau} (C,A)^{e}$-balanced, \textit{i.e.} 
\[     \beta((g \otimes f).(\phi \otimes \psi) \otimes (a \otimes c)) = \beta((g \otimes f) \otimes (\phi \otimes \psi).(a \otimes c)),     \]
where $\phi \otimes \psi \in \mathcal{H}om^{\tau} (C,A)^{e}$, and we are using the obvious equivalence between dg bimodules and either left or right dg modules over the enveloping algebra.  
This implies that there is an isomorphism of  dg $\mathcal{H}om^{\tau} (C,A)$-bimodules between $C^{\#} \otimes_{\tau^{\#}} A^{\#}$ and $(A \otimes_{\tau} C)^{\#}$, 
given explicitly by $(g \otimes f) \mapsto ((a \otimes c)\mapsto f(a) g(c))$.  
\end{proposition}

We will apply the previous result to obtain the isomorphism statements between Hochschild homology and cohomology groups of Koszul dual pairs. 
Indeed, following B. Keller, the \emph{Koszul dual} of an augmented dg algebra $A$ is defined as the augmented dg algebra $B^{+}(A)^{\#}$ (see \cite{Ke}, Section 10.2), 
and shall be denoted by $E(A)$.  

We are interested in the case that augmented dg algebras are \emph{Adams connected}, in the sense introduced in \cite{LPWZ08}, Def. 2.1, \textit{i.e.} an augmented dg algebra 
where the grading group $G$ is $\ZZ \times \ZZ$ such that its augmentation kernel $I_{A}$ satisfies that $\oplus_{n \in \ZZ} I_{A}^{(n,m)}$ is finite dimensional, for all $m \in G' = \ZZ$, and 
either $I_{A}^{(n,m)}$ vanishes for all $n \in \ZZ$ and all $m \geq 0$, or $I_{A}^{(n,m)}$ vanishes for all $n \in \ZZ$ and all $m \leq 0$.
By the previously cited article, an Adams connected augmented dg algebra is also locally finite dimensional, but moreover, its Koszul dual $E(A)$ is also locally finite dimensional and 
also Adams connected (see \cite{LPWZ08}, Lemma 2.2). 
If we apply this to the case $\Lambda = E(A)$, for $A$ an Adams connected augmented dg algebra, we get a quasi-isomorphism $\beta_{A}' : E(E(A)) \rightarrow A$ 
of augmented dg algebras given by $\beta_{A} \circ \Omega^{+}(\iota_{B^{+}(A)})^{-1} \circ j_{E(A)}^{-1}$, or else $\beta_{A} \circ \iota_{\Omega^{+}(B^{+}(A))}^{-1} \circ ((j^{B^{+}(A)})^{-1})^{\#}$, 
which follows by the first identity \eqref{eq:impdual2} with $D = B^{+}(A)$. 
Thus, we have a quasi-isomorphism $(\beta_{A}')^{\#} : A^{\#} \rightarrow E(E(A))^{\#}$ of coaugmented dg coalgebras. 
Moreover, it can be easily checked that $\tau_{E(A)} \circ \iota_{B^{+}(E(A))}^{-1} \circ (\beta_{A}')^{\#} = \tau_{A}^{\#}$. 
Indeed, by the first identity of \eqref{eq:impdual1} with $D = B^{+}(A)$, the previous expression reduces to prove $\tau_{E(A)} \circ (j^{B^{+}(A)})^{-1} \circ \beta_{A}^{\#} = \tau_{A}^{\#}$, 
which follows from a direct but long computation. 
This implies that we have a quasi-isomorphism of augmented dg algebras 
\begin{equation}
\label{eq:quasi-isohoch1}
     \mathcal{H}om(\iota_{B^{+}(E(A))}^{-1} \circ (\beta_{A}')^{\#},E(A))\hskip -0.6mm:\hskip -0.6mm\mathcal{H}om^{\tau_{E(A)}} (B^{+}(E(A)),E(A)) \rightarrow \mathcal{H}om^{\tau_{A}^{\#}} (A^{\#},E(A)).     
\end{equation}
By the proposition applied to $C = B^{+}(A)$ the latter augmented dg algebra is isomorphic to $\mathcal{H}om^{\tau_{A}} (B^{+}(A),A)$, which implies thus 
we have a quasi-isomorphism of augmented dg algebras of the form  
\begin{equation}
\label{eq:quasi-isohoch2}
     \mathcal{H}om^{\tau_{E(A)}} (B^{+}(E(A)),E(A)) \rightarrow \mathcal{H}om^{\tau_{A}} (B^{+}(A),A).     
\end{equation}
This implies in particular that there is an isomorphism of augmented graded algebras between the Hochschild cohomology rings $HH^{\bullet}(E(A)) \rightarrow HH^{\bullet}(A)$. 

We would like to remark that an explicit quasi-inverse for \eqref{eq:quasi-isohoch1} may be constructed. 
Indeed, the morphism 
\begin{equation}
\label{eq:quasi-isohoch1inv}
     \mathcal{H}om(B^{+}(j^{A^{\#}}),j^{A^{\#}}) \circ \mathcal{H}om_{\Omega^{+}(A^{\#})^{e}}(\gamma^{A^{\#}} , \Omega^{+}(A^{\#})) \circ \mathcal{H}om(A^{\#},(j^{A^{\#}})^{-1})    
\end{equation}
is such a quasi-inverse. 
This can be proved as follows. 
We consider the composition of $\mathcal{H}om(A^{\#},(j^{A^{\#}}))$, \eqref{eq:quasi-isohoch1} and $\mathcal{H}om(B^{+}(E(A)),(j^{A^{\#}})^{-1})$, which is just 
$\mathcal{H}om((j^{B^{+}(A)})^{-1} \circ \beta_{A}^{\#},\Omega^{+}(A^{\#}))$. 
Our claim is tantamount to proving that the latter is a quasi-inverse of $\mathcal{H}om(B^{+}(j^{A^{\#}}),\Omega^{+}(A^{\#})) \circ \mathcal{H}om_{\Omega^{+}(A^{\#})^{e}}(\gamma^{A^{\#}} , \Omega^{+}(A^{\#}))$, which is in turn equivalent to the fact that $\mathcal{H}om(B^{+}((j^{A^{\#}})^{-1}) \circ (j^{B^{+}(A)})^{-1} \circ \beta_{A}^{\#},\Omega^{+}(A^{\#}))$ 
is a quasi-inverse of $\mathcal{H}om_{\Omega^{+}(A^{\#})^{e}}(\gamma^{A^{\#}} , \Omega^{+}(A^{\#}))$. 
By second identity of \eqref{eq:candual2}, the former of these two morphism coincides with $\mathcal{H}om(\beta^{A^{\#}},\Omega^{+}(A^{\#}))$, which is clearly a quasi-inverse 
of $\mathcal{H}om_{\Omega^{+}(A^{\#})^{e}}(\gamma^{A^{\#}} , \Omega^{+}(A^{\#}))$, by the comments of the last paragraph of Subsubsection \ref{subsubsec:cobar}. 

We may apply similar arguments to the tensor products. 
Indeed, by taking $C = B^{+}(A)$ in the previous proposition, we see that there is an isomorphism of dg bimodules over $\mathcal{H}om^{\tau} (B^{+}(A),A)$
of the form $E(A) \otimes_{\tau_{A}^{\#}} A^{\#} \rightarrow (A \otimes_{\tau_{A}} B^{+}(A))^{\#}$. 
On the other hand, 
the morphism of dg vector spaces $\mathrm{id}_{E(A)} \otimes ((j^{B^{+}(A)})^{-1} \circ \beta_{A}^{\#}) : E(A) \otimes_{\tau_{A}^{\#}} A^{\#} \rightarrow E(A) \otimes_{\tau_{E(A)}} B^{+}(E(A))$ is a quasi-isomorphism. 
A quasi-inverse is given by the $\mathrm{id}_{E(A)} \otimes (B^{+}(j_{A}) \circ \gamma_{A^{\#}})$, as it directly follows from the second identity of \eqref{eq:candual2} 
and the comments of the last paragraph of Subsubsection \ref{subsubsec:cobar}. 

The composition of the inverse of the isomorphism $E(A) \otimes_{\tau_{A}^{\#}} A^{\#} \rightarrow (A \otimes_{\tau_{A}} B^{+}(A))^{\#}$ of the proposition 
together with the previous map $\mathrm{id}_{E(A)} \otimes ((j^{B^{+}(A)})^{-1} \circ \beta_{A}^{\#})$ thus gives a quasi-isomorphism 
\begin{equation}
\label{eq:quasi-isoch2}
      (A \otimes_{\tau_{A}} B^{+}(A))^{\#} \rightarrow E(A) \otimes_{\tau_{E(A)}} B^{+}(E(A)).     
\end{equation}
It is clear by the previous proposition and the comments in the penultimate paragraph of the Subsubsection \ref{subsubsec:barcobar} 
that it is also a morphism of dg bimodules 
over $\mathcal{H}om^{\tau_{E(A)}} (B^{+}(E(A)),E(A))$, where the domain of the map has a bimodule structure given by the morphism of augmented dg algebras 
\eqref{eq:quasi-isohoch2}. 
Since taking (co)ho\-mology commutes with taking duals we obtain an isomorphism $HH_{\bullet}(A)^{\#} \rightarrow HH_{\bullet}(E(A))$ 
of graded bimodules over $HH^{\bullet}(E(A))$, where the Hochschild homology $HH_{\bullet}(A)$ has the structure of bimodule over the Hochschild cohomology of $E(A)$ 
via the previous isomorphism $HH^{\bullet}(E(A)) \rightarrow HH^{\bullet}(A)$. 

We may summarize the previous comments in the following result, which may be regarded as the Koszul duality phenomenon for Hochschild (co)homology. 
\begin{theorem}
\label{theorem:koszuldual}
Let $A$ be an augmented dg algebra which is assumed to be Adams connected. 
We have a quasi-isomorphism of augmented dg algebras 
\[     \mathcal{H}om^{\tau_{E(A)}} (B^{+}(E(A)),E(A)) \rightarrow \mathcal{H}om^{\tau_{A}} (B^{+}(A),A),     \]
which yields an isomorphism of augmented graded algebras $HH^{\bullet}(E(A)) \rightarrow HH^{\bullet}(A)$. 
We also have a quasi-isomorphism of dg $\mathcal{H}om^{\tau_{E(A)}} (B^{+}(E(A)),E(A))$-bimodules 
\[     (A \otimes_{\tau_{A}} B^{+}(A))^{\#} \rightarrow E(A) \otimes_{\tau_{E(A)}} B^{+}(E(A)),        \]
where the domain has structure of bimodule over $\mathcal{H}om^{\tau_{E(A)}} (B^{+}(E(A)),E(A))$ via the first morphism of this proposition. 
We obtain an isomorphism $HH_{\bullet}(A)^{\#} \rightarrow HH_{\bullet}(E(A))$ of graded bimodules over $HH^{\bullet}(E(A))$, where the domain has structure of bimodule 
given by the isomorphism $HH^{\bullet}(E(A)) \rightarrow HH^{\bullet}(A)$. 
\end{theorem}

The isomorphism between the Hochschild homology groups has not received much attention, though in case $A$ is a (quadratic) Koszul algebra a linear isomorphism 
(\textit{i.e.} as graded modules over $k$) between $HH_{\bullet}(E(A))$ and $HH_{\bullet}(A)^{\#}$ can be obtained from (or following the lines of) the isomorphism between 
the corresponding cyclic homology groups given in \cite{FT}, Thm. 2.4.1. 
At that article, the authors suppose $k$ has characteristic zero, but this assumption is not necessary if dealing only with Hochschild homology groups (see \cite{Lo} for a more detailed analysis on 
the corresponding gradings). 
An isomorphism of graded algebras between $HH^{\bullet}(E(A))$ and $HH^{\bullet}(A)$ in case $A$ is a (quadratic) Koszul algebra was already announced by R.-O. Buchweitz in 
the Conference on Representation Theory, Canberra, in July 2003. 
Furthermore, the previous morphism \eqref{eq:quasi-isohoch2} for Hochschild cohomology has already appeared: 
though we have scarcely encountered it through the literature, we believe it should be well-known to the experts. 
By using the previous corresponding identifications, it is easy to see that it is essentially the same as the one appearing in \cite{FMT}, Prop. 5.1, 
though some caution should be taken into account. 
Their construction is somehow more delicate (and they use \emph{a priori} different versions of bar and cobar constructions) 
and their assumptions on the grading are in principle weaker, for they have just a cohomological grading. 
Their quasi-isomorphism result is though slightly different, since they need to assume that the coalgebra $C$ (whose dual would be our algebra $A$) is not only connected but also 
\emph{simply connected}, \textit{i.e.} the coaugmentation cokernel of $C$ lies in positive (homological) degrees greater than or equal to $2$, which is a very harsh assumption for us.  
They need it in order to ascertain that their map $\Theta$ (the analogous of our morphism $j_{C^{\#}}$) given before Prop. 5.1 of that article is an isomorphism.       
In any case, under our Adams connected assumption, the morphism they obtain is essentially given as follows. 
For a coaugmented dg coalgebra $C$ one may consider the morphism of dg modules over $k$ of the form  
\begin{equation}
\label{eq:morfmt} 
    \mathcal{H}om^{\tau_{\Omega^{+}(C)}} (B^{+}(\Omega^{+}(C)),\Omega^{+}(C)) \rightarrow \mathcal{H}om^{\tau_{C^{\#}}} (B^{+}(C^{\#}),C^{\#}),     
\end{equation}
which coincides (up to quasi-isomorphism) with the one given in \eqref{eq:quasi-isohoch2} for $A = C^{\#}$ 
using the identification of $\Omega^{+}(C)$ with $E(A)$ given by $(j^{C})^{\#} \circ \iota_{\Omega^{+}(C)}$. 
We want to remark that at least in this situation the previous proposition (or its variants) are essentially a direct consequence of the constructions involved: 
the only computation which may be considered as proper to this situation is $\tau_{E(A)} \circ (j^{B^{+}(A)})^{-1} \circ \beta_{A}^{\#} = \tau_{A}^{\#}$ (which is straightforward) 
that allows us to apply the general machinery described in the previous subsection. 
Indeed, the previous proposition is just the direct consequence of such a general machinery and the also completely general result given by Proposition \ref{prop:prehoch}.  

On the other hand, we want to remark that the authors in \cite{FMT}, Prop. 5.3, have further proved that the morphism \eqref{eq:morfmt} induces in fact an isomorphism of graded Lie algebras 
for the Gerstenhaber bracket on the Hochschild cohomology rings, implying that it thus induces a morphism of Gerstenhaber algebras. 
In fact a (slightly different but essentially equivalent) proof can be summarized as follows. 
As explained before, the morphism of dg modules \eqref{eq:quasi-isohoch1} can be considered as a composition of $\mathcal{H}om(\beta_{A}^{\#} \circ (j^{E(A)})^{-1},E(A))$ 
together with the inverse of the morphism $\mathcal{H}om^{\tau_{A}}(B^{+}(A),A) \rightarrow \mathcal{H}om^{\tau_{A}}(A^{\#},E(A))$ given by $\phi \mapsto \phi^{\#}$. 
Note the commutativity of the diagram 
\[
\xymatrix 
{
\mathcal{H}om^{\tau_{A}}(B^{+}(A),A)[1]
\ar[d]^{(\place)^{\#}}
\ar[r]^-{\sim}
&
\operatorname{Coder}(\tilde{B}^{+}(A))
\ar[d]^{(\place)^{\#}}
\\
\mathcal{H}om^{\tau_{A}^{\#}}(A^{\#},B^{+}(A)^{\#})[1]
\ar[r]^-{\sim}
&
\operatorname{Der}(\tilde{B}^{+}(A)^{\#})
\\
\mathcal{H}om^{\tau^{A^{\#}}}(A^{\#},\Omega^{+}(A^{\#}))[1]
\ar[u]^{\mathcal{H}om(A^{\#},j_{A})}
\ar[r]^-{\sim}
&
\operatorname{Der}(\tilde{\Omega}^{+}(A^{\#}))
\ar[u]^{\mathcal{H}om(\tilde{j}_{A}^{-1},\tilde{j}_{A})}
}
\]
where we have employed the abuse of notation explained in the penultimate paragraph of Subsection \ref{subsec:int} for the left vertical maps, and the horizontal maps 
are given as follows. 
Using the same abuse of notation, the upper horizontal morphism is given by the composition of 
$\mathcal{H}om^{\tau_{A}}(B^{+}(A),A) \rightarrow \mathcal{H}om^{\tilde{\tau}_{A}}(\tilde{B}^{+}(A),A)$ induced by the quasi-isomorphism $\tilde{B}^{+}(A) \rightarrow B^{+}(A)$ 
recalled in the third paragraph of Subsubsection \ref{subsubsec:bar}, with $\delta_{A}$. 
The bottom horizontal morphism is given by the composition of 
$\mathcal{H}om^{\tau^{A^{\#}}}(A^{\#},\Omega^{+}(A^{\#})) \rightarrow \mathcal{H}om^{\tilde{\tau}^{A^{\#}}}(A^{\#},\tilde{\Omega}^{+}(A^{\#}))$ induced by the quasi-isomorphism 
$\Omega^{+}(A^{\#}) \rightarrow \tilde{\Omega}^{+}(A^{\#})$ 
recalled in the second paragraph of Subsubsection \ref{subsubsec:cobar}, with $\delta^{A^{\#}}$.  
The middle horizontal map is just given in order to make the lower square diagram commute. 
We define thus a dg Lie algebra structure on $\mathcal{H}om^{\tau_{A}^{\#}}(A^{\#},B^{+}(A)^{\#})[1]$. 
Hence, each of the horizontal maps are thus quasi-isomorphisms of dg Lie algebras, and moreover the right vertical maps of the lower square diagram are isomorphisms of dg Lie algebras. 
Since the right vertical map $(\place)^{\#}$ is an anti-isomorphism of dg Lie algebras by the comments in the last paragraph of Subsubsection \ref{subsubsec:ger}, 
the same holds for the upper most left vertical map $(\place)^{\#}$. 

Furthermore, we also have the diagram 
\[
\xymatrix 
{
\mathcal{H}om^{\tau^{A^{\#}}}(A^{\#},\Omega^{+}(A^{\#}))[1]
\ar[d]^{\gamma}
\ar[r]^-{\sim}
&
\operatorname{Der}(\tilde{\Omega}^{+}(A^{\#}))
\\
\mathcal{H}om^{\tau_{\Omega^{+}(A^{\#})}}(B^{+}(\Omega^{+}(A^{\#})),\Omega^{+}(A^{\#}))[1]
\ar[r]^-{\sim}
\ar[d]^{\mathcal{H}om(B^{+}(j_{A}^{-1}),j_{A})}
&
\operatorname{Coder}(\tilde{B}^{+}(\Omega^{+}(A^{\#})))
\ar[d]^{\mathcal{H}om(\tilde{B}^{+}(j_{A}^{-1}),\tilde{B}^{+}(\tilde{j}_{A}))}
\\
\mathcal{H}om^{\tau_{E(A)}}(B^{+}(E(A)),E(A))[1]
\ar[r]^-{\sim}
&
\operatorname{Coder}(\tilde{B}^{+}(E(A)))
}
\]
where we have employed the abuse of notation explained in the penultimate paragraph of Subsection \ref{subsec:int} for the vertical maps, 
and the morphism $\gamma$ is induced by $\mathcal{H}om_{\Omega^{+}(A^{\#})^{e}}(\gamma^{A^{\#}},A)$ by making use of the identifications 
$\mathcal{H}om_{\Lambda^{e}}(\Lambda \otimes V \otimes \Lambda,\Lambda) \simeq \mathcal{H}om(V,\Lambda)$ of graded modules over $k$ 
(for $\Lambda$ an augmented dg algebra and $V$ a graded module over $k$), which in our case are also compatible with the differentials involved. 
Using the same abuse of notation, the upper horizontal map of this diagram is given by the composition of 
$\mathcal{H}om^{\tau^{A^{\#}}}(A^{\#},\Omega^{+}(A^{\#})) \rightarrow \mathcal{H}om^{\tilde{\tau}^{A^{\#}}}(A^{\#},\tilde{\Omega}^{+}(A^{\#}))$ induced by the quasi-isomorphism 
$\Omega^{+}(A^{\#}) \rightarrow \tilde{\Omega}^{+}(A^{\#})$ 
recalled in the second paragraph of Subsubsection \ref{subsubsec:cobar}, with $\delta^{A^{\#}}$. 
The middle and lower horizontal morphisms are given by the composition of 
$\mathcal{H}om^{\tau_{\Lambda}}(B^{+}(\Lambda),\Lambda) \rightarrow \mathcal{H}om^{\tilde{\tau}_{\Lambda}}(\tilde{B}^{+}(\Lambda),\Lambda)$ induced by the quasi-isomorphism 
$\tilde{B}^{+}(\Lambda) \rightarrow B^{+}(\Lambda)$ recalled in the third paragraph of Subsubsection \ref{subsubsec:bar}, with $\delta_{\Lambda}$, for $\Lambda = \Omega^{+}(A^{\#})$ and 
$\Lambda = E(A)$, respectively. 
Thus, each of the horizontal maps is a quasi-isomorphism of dg Lie algebras, and moreover the right vertical maps of the lower square diagram are clearly isomorphisms of dg Lie algebras. 
By the explicit expression of the morphisms involved it is straightforward to show that the upper left vertical map $\gamma$ is in fact an anti-morphism of dg Lie algebras. 
Note that it is a quasi-isomorphism, and in fact the composition of the left vertical maps of the two previous diagrams give a quasi-inverse to the map \eqref{eq:quasi-isohoch2}.
This proves that the latter induces a morphism of graded Lie algebras on Hochschild cohomology, as claimed before. 

Moreover, it is a direct consequence of the previous considerations that the isomorphism of dg bimodules over the Hochschild cohomology $HH^{\bullet}(E(A))$ 
between the Hochschild homology groups $HH_{\bullet}(A)^{\#}$ and $HH_{\bullet}(E(A))$ given in Theorem \ref{theorem:koszuldual} is also a morphism of graded Lie bimodules. 
We first note that the left and right Lie actions on the Hochschild homology group $HH_{\bullet}(\Lambda)$ over the graded Lie algebra $HH^{\bullet}(\Lambda)[1]$, where $\Lambda$ is 
an augmented dg algebra, are induced by the graded commutators of the corresponding dg left or right action of the augmented dg algebra 
$\mathcal{H}om^{\tau_{\Lambda}}(B^{+}(\Lambda),\Lambda)$ on $\Lambda \otimes_{\tau_{\Lambda}} B^{+}(\Lambda)$ with the Connes map $B_{\Lambda}$, which is defined on $\Lambda \otimes_{\tau_{\Lambda}} B^{+}(\Lambda)$ by the formula 
\begin{equation}
\label{eq:con}
   B_{\Lambda} (\lambda_{0} \otimes [\lambda_{1} | \dots | \lambda_{n}]) = \sum_{i=0}^{n} (-1)^{\epsilon_{i+1} \epsilon^{i}} 1_{\Lambda} \otimes [\lambda_{i+1} | \dots | \lambda_{n} | \lambda_{0} | \dots | \lambda_{i}],
\end{equation}
where $\lambda_{j} \in \Lambda$, for $j = 0, \dots, n$, $\epsilon_{i} = (\sum_{j=0}^{i-1} \deg \lambda_{j}) - i$ and $\epsilon^{i} = (\sum_{j=i+1}^{n} \deg \lambda_{j}) - n + i$ (see \cite{T}, Section 2.1, (11)). 
More precisely, consider the augmented dg algebra $k[\varepsilon]/(\varepsilon^{2})$, where $\varepsilon$ has complete (cohomological) degree $(-1,0_{G'})$ and the differential is trivial. 
It is clearly a graded commutative algebra. 
The Connes map $B_{\Lambda}$ defines a left action of it on $\Lambda \otimes_{\tau_{\Lambda}} B^{+}(\Lambda)$ by the formula 
\[     \varepsilon \cdot (\lambda_{0} \otimes [\lambda_{1} | \dots | \lambda_{n}]) = B_{\Lambda} (\lambda_{0} \otimes [\lambda_{1} | \dots | \lambda_{n}]),     \]
(see \cite{T}, Prop 2.1), and a right action by the usual expression $v \cdot \varepsilon = (-1)^{\deg v} \varepsilon \cdot v$, for $v \in \Lambda \otimes_{\tau_{\Lambda}} B^{+}(\Lambda)$ 
homogeneous (this follows from the identities $B_{\Lambda}^{2} = 0$ and 
$B_{\Lambda} \circ D' = - D' \circ B_{\Lambda}$, where $D'$ is the differential of $\Lambda \otimes_{\tau_{\Lambda}} B^{+}(\Lambda)$). 
Moreover, these two actions clearly commute, so $\Lambda \otimes_{\tau_{\Lambda}} B^{+}(\Lambda)$ is a dg bimodule over $k[\varepsilon]/(\varepsilon^{2})$. 
The left Lie action of $\phi \in \mathcal{H}om^{\tau_{\Lambda}}(B^{+}(\Lambda),\Lambda)$ on $\bar{\lambda} \in \Lambda \otimes_{\tau_{\Lambda}} B^{+}(\Lambda)$, 
which we denote by $L_{\phi}(\bar{\lambda})$ is given by the graded commutator of 
the left action operator of the augmented dg algebra $\mathcal{H}om^{\tau_{\Lambda}}(B^{+}(\Lambda),\Lambda)$ on $\Lambda \otimes_{\tau_{\Lambda}} B^{+}(\Lambda)$ 
with the left action operator of $k[\varepsilon]/(\varepsilon^{2})$ on the same space, 
\textit{i.e.} $L_{\phi}(\bar{\lambda}) = \epsilon \cdot (\phi \cdot \bar{\lambda}) - (-1)^{\deg \phi}  \phi \cdot (\epsilon \cdot \bar{\lambda})$. 
The action on the dual space $(\Lambda \otimes_{\tau_{\Lambda}} B^{+}(\Lambda))^{\#}$ is given by the usual formula in representation theory of dg Lie algebras
$L_{\phi}(\bar{\lambda}') = - (-1)^{\deg \bar{\lambda}' \deg \phi} \bar{\lambda}' \circ L_{\phi}$, where $\bar{\lambda}' \in (\Lambda \otimes_{\tau_{\Lambda}} B^{+}(\Lambda))^{\#}$. 
In order to be coherent with the previous left Lie action convention on the graded dual, we remark that in the case of a dg bimodule $M$ over $k[\varepsilon]/(\varepsilon^{2})$ 
we shall use the following convention (only of interest to us for the case $M = \Lambda \otimes_{\tau_{\Lambda}} B^{+}(\Lambda)$) for the left action (and thus right by the usual Koszul sign rule) on the graded dual $M^{\#}$: 
\[     (\epsilon \cdot f) (m) = - (-1)^{\deg f} f (\epsilon m).     \] 
For more details on these definitions and the fact that the usual dg bimodule structures together with the Lie module structures are part (together with the Gerstenhaber algebra structure on Hochschild cohomology) of a \emph{Tamarkin-Tsygan (pre)calculus}, we refer the reader to \cite{T}, in particular Prop. 2.9 and Subsection 4.3. 

On the other hand, we recall that the quasi-isomorphism of dg bimodules over $\mathcal{H}om^{\tau_{E(A)}}(B^{+}(E(A),E(A))$ stated in Theorem \ref{theorem:koszuldual} is given by 
$(\mathrm{id}_{E(A)} \otimes (B^{+}(j_{A}) \circ \beta^{A^{\#}})) \circ \tau_{A^{\#},B^{+}(A)^{\#}} \circ \iota_{A,B^{+}(A)}^{-1}$. 
Our claim is tantamount to the fact that the previous map induces a morphism of graded bimodules over $k[\varepsilon]/(\varepsilon^{2})$ 
between the Hochschild homology groups $HH_{\bullet}(A)^{\#}$ and $HH_{\bullet}(E(A))$. 
By the comments in the last paragraph of Subsubsection \ref{subsubsec:cobar}, it has an explicit quasi-inverse of the form 
\[     \iota_{A,B^{+}(A)} \circ \tau_{B^{+}(A)^{\#},A^{\#}} \circ  (j_{A} \otimes \mathrm{id}_{A^{\#}}) \circ (\mathrm{id}_{\Omega^{+}(A^{\#})} \otimes_{\Omega^{+}(A^{\#})^{e}} \gamma^{A^{\#}}) \circ (j_{A}^{-1} \otimes B^{+}(j_{A}^{-1})).     \] 
Let us denote it by $\gamma'$. 
Hence, our claim is equivalent to the fact that $\gamma'$ is a morphism of dg bimodules over $k[\varepsilon]/(\varepsilon^{2})$, 
which reduces to proving that $\gamma' \circ B_{E(A)} = B_{A}^{\#} \circ \gamma'$. 
By the explicit expression of the Connes map \eqref{eq:con}, we see that $(j_{A} \otimes B^{+}(j_{A})) \circ B_{\Omega^{+}(A^{\#})} = B_{E(A)} \circ (j_{A} \otimes B^{+}(j_{A}))$, so it suffices to show that $\gamma'' \circ B_{\Omega^{+}(A^{\#})} = B_{A}^{\#} \circ \gamma''$, where $\gamma''$ is given by 
\[     \iota_{A,B^{+}(A)} \circ \tau_{B^{+}(A)^{\#},A^{\#}} \circ  (j_{A} \otimes \mathrm{id}_{A^{\#}}) \circ (\mathrm{id}_{\Omega^{+}(A^{\#})} \otimes_{\Omega^{+}(A^{\#})^{e}} \gamma^{A^{\#}}).     \] 
Moreover, by transport of structures by means of the composite morphism given by $\iota_{A,B^{+}(A)} \circ \tau_{B^{+}(A)^{\#},A^{\#}} \circ  (j_{A} \otimes \mathrm{id}_{A^{\#}})$, 
it is easy to show that the left action of $k[\varepsilon]/(\varepsilon^{2})$ on the (intermediate) graded $k$-module $\Omega^{+}(A^{\#}) \otimes A^{\#}$ is given by 
\begin{multline*}
     \epsilon \cdot (\cl{\omega_{1} | \dots | \omega_{m}} \otimes \omega_{0}) 
     \\
     = \omega_{0}(1_{A}) \sum_{i=1}^{m} (-1)^{\bar{\epsilon}_{i+1} \bar{\epsilon}^{i} + \bar{\epsilon}_{n} + \deg \omega_{i}} \cl{\omega_{i+1} | \dots | \omega_{m} | \omega_{1} | \dots | \omega_{i-1}} \otimes \omega_{i},     
\end{multline*}
where $\omega_{i} \in A^{\#}$ are homogeneous, for $j = 0, \dots, m$, $\bar{\epsilon}_{i} = (\sum_{j=1}^{i-1} \deg \omega_{j}) + i - 1$ and $\bar{\epsilon}^{i} = (\sum_{j=i+1}^{m} \deg \omega_{j}) + m - i$. 
We shall denote by $B'$ the operator defined on $\Omega^{+}(A^{\#}) \otimes A^{\#}$ given by left multiplication by $ \varepsilon$. 
Our original claim is now equivalent to prove that 
\[     (\mathrm{id}_{\Omega^{+}(A^{\#})} \otimes_{\Omega^{+}(A^{\#})^{e}} \gamma^{A^{\#}}) \circ B_{\Omega^{+}(A^{\#})} = B' \circ (\mathrm{id}_{\Omega^{+}(A^{\#})} \otimes_{\Omega^{+}(A^{\#})^{e}} \gamma^{A^{\#}}),     \]
which is straightforward. 

We may summarize the previous comments in the following result. 
\begin{theorem}
\label{theorem:koszuldual2}
Let $A$ be an augmented dg algebra which is assumed to be Adams connected. 
The quasi-isomorphism of augmented dg algebras 
\[     \mathcal{H}om^{\tau_{E(A)}} (B^{+}(E(A)),E(A)) \rightarrow \mathcal{H}om^{\tau_{A}} (B^{+}(A),A)     \]
given in Theorem \ref{theorem:koszuldual} yields an isomorphism of augmented graded algebras 
\[     HH^{\bullet}(E(A)) \rightarrow HH^{\bullet}(A),     \] 
which is compatible with Gerstenhaber brackets (by adapating ideas from \cite{FMT}), so an isomorphism of Gerstenhaber algebras. 
On the other hand, the quasi-isomorphism of dg $\mathcal{H}om^{\tau_{E(A)}} (B^{+}(E(A)),E(A))$-bimodules 
\[     (A \otimes_{\tau_{A}} B^{+}(A))^{\#} \rightarrow E(A) \otimes_{\tau_{E(A)}} B^{+}(E(A)),        \]
also stated in that proposition, gives an isomorphism $HH_{\bullet}(A)^{\#} \rightarrow HH_{\bullet}(E(A))$ of graded bimodules over 
$HH^{\bullet}(E(A))$ and of graded Lie modules , where the domain has structure of bimodule given by the isomorphism $HH^{\bullet}(E(A)) \rightarrow HH^{\bullet}(A)$. 
The previous isomorphism between the homology groups is also compatible with the left actions of the graded Lie algebra $HH^{\bullet}(E(A))[1] \rightarrow HH^{\bullet}(A)[1]$. 
\end{theorem}

\begin{remark}
\label{rem:tt}
The previous result means that the Tamarkin-Tsygan precalculus of $A$ and of $E(A)$ are dual. 
However, since we have further proved that the previous isomorphisms are compatible with the Connes' maps, we may conclude that $A$ and $E(A)$ have dual Tamarkin-Tsygan calculus 
(following the notation of \cite{TT}, Def. 3.2.1 and 3.2.2).
\end{remark}

\section{{\texorpdfstring{Koszul duality and $A_{\infty}$-(co)algebras}{Koszul duality and A-infinity-(co)algebras}}}
\label{sec:ainf}

\subsection{{\texorpdfstring{Generalities on $A_{\infty}$-(co)algebras}{Generalities on A-infinity-(co)algebras}}}
\label{subsec:gena}

For the following definitions we refer to \cite{Prou}, Chapitre 3, Section 3.1 (or also \cite{LH}, D\'ef. 1.2.1.1, 1.2.1.8, using the obvious equivalences 
between non(co)unitary objects and (co)augmented ones), even though we do not follow the same sign conventions and they do not consider any Adams grading 
(see for instance \cite{LPWZ09} for several uses of Adams grading in $A_{\infty}$-algebra theory). 
We first recall that an \emph{augmented $A_{\infty}$-algebra} structure on a cohomological graded vector space $A$ is the following data: 
\begin{itemize}
\item[(i)] a collection of maps $m_{i} : A^{\otimes i} \rightarrow A$ for $i \in \NN$ of cohomological degree $2-i$ 
and Adams degree zero satisfying the \emph{Stasheff identities} given by   
\begin{equation}
\label{eq:ainftyalgebra}
   \sum_{(r,s,t) \in \mathcal{I}_{n}} (-1)^{r + s t}  m_{r + 1 + t} \circ (\mathrm{id}_{A}^{\otimes r} \otimes m_{s} \otimes \mathrm{id}_{A}^{\otimes t}) = 0,
\end{equation} 
for $n \in \NN$, where $\mathcal{I}_{n} = \{ (r,s,t) \in \NN_{0} \times \NN \times \NN_{0} : r + s + t = n \}$. 
We shall denote the (sum of) morphism(s) on the left by $\mathrm{SI}^{m_{\bullet}}(n)$. 
\item[(ii)] a map $\eta_{A} : k \rightarrow A$ of complete degree $0_{G}$ such that 
\[     m_{i} \circ (\mathrm{id}_{A}^{\otimes r} \otimes \eta_{A} \otimes \mathrm{id}_{A}^{\otimes t})     \]
vanishes for all $i \neq 2$ and all $r, t \geq 0$ such that $r+1+t = i$, and 
\[     m_{2} \circ (\mathrm{id}_{A} \otimes \eta_{A}) = \mathrm{id}_{A} = m_{2} \circ (\eta_{A} \otimes \mathrm{id}_{A}).    \]
\item[(iii)] a map $\epsilon_{A} : A \rightarrow k$ of complete degree $0_{G}$ such that $\epsilon_{A} \circ \eta_{A} = \mathrm{id}_{k}$, $\epsilon \circ m_{2} = \epsilon_{A}^{\otimes 2}$, and 
$\epsilon_{A} \circ m_{i} =0$, for all $i \in \NN \setminus \{ 2 \}$. 
\end{itemize}
It is further called \emph{minimal} if $m_{1}$ vanishes. 
If we do not assume the items (ii) and (iii) in the definition, then $A$ is called an \emph{$A_{\infty}$-algebra}. 

We recall that a family of linear maps $\{f_{i} : C \rightarrow C_{i}\}_{i \in \NN}$, where $C$ and $C_{i}$, for $i \in \NN$, are vector spaces, is called \emph{locally finite} if, for all $c \in C$, 
there exists a finite subset $S \subseteq \NN$, which depends on $c$, such that $f_{i}(c)$ vanishes for all $i \in \NN \setminus S$. 
An \emph{coaugmented $A_{\infty}$-coalgebra} structure on a homological graded vector space $C$ 
is the following data: 
\begin{itemize}
\item[(i)] a locally finite collection of maps $\Delta_{i} : C \rightarrow C^{\otimes i}$ for $i \in \NN$ of homological degree $i-2$ and Adams degree zero satisfying the following identities  
\begin{equation}
\label{eq:ainftycoalgebra}
   \sum_{(r,s,t) \in \mathcal{I}_{n}} (-1)^{r s + t}  (\mathrm{id}_{C}^{\otimes r} \otimes \Delta_{s} \otimes \mathrm{id}_{C}^{\otimes t}) \circ \Delta_{r + 1 + t} = 0,
\end{equation}
for $n \in \NN$. 
\item[(ii)] a map $\epsilon_{C} : C \rightarrow k$ of complete degree $0_{G}$ such that 
\[     (\mathrm{id}_{C}^{\otimes r} \otimes \epsilon_{C} \otimes \mathrm{id}_{C}^{\otimes t}) \circ \Delta_{i}     \]
vanishes for all $i \neq 2$ and all $r, t \geq 0$ such that $r+1+t = i$, and 
\[     (\mathrm{id}_{C} \otimes \epsilon) \circ \Delta_{2} = \mathrm{id}_{C} = (\epsilon_{C} \otimes \mathrm{id}_{C}) \circ \Delta_{2}.    \]
\item[(iii)] a map $\eta_{C} : k \rightarrow C$ of complete degree $0_{G}$ such that $\epsilon_{C} \circ \eta_{C} = \mathrm{id}_{k}$, $\Delta_{2} \circ \eta_{C}(1_{k}) = \eta_{C}(1_{k})^{\otimes 2}$, 
and $\Delta_{i} \circ \eta_{C}(1_{k}) = 0$, for all $i \in \NN \setminus \{ 2 \}$. 
\end{itemize}
We shall usually denote $\eta_{C}(1_{k})$ by $1_{C}$.  
An Adams graded coaugmented $A_{\infty}$-coalgebra $C$ is called \emph{minimal} if $\Delta_{1} = 0$. 
Note that the condition that the family $\{ \Delta_{n} \}_{n \in \NN}$ is locally finite follows from the other data if we further suppose that $\Ker(\epsilon)$ is positively graded for the Adams degree. 
Again, an \emph{$A_{\infty}$-coalgebra} is defined as the graded $k$-module $C$ provided with the maps $\{ \Delta_{i} \}_{i \in \NN}$ satisfying the identities of the first item. 

Note that an $A_{\infty}$-algebra $A$ is \emph{a fortiori} a dg k-module where the differential is given by $m_{1}$. 
Analogously, a $A_{\infty}$-algebra $A$ is also a dg k-module for the differential $\Delta_{1}$. 
Moreover, an augmented dg algebra structure on $A$ is tantamount to an augmented $A_{\infty}$-algebra structure with vanishing \emph{higher multiplications} $m_{n}$ for $n \geq 3$, 
where the differential is $m_{1}$ and the multiplication is $m_{2}$. 
In the same manner, a coaugmented dg coalgebra structure on $C$ is equivalent to a coaugmented $A_{\infty}$-coalgebra structure with vanishing \emph{higher comultiplications} $\Delta_{n}$ for $n \geq 3$, 
where the differential is $\Delta_{1}$ and the coproduct is $\Delta_{2}$. 

As for the case of augmented dg algebras, given an augmented $A_{\infty}$-algebra $A$ there exists a coaugmented dg coalgebra $B^{+}(A)$, called the \emph{(reduced) bar construction of $A$}. 
Its underlying graded coalgebra structure is given by the tensor coalgebra $\oplus_{i \in \NN_{0}} I_{A}[1]^{\otimes i}$, where $I_{A} = \operatorname{Ker}(\epsilon_{A})$. 
As before, if $n \in \NN$ we will typically denote an element $s(\bar{a}_{1}) \otimes \dots \otimes s(\bar{a}_{n}) \in I_{A}[1]^{\otimes n}$ in the form $[a_{1} | \dots | a_{n}]$, 
where $a_{1}, \dots, a_{n} \in A$, $\bar{a} \in A/k \simeq I_{A}$ denotes the canonical projection of an element $a \in A$, and $s : I_{A} \rightarrow I_{A}[1]$ is the canonical morphism of degree $- 1$ 
recalled in the third paragraph of Subsection \ref{subsec:int}.  
The coproduct is thus given by the usual deconcatenation 
\[     \Delta ([a_{1}|\dots |a_{n}]) = \sum_{i=0}^{n} [a_{1}|\dots | a_{i}] \otimes [a_{i+1} | \dots | a_{n}],     \]
where we set $[a_{i} | \dots | a_{j}] = 1_{B^{+}(A)}$ if $i > j$, for $1_{B^{+}(A)}$ the image of $1_{k}$ under the canonical inclusion $k = I_{A}[1]^{\otimes 0} \subseteq B^{+}(A)$, 
which may be also denoted by $[]$. 
The counit is defined as the canonical projection $B^{+}(A) \rightarrow I_{A}[1]^{\otimes 0} = k$, and the coaugmentation is given by the obvious inclusion $k = I_{A}[1]^{\otimes 0} \subseteq B^{+}(A)$. 
We recall that since $B^{+}(A)$ is a coaugmented tensor graded coalgebra, it is cocomplete, and its differential $B$ is defined as follows. 
It is the unique coderivation whose image lies inside the augmentation kernel $\operatorname{Ker}(\epsilon_{B^{+}(A)})$ of $B^{+}(A)$, so $B$ is thus uniquely determined by $\pi_{1} \circ B$, where $\pi_{1} : B^{+}(A) \rightarrow I_{A}[1]$ is the canonical projection (see \cite{LH}, Lemme 1.1.2.2, Sections 2.1.1 and 2.1.2, and Notation 2.2.1.4), such that this composition map is given by
the sum $b = \sum_{i \in \NN} b_{i}$, where $b_{i} : I_{A}[1]^{\otimes i} \rightarrow I_{A}[1]$ is defined as $b_{i} = - s_{I_{A}} \circ m_{i} \circ (s_{I_{A}}^{\otimes i})^{-1}$. 
In fact, equation \eqref{eq:ainftyalgebra} is precisely the condition for this coderivation to be a differential. 
Our convention for the bar construction clearly coincides with the one given for augmented dg algebras in the case the higher multiplications vanish, 
but it differs from others in the literature (\textit{e.g.} those in the thesis \cite{LH} of K. Lef\`evre-Hasegawa, Ch. 1 and 2). 

Dually, given a coaugmented $A_{\infty}$-coalgebra $C$ there exists an augmented dg algebra $\Omega^{+}(C)$, called (reduced) bar construction of $C$. 
Its underlying graded coalgebra structure is given by the tensor algebra $\oplus_{i \in \NN_{0}} J_{C}[-1]^{\otimes i}$, where $J_{C} = \operatorname{Coker}(\eta_{C})$. 
As before, if $n \in \NN$ we will typically denote an element $s^{-1}(\bar{c}_{1}) \otimes \dots \otimes s(\bar{c}_{n}) \in J_{C}[-1]^{\otimes n}$ in the form $\cl{c_{1} | \dots | c_{n}}$, 
where $c_{1}, \dots, c_{n} \in C$, $\bar{c} \in J_{C}$ denotes the canonical projection of an element $c \in C$, and $s : J_{C}[-1] \rightarrow J_{C}$ is the canonical morphism of degree $- 1$ 
recalled in the third paragraph of Subsection \ref{subsec:int}. 
The unit is given by the obvious inclusion $k = J_{C}[-1]^{\otimes 0} \subseteq \Omega^{+}(C)$, and we denote the image of $1_{k}$ under the previous map either by $1_{\Omega^{+}(C)}$ or by $\cl{}$. 
The augmentation is defined as the canonical projection $\Omega^{+}(C) \rightarrow J_{C}[-1]^{\otimes 0} = k$. 
The differential $D$ of $\Omega^{+}(C)$ is defined as follows. 
Since $\Omega^{+}(C)$ is graded free algebra, it is the unique derivation $D$ whose composition $d$ with the canonical injection $J_{C}[-1] \rightarrow \Omega^{+}(C)$, 
where we define $d = \sum_{i \in \NN} d_{i}$ for $d_{i} : J_{C}[-1] \rightarrow J_{C}[-1]^{\otimes i}$ given by $d_{i} = (-1)^{i} (s_{J_{C}[-1]}^{\otimes i})^{-1} \circ \Delta_{i} \circ s_{J_{C}[-1]}$. 
Note again that \eqref{eq:ainftycoalgebra} is exactly the condition for this derivation to be a differential. 
As before, our definition for the cobar construction coincides with the one given for coaugmented dg coalgebras in the case the higher comultiplications $\Delta_{i}$ for $i \geq 3$ vanish, 
but it differs from others in the literature (\textit{e.g.} those in the thesis \cite{LH} of K. Lef\`evre-Hasegawa, Ch. 1 and 2). 

We will also be particularly interested in the case that (co)augmented $A_{\infty}$-(co)algebras are \emph{Adams connected}, in the sense introduced in \cite{LPWZ08}, Def. 2.1, 
\textit{i.e.} a (co)augmented $A_{\infty}$-(co)algebra $A$ (resp., $C$) where the grading group $G$ is $\ZZ \times \ZZ$ such that its augmentation kernel $I_{A}$ (resp., $J_{C}$) satisfies that $\oplus_{n \in \ZZ} I_{A}^{(n,m)}$ (resp., $\oplus_{n \in \ZZ} J_{C}^{(n,m)}$) is finite dimensional, for all $m \in G' = \ZZ$, and 
either $I_{A}^{(n,m)}$ (resp., $J_{C}^{(n,m)}$) vanishes for all $n \in \ZZ$ and all $m \geq 0$, or $I_{A}^{(n,m)}$ (resp., $J_{C}^{(n,m)}$) vanishes for all $n \in \ZZ$ and all $m \leq 0$.
By the previously cited article, an Adams connected augmented $A_{\infty}$-algebra is also locally finite dimensional, and its Koszul dual $E(A) = B^{+}(A)^{\#}$ is also locally finite dimensional and 
also Adams connected (see \cite{LPWZ08}, Lemma 2.2). 

A \emph{morphism of augmented $A_{\infty}$-algebras} $f_{\bullet} : A \rightarrow B$ between two augmented $A_{\infty}$-algebras $A$ and $B$ is a collection of morphisms of the underlying graded $k$-modules 
$f_{i} : A^{\otimes n} \rightarrow B$ of complete degree $(1-i,0_{G'})$ for $i \in \NN$ such that 
\begin{equation}
\label{eq:ainftyalgebramor}
   \sum_{(r,s,t) \in \mathcal{I}_{n}} (-1)^{r + s t}  f_{r + 1 + t} \circ (\mathrm{id}_{A}^{\otimes r} \otimes m_{s}^{A} \otimes \mathrm{id}_{A}^{\otimes t}) 
   = \sum_{q \in \NN} \sum_{\bar{i} \in \NN^{q, n}} (-1)^{w} m_{q}^{B} \circ (f_{i_{1}} \otimes \dots \otimes f_{i_{q}}),
\end{equation} 
where $w = \sum_{j=1}^{q} (q-j) (i_{j} - 1)$ and $\NN^{q,n}$ is the subset of $\NN^{q}$ of elements $\bar{i} = (i_{1},\dots,i_{q})$ such that $|\bar{i}| = i_{1} + \dots + i_{q} = n$. 
We also assume that $f_{1}(1_{A}) = 1_{B}$, for all $i \geq 2$ we have that $f_{i}(a_{1}, \dots, a_{i})$ vanishes if there exists $j \in \{1, \dots, i \}$ such that $a_{j} = 1_{A}$, 
and that $\epsilon_{B} \circ f_{1} = \epsilon_{A}$ and $\epsilon_{B} \circ f_{i}$ vanishes for $i \geq 2$. 
If we do not suppose this last collection of extra-assumptions the family of maps $\{ f_{i} \}_{i \in \NN}$ is only called a \emph{morphism of $A_{\infty}$-algebras}. 
We shall denote the (sum of) morphism(s) of the left (resp., right) member of \eqref{eq:ainftyalgebramor} by $\mathrm{MI}^{m_{\bullet}}(n)_{l}$ (resp., $\mathrm{MI}^{m_{\bullet}}(n)_{r}$). 
Notice that $f_{1}$ is a morphism of dg $k$-modules for the underlying structures on $A$ and $B$. 
We say that a morphism $f_{\bullet}$ is \emph{strict} if $f_{i}$ vanishes for $i \geq 2$. 

Dually, a \emph{morphism of coaugmented $A_{\infty}$-coalgebras} $f_{\bullet} : C \rightarrow D$ between two coaugmented $A_{\infty}$-coalgebras $C$ and $D$ is a locally finite collection of morphisms of the underlying graded $k$-modules $f_{i} : C \rightarrow D^{\otimes i}$ of homological degree $i-1$ and Adams degree zero for $i \in \NN$ such that 
\begin{equation}
\label{eq:ainftycoalgebramor}
   \sum_{(r,s,t) \in \mathcal{I}_{n}} (-1)^{r s + t}  (\mathrm{id}_{D}^{\otimes r} \otimes \Delta_{s}^{D} \otimes \mathrm{id}_{D}^{\otimes t}) \circ f_{r + 1 + t} 
   = \sum_{q \in \NN} \sum_{\bar{i} \in \NN^{q, n}} (-1)^{w'} (f_{i_{1}} \otimes \dots \otimes f_{i_{q}}) \circ \Delta_{q}^{C},
\end{equation} 
where $w' = \sum_{j=1}^{q} (j-1) (i_{j} + 1)$. 
We also suppose that $\epsilon_{D} \circ f_{1} = \epsilon_{D}$, for all $i \geq 2$ and $j \in \{1, \dots, i \}$ we have that $(\mathrm{id}_{D}^{\otimes (j-1)}  \otimes \epsilon_{D} \otimes \mathrm{id}_{D}^{\otimes (i-j)}) \circ f_{i}$ vanishes, and that $f_{1} \circ \eta_{C} = \eta_{D}$ and 
$f_{i} \circ \eta_{C}$ vanishes for $i \geq 2$.
If we do not suppose this last collection of extra-assumptions the family of maps $\{ f_{i} \}_{i \in \NN}$ is only called a \emph{morphism of $A_{\infty}$-coalgebras}. 
Notice that $f_{1}$ is also a morphism of dg $k$-modules for the underlying structures on $C$ and $D$.
In this case we also say that a morphism $f_{\bullet}$ is \emph{strict} if $f_{i}$ vanishes for $i \geq 2$. 

Given $f_{\bullet} : A \rightarrow B$ a morphism of augmented $A_{\infty}$-algebras, it induces a morphism of coaugmented dg coalgebras $B^{+}(f_{\bullet}) : B^{+}(A) \rightarrow B^{+}(B)$ 
between the bar constructions as follows. 
We first note that the unitarity condition on $f_{\bullet}$ implies that it is completely detemined by the induced collection of morphism $I_{A}^{\otimes i} \rightarrow I_{B}$, which we are going to denote also by $f_{i}$,  
for $i \in \NN$. 
The morphism $B^{+}(f_{\bullet})$ being of coaugmented graded coalgebras implies that it sends $1_{B^{+}(A)}$ to $1_{B^{+}(B)}$, 
and the coaugmentation cokernel of $B^{+}(A)$ to the coaugmentation cokernel of $B^{+}(B)$. 
Moreover, since $B^{+}(B)$ is a cocomplete graded coalgebra, such a morphism of graded coalgebras is completely determined by the composition $\pi_{1}^{B} \circ B^{+}(f_{\bullet}) : B^{+}(A) \rightarrow I_{B}[1]$, 
which vanishes on $1_{B^{+}(A)}$. 
The latter composition is thus given by a sum $\sum_{i \in \NN} F_{i}$, where $F_{i} : I_{A}[1]^{\otimes i} \rightarrow I_{B}[1]$, which we define to be $F_{i} = s_{I_{B}} \circ f_{i} \circ (s_{I_{A}}^{\otimes i})^{-1}$, for 
$i \in \NN$. 
In fact, \eqref{eq:ainftyalgebramor} is precisely the condition for this morphism to commute with the differentials. 

Dually, given $f_{\bullet} : C \rightarrow D$ a morphism of coaugmented $A_{\infty}$-coalgebras, it induces a morphism of augmented dg algebras $\Omega^{+}(f_{\bullet}) : \Omega^{+}(C) \rightarrow \Omega^{+}(D)$ 
between the cobar constructions as follows. 
We first note that the counitarity condition on $f_{\bullet}$ implies that it is completely detemined by the induced collection of morphism $J_{C} \rightarrow J_{D}^{\otimes i}$, 
which we are going to denote also by $f_{i}$, for $i \in \NN$. 
We suppose that it sends $1_{B^{+}(A)}$ to $1_{B^{+}(B)}$, and the augmentation kernel of $\Omega^{+}(C)$ to the augmentation kernel of $\Omega^{+}(D)$. 
Moreover, since $\Omega^{+}(C)$ is a free graded algebra, such a morphism $B^{+}(f_{\bullet})$ of graded algebras is completely determined by the composition of the canonical inclusion $J_{C}[-1] \rightarrow \Omega^{+}(C)$ with it. 
Let us denote this latter composition by $F$. 
Hence, $F = \sum_{i \in \NN} F_{i}$, where $F_{i} : J_{C}[-1] \rightarrow J_{D}[-1]^{\otimes i}$, which we define to be $F_{i} = (-1)^{i+1} (s_{J_{D}[-1]}^{\otimes i})^{-1} \circ f_{i} \circ s_{J_{C}[-1]}$, for 
$i \in \NN$. 
As expected, \eqref{eq:ainftycoalgebramor} is precisely the condition for this morphism to commute with the differentials. 
We remark that our definition of $B^{+}(f_{\bullet})$ ($\Omega^{+}(f_{\bullet})$) agrees with the corresponding one for (co)augmented dg (co)algebras in that case if the morphism $f_{\bullet}$ 
is further assumed to be strict. 

A morphism of (co)augmented $A_{\infty}$-(co)algebras is called a \emph{quasi-isomorphism} if the map $f_{1}$ is so. 
The same definition may be stated for the nonaugmented case. 
We further say that a morphism $f_{\bullet}$ of coaugmented $A_{\infty}$-coalgebras is a \emph{weak equivalence} provided its cobar $\Omega^{+}(f_{\bullet})$ is a quasi-isomorphism of augmented dg algebras. 
Note that a morphism of augmented $A_{\infty}$-algebras is a quasi-isomorphism if and only if $B^{+}(f_{\bullet})$ is also (see \cite{Prou}, Thm. 3.25). 
We refer to \cite{Prou}, Sections 3.2 and 3.3 (D\'ef. 3.3, 3.4, and 3.11), or \cite{LH}, Sections 1.2 and 1.3 for more details on these definitions (though we follow a different sign convention), and 
we remark that these morphisms are supposed to preserve the Adams degree (\textit{cf}. \cite{LPWZ09}, Section 2). 

Notice that if $C$ is a coaugmented dg coalgebra and $A$ is an augmented $A_{\infty}$-algebra, the dg $k$-module $\mathcal{H}om(C,A)$ has in fact a structure of augmented $A_{\infty}$-algebra where
$m_{1}$ is given by the usual differential $d_{\mathcal{H}om(C,A)}(\phi) = m_{1}^{A} \circ \phi - (-1)^{\deg \phi} \phi \circ d_{C}$. 
Indeed, if we further define 
\[       m_{n} (\phi_{1} \otimes \dots \otimes \phi_{n}) = m_{n}^{A} \circ (\phi_{1} \otimes \dots \otimes \phi_{n}) \circ \Delta_{C}^{(n)},     \]
for $n \geq 2$, $1_{\mathcal{H}om(C,A)} = \eta_{A} \circ \epsilon_{C}$ and $\epsilon_{\mathcal{H}om(C,A)} (\phi) = \epsilon_{A} \circ \phi \circ \eta_{C}(1_{k})$, 
it is easily verified that they provide the structure of augmented $A_{\infty}$-algebra on $\mathcal{H}om(C,A)$. 
Furthermore, if $f_{\bullet} : A \rightarrow B$ is a morphism of augmented $A_{\infty}$-algebras, then the collection of morphisms 
\[      f_{n}^{*} : \mathcal{H}om(C,A)^{\otimes n} \rightarrow \mathcal{H}om(C,B),       \]
for $n \in \NN$, of graded $k$-modules of complete degree $(1,0_{G'})$ given by $f_{1}^{*}(\phi) = f_{1} \circ \phi$, and 
\[      f_{n}^{*}(\phi_{1} \otimes \dots \otimes \phi_{n}) = f_{n} \circ (\phi_{1} \otimes \dots \otimes \phi_{n}) \circ \Delta_{C}^{(n)},       \]
for $n \geq 2$, is a morphism of augmented $A_{\infty}$-algebras. 

Dually, if $C$ is a coaugmented $A_{\infty}$-coalgebra and $A$ is an augmented dg algebra, the dg $k$-module $\mathcal{H}om(C,A)$ has in fact a structure of augmented $A_{\infty}$-algebra where
$m_{1}$ is also given by the usual differential $d_{\mathcal{H}om(C,A)}(\phi) = d_{A} \circ \phi - (-1)^{\deg \phi} \phi \circ \Delta_{1}^{C}$. 
In this case the rest of the structure is given by 
\[       m_{n} (\phi_{1} \otimes \dots \otimes \phi_{n}) = (-1)^{n (\deg \phi_{1} + \dots + \deg \phi_{n} + 1)} \mu_{A}^{(n)} \circ (\phi_{1} \otimes \dots \otimes \phi_{n}) \circ \Delta_{n}^{C},     \]
for $n \geq 2$, $1_{\mathcal{H}om(C,A)} = \eta_{A} \circ \epsilon_{C}$ and $\epsilon_{\mathcal{H}om(C,A)} (\phi) = \epsilon_{A} \circ \phi \circ \eta_{C}(1_{k})$. 
Moreover, if $f_{\bullet} : C \rightarrow D$ is a morphism of coaugmented $A_{\infty}$-coalgebras, then the collection of morphisms 
\begin{equation}     
\label{eq:cambiocinfcoal}
 (f_{n})_{*} : \mathcal{H}om(D,A)^{\otimes n} \rightarrow \mathcal{H}om(C,A),       
\end{equation}
for $n \in \NN$, of graded $k$-modules of complete degree $(1,0_{G'})$ given by $(f_{1})_{*}(\phi) = \phi \circ f_{1}$, and 
\[      (f_{n})_{*}(\phi_{1} \otimes \dots \otimes \phi_{n}) = (-1)^{(n-1)(\deg \phi_{1} + \dots + \phi_{n} + 1)} \mu_{A}^{(n)} \circ (\phi_{1} \otimes \dots \otimes \phi_{n}) \circ f_{n},       \]
for $n \geq 2$, is a morphism of augmented $A_{\infty}$-algebras.

If $f_{\bullet} : A \rightarrow A'$ and $g_{\bullet} : A' \rightarrow B$ are morphisms of augmented $A_{\infty}$-algebras, the composition $g_{\bullet} \circ f_{\bullet}$ is the morphism of augmented $A_{\infty}$-algebras 
given by the collection of maps $\{ (g \circ f)_{n} : A^{\otimes n} \rightarrow B \}_{n \in \NN}$ defined as  
\begin{equation}
\label{eq:ainftyalgebramorcomp}
   (g \circ f)_{n} = \sum_{q \in \NN} \sum_{\bar{i} \in \NN^{q, n}} (-1)^{w} g_{q} \circ (f_{i_{1}} \otimes \dots \otimes f_{i_{q}}),
\end{equation} 
where $w = \sum_{j=1}^{q} (q-j) (i_{j} - 1)$.
Dually, if $f_{\bullet} : C \rightarrow C'$ and $g_{\bullet} : C' \rightarrow D$ are morphisms of coaugmented $A_{\infty}$-coalgebras, the composition $g_{\bullet} \circ f_{\bullet}$ is the morphism of coaugmented $A_{\infty}$-coalgebras given by the collection of maps $\{ (g \circ f)_{n} : C \rightarrow D^{\otimes n} \}_{n \in \NN}$ of the form  
\begin{equation}
\label{eq:ainftycoalgebramorcomp}
   (g \circ f)_{n} = \sum_{q \in \NN} \sum_{\bar{i} \in \NN^{q, n}} (-1)^{w'} (g_{i_{1}} \otimes \dots \otimes g_{i_{q}}) \circ f_{q},
\end{equation} 
where $w' = \sum_{j=1}^{q} (j-1) (i_{j} + 1)$. 

We remark that the previous construction defines an augmented $A_{\infty}$-algebra structure on the graded dual $C^{\#}$ of $C$. 
If $C$ is Adams connected, we see that $\Omega^{+}(C)^{\#}$ is isomorphic to $B^{+}(C^{\#})$ (using the isomorphism $j^{C}$ defined in Subsubsection \ref{subsubsec:barcobar}). 
In this case we get that a quasi-isomorphism of Adams connected coaugmented $A_{\infty}$-coalgebras is a weak equivalence (the converse is always true), 
for a quasi-isomorphism $C \rightarrow D$ induces a quasi-isomorphism of augmented $A_{\infty}$-algebras $D^{\#} \rightarrow C^{\#}$, which induces a quasi-isomorphisms between the bar constructions 
$B^{+}(D^{\#}) \rightarrow B^{+}(C^{\#})$, and by the previously recalled isomorphism we get a quasi-isomorphism $\Omega^{+}(D)^{\#} \rightarrow \Omega^{+}(C)^{\#}$ of coaugmented dg coalgebras. 
Taking duals again we obtain a quasi-isomorphism $\Omega^{+}(C) \rightarrow \Omega^{+}(D)$ of augmented dg algebras, so a weak equivalence $C \rightarrow D$. 

For the following definitions we refer to \cite{LH}, Ch. 2, Section 5. 
Given an augmented $A_{\infty}$-algebra $A$, an \emph{$A_{\infty}$-bimodule over $A$} is a graded $k$-module $M$ provided with morphisms 
$m_{p,q}^{M} : A^{\otimes p} \otimes M \otimes A^{\otimes q} \rightarrow M$ of complete degree $(1-(p+q),0_{G'})$, for each $p, q \in \NN_{0}$ satisfying the following identity on morphisms 
from $A^{\otimes n'} \otimes M \otimes A^{\otimes n''}$ to $M$ given by 
\begin{equation}
\label{eq:ainftyalgebramod}
   \sum_{(r,s,t) \in \mathcal{I}_{n'+n''+1}} (-1)^{r + s t}  \tilde{m}_{r,t}^{M} \circ (\mathrm{id}^{\otimes r} \otimes \tilde{m}_{s} \otimes \mathrm{id}^{\otimes t}) = 0,
\end{equation} 
for all $n', n'' \in \NN$, where we recall that $\mathcal{I}_{n} = \{ (r,s,t) \in \NN_{0} \times \NN \times \NN_{0} : r + s + t = n \}$, and where $\tilde{m}_{s}$ is interpreted as the corresponding multiplication map $m_{s}$
of $A$ if either $r+s \leq n'$ or $s+t \leq n''$, or $\tilde{m}_{s}$ is understood as $m_{n'-r,n''-t}^{M}$ else. 
In the first case, $\tilde{m}_{r,t}^{M}$ is $m_{n'-s+1,n''}^{M}$ if $r+s \leq n'$ or $m_{n',n''-s+1}^{M}$ if $s+t \leq n''$, and it is $m_{r,t}^{M}$ else. 
We have omitted the subindex ($A$ or $M$) on the identity morphisms for it depends on the indices $r, s, t$, and it is clearly deduced from the previous explanation. 
We also assume that $M$ satisfies that $m_{p,q}^{M} \circ (\mathrm{id}^{\otimes r} \otimes \eta_{A} \otimes \mathrm{id}^{\otimes t})$ vanishes for $r \neq p$ and $(p,q) \notin \{ (0,1), (1,0) \}$, 
and that $m_{1,0}^{M} \circ (\eta_{A} \otimes \mathrm{id}_{M}) = \mathrm{id}_{M} = m_{0,1}^{M} \circ (\mathrm{id}_{M} \otimes \eta_{A})$. 
Note that an augmented $A_{\infty}$-algebra is also an $A_{\infty}$-bimodule for the structure maps $m_{p,q} = m_{p+q+1}$, where $p, q \in \NN_{0}$. 
There are also obvious notions of left and right $A_{\infty}$-modules but we will not need them. 

Given two $A_{\infty}$-bimodules $M$ and $N$, a \emph{morphism $f_{\bullet, \bullet}$ of $A_{\infty}$-bimodules from $M$ to $N$} is a collection of morphisms of graded $k$-modules 
$f_{p,q} : A^{\otimes p} \otimes M \otimes A^{\otimes q} \rightarrow N$ for $p, q \in \NN_{0}$ of complete degree $(-p-q,0_{G'})$ satisfying the following identity on the space of morphisms 
from $A^{\otimes n'} \otimes M \otimes A^{\otimes n''}$ to $N$ given by 
\begin{equation}
\label{eq:ainftyalgebramormod}
\begin{split}
   \sum_{(r,s,t) \in \mathcal{I}_{n'+n''+1}} (-1)^{r + s t}  &f_{r',t'} \circ (\mathrm{id}^{\otimes r} \otimes \tilde{m}_{s} \otimes \mathrm{id}^{\otimes t}) 
\\
   &= \sum_{(a,k,l,b) \in \NN_{0,n',n''}} (-1)^{b(-k-l)} m_{a,b}^{N} \circ (\mathrm{id}^{\otimes a}_{A} \otimes f_{k,l} \otimes \mathrm{id}^{\otimes b}_{A}),
\end{split}
\end{equation} 
where $\NN_{0,n',n''}$ is the subset of $\NN^{4}_{0}$ of elements $(a,k,l,b)$ such that $a + k = n'$ and $l + b = n''$, 
and where we should understand $\tilde{m}_{s}$ as $m_{s}^{A}$ if either $r+s \leq n'$ or $s+t \leq n''$, or as $m_{n'-r,n''-t}^{M}$ else. 
The indices $(r',t')$ are completely determined from the previous cases. 
We also suppose that $f_{\bullet,\bullet}$ satisfies that $f_{p,q} \circ (\mathrm{id}^{\otimes r} \otimes \eta_{A} \otimes \mathrm{id}^{\otimes t})$ vanishes for $r \neq p$ and $(p,q) \notin \{ (0,0) \}$. 
We say that it is \emph{strict} if $f_{p,q}$ vanishes for all $(p, q) \neq (0,0)$.
The \emph{composition} of two morphisms $f_{\bullet, \bullet} : M \rightarrow N$ and $g_{\bullet, \bullet} : N \rightarrow P$ is given by the family of maps 
\[     (g \circ f)_{p,q} = \sum_{(a,k,l,b) \in \NN_{0,p,q}} (-1)^{b (-k-l)} g_{a,b} \circ (\mathrm{id}^{\otimes a}_{A} \otimes f_{k,l} \otimes \mathrm{id}^{\otimes b}_{A}).     \]
If $f_{\bullet} : A \rightarrow B$ is a morphism of augmented $A_{\infty}$-algebras and $N$ is an $A_{\infty}$-bimodule over $B$ with structure maps $m_{\bullet,\bullet}$, then it can be easily regarded as an $A_{\infty}$-bimodule over $A$ via the maps $m'_{\bullet,\bullet}$ given by 
\begin{multline}
\label{eq:passage}
      m'_{p,q}
      \\
         = \sum_{r, s \in \NN_{0}} \sum_{(\bar{i},\bar{j}) \in \NN^{r,p} \times \NN^{s,q}} (-1)^{\varepsilon} m_{r,s} \circ (f_{i_{1}} \otimes \dots \otimes f_{i_{r}} \otimes \mathrm{id}_{N} \otimes f_{j_{1}} \otimes \dots \otimes f_{j_{s}}),     
\end{multline}
where we recall that $\NN^{m,n}$ is the subset of $\NN^{m}$ of elements $\bar{i} = (i_{1}, \dots, i_{m})$ such that 
$|\bar{i}| = i_{1} + \dots + i_{m} = n$, and $\varepsilon = \sum_{u=1}^{r} (r+s+1-u) (i_{u} - 1) + \sum_{u=1}^{s} (s-u) (j_{u} - 1)$. 

If $M$ is a dg $A$-bimodule over an augmented dg algebra $A$ and $C$ is a coaugmented $A_{\infty}$-coalgebra, 
then $M \otimes C$ is in fact an $A_{\infty}$-bimodule over $\mathcal{H}om(C,A)$ 
with the structure morphisms given by $m_{0,0}^{M} = d_{M} \otimes \mathrm{id}_{C} + \mathrm{id}_{M} \otimes \Delta_{1}^{C}$, and, for $p + q \geq 1$, 
\begin{multline*}
     m_{p,q}^{M \otimes C} (\phi_{1} \otimes \dots \otimes \phi_{p} \otimes (m \otimes c) \otimes \psi_{1} \otimes \dots \otimes \psi_{q}) 
     \\
     = (-1)^{\epsilon'} (\phi_{1}(c_{(q+2)}) \dots \phi_{p} (c_{(q+p+1)})) . m . (\psi_{1}(c_{(1)}) \dots \psi_{q}(c_{(q)})) \otimes c_{(q+1)},     
\end{multline*}
where $\Delta_{p+q+1}^{C}(c) = c_{(1)} \otimes \dots \otimes c_{(p+q+1)}$, and  
\begin{multline*}
     \epsilon' = \deg c (\sum_{j=1}^{q} \deg \psi_{j}) + \sum_{j=1}^{q-1} \deg c_{(j)} (\deg \psi_{j+1} + \dots + \deg \psi_{q}) 
      \\ 
     + \sum_{l=1}^{p} \deg c_{(q+1+l)} ((\sum_{j=1}^{q} \deg \psi_{j}) + \deg m + (\sum_{j=1}^{q+1} \deg c_{(j)}) + (\sum_{j=1}^{p-1} \deg \phi_{l+j})).     
\end{multline*}
If $M$ is only a left (resp., right) dg module over $A$, we may regard it as a dg $A$-bimodule by means of the augmentation $\epsilon_{A}$, \textit{i.e.} $a . m . a' = \epsilon_{A}(a') a . m$ 
(resp., $a . m . a' = \epsilon_{A}(a) m. a'$), so we may apply the previous construction.  
If $f : M \rightarrow N$ is a morphism of dg $A$-bimodules over an augmented dg algebra $A$ and $C$ is an coaugmented $A_{\infty}$-coalgebra, then the map $f \otimes \mathrm{id}_{C} : M \otimes C \rightarrow N \otimes C$ 
is a strict morphism of $A_{\infty}$-bimodules over $\mathcal{H}om(C,A)$. 
On the other hand, let $f_{\bullet} : C \rightarrow D$ be a morphism of coaugmented $A_{\infty}$-coalgebras and let $A$ be an augmented dg algebra. 
This induces a morphism of augmented $A_{\infty}$-algebras $(f_{\bullet})_{*} : \mathcal{H}om(D,A) \rightarrow \mathcal{H}om(C,A)$, as seen in \eqref{eq:cambiocinfcoal}. 
In particular, given any dg $A$-bimodule $M$, this allows to consider $M \otimes C$ as an $A_{\infty}$-bimodule over $\mathcal{H}om(D,A)$ by means of \eqref{eq:passage}. 
Then, the collection of morphisms 
\[     F_{p,q} : \mathcal{H}om(D,A)^{\otimes p} \otimes (M \otimes C) \otimes \mathcal{H}om(D,A)^{\otimes q} \rightarrow M \otimes D     \]
given by 
\begin{multline*}
     F_{p,q} (\phi_{1} \otimes \dots \otimes \phi_{p} \otimes (m \otimes c) \otimes \psi_{1} \otimes \dots \otimes \psi_{q}) 
     \\
     = (-1)^{\varepsilon'} (\phi_{1}(d_{(q+2)}) \dots \phi_{p} (d_{(q+p+1)})) . m . (\psi_{1}(d_{(1)}) \dots \psi_{q}(d_{(q)})) \otimes d_{(q+1)},     
\end{multline*}
where $f_{p+q+1}(c) = d_{(1)} \otimes \dots \otimes d_{(p+q+1)}$, and  
\begin{multline*}
     \varepsilon' = \deg c (\sum_{j=1}^{q} \deg \psi_{j}) + \sum_{j=1}^{q-1} \deg d_{(j)} (\deg \psi_{j+1} + \dots + \deg \psi_{q}) 
      \\ 
     + \sum_{l=1}^{p} \deg d_{(q+1+l)} ((\sum_{j=1}^{q} \deg \psi_{j}) + \deg m + (\sum_{j=1}^{q+1} \deg d_{(j)}) + (\sum_{j=1}^{p-1} \deg \phi_{l+j})),     
\end{multline*}
defines a morphism of $A_{\infty}$-bimodules over $\mathcal{H}om(D,A)$. 

\subsection{{\texorpdfstring{Twists of $A_{\infty}$-algebras}{Twists of A-infinity-algebras}}}
\label{subsec:twa}

We will now recall the twisting procedure for $A_{\infty}$-algebras given a Maurer-Cartan element, which was introduced in \cite{FOOO1}, Chapter 4 (see also \cite{Fuk}, and \cite{LH}, Ch. 6, Section 6.1). 
Let $A$ be an augmented $A_{\infty}$-algebra and $a \in A^{(1,0_{G'})}$.
Given $k \in \NN_{0}$ and $n \in \NN_{0}$, let us define 
\[       p_{k,n}^{a} : A^{\otimes n} \rightarrow A^{\otimes (k+n)}     \]
as follows. 
If $n \geq 1$,  
\[     p_{k,n}^{a} (x_{1} \otimes \dots \otimes x_{n}) = \sum_{\bar{k} \in \NN_{0}^{n+1,k}} (-1)^{w''} a^{\otimes k_{1}} \otimes x_{1} \otimes a^{\otimes k_{2}} \otimes x_{2} \otimes \dots \otimes a^{\otimes k_{n}} \otimes x_{n} a^{\otimes k_{n+1}},    \]
where $w'' = \sum_{j=2}^{n+1} k_{j} (\deg x_{1} + \dots + \deg x_{j-1} + j - 1)$. 
If $n = 0$, we set 
\[     p_{k,0}^{a} (1_{k}) = a^{\otimes k},     \]
where $a^{\otimes 0} = 1_{k} \in k$. 
We say that the element $a$ is \emph{admissible} if for each $n \in \NN_{0}$, the family of maps $\{ m_{n+k} \circ p_{k,n}^{a} \}_{k \in \NN_{0}}$ is locally finite, if we use the usual notation $m_{0} = \eta_{A}$. 
An admissible element $a \in A^{(1,0_{G'})}$ is said to satisfy the \emph{Maurer-Cartan equation} if the following sum 
\begin{equation}
\label{eq:mc1}
     \sum_{n \in \NN} (-1)^{\frac{n(n+1)}{2}} m_{n}(a, \dots, a)      
\end{equation} 
vanishes. 
Note that the locally finiteness hypothesis on $\{ m_{k} \circ p_{k,0}^{a} \}_{k \in \NN_{0}}$ implies that the previous sum has finite support. 
We remark that we do not assume the usual \emph{tensor-nilpotency} assumption on $a$, for we have replaced it by (weaker) finiteness assumptions.  
In this case, we can twist the augmented $A_{\infty}$-algebra structure on $A$ by defining 
\[     m_{n}^{a} = \sum_{k \in \NN_{0}} (-1)^{\frac{k(k+1)}{2} + k n} m_{n+k} \circ p_{k,n}^{a},     \] 
for $n \in \NN$. 
The admissibility assumption on $a \in A$ implies that the previous higher multiplication morphisms are well-defined. 
It is easy to verify that the have complete degree $(2-n,0_{G'})$ and that in fact 
\[     \mathrm{SI}^{m_{\bullet}^{a}}(n) = \sum_{k \in \NN_{0}} (-1)^{\frac{k(k+1)}{2} + k n} \mathrm{SI}^{m_{\bullet}}(k+n) \circ p_{k,n}^{a},     \]
so the maps $m_{n}^{a}$ together with $\epsilon_{A}$ indeed provide an augmented $A_{\infty}$-algebra structure on $A$, which we shall denote by $(A,m_{\bullet}^{a})$, 
and usually called the \emph{twisted augmented $A_{\infty}$-algebra of $(A,m_{\bullet})$ by $a$}. 
Note that these definitions coincide with the corresponding ones for augmented dg algebras, in the case that $m_{i}$ vanishes for $i \geq 3$.

Let $f_{\bullet} : A \rightarrow B$ be a morphism of augmented $A_{\infty}$-algebras and $a \in A^{(1,0_{G'})}$ an admissible element which satisfies the Maurer-Cartan equation. 
We say that $a$ is \emph{compatible} with $f_{\bullet}$ if the family $\{ f_{n+k} \circ p_{k,n}^{a} \}_{k \in \NN_{0}}$ is locally finite for every $n \in \NN$, where we denote $f_{0} = 0$, 
and the element of $B^{(1,0_{G'})}$ defined by the sum (of finite support, due to the local finiteness assumption on $\{ f_{k} \circ p_{k,0}^{a} \}_{k \in \NN_{0}}$) 
\begin{equation}
\label{eq:mc2}
     \sum_{n \in \NN} (-1)^{\frac{n(n+1)}{2} + 1} f_{n}(a, \dots, a),      
\end{equation} 
which we are going to denote by $b$ and call \emph{the image of $a$ under $f_{\bullet}$}, is admissible. 
Moreover, one also sees that $b$ satisfies the Maurer-Cartan equation. 
Indeed, it is direct (but rather lengthy) to prove that  
\[     \sum_{n \in \NN} (-1)^{\frac{n(n+1)}{2}} m_{n}^{B}(b, \dots, b)      \]
coincides with 
\[      \sum_{n \in \NN} (-1)^{\frac{n(n+1)}{2}} \mathrm{MI}(n)^{m_{\bullet}^{A}}_{r}(a, \dots, a),     \]
and that 
\[      \sum_{n \in \NN} (-1)^{\frac{n(n+1)}{2}} \mathrm{MI}(n)^{m_{\bullet}^{A}}_{l}(a, \dots, a)     \]
vanishes, by the Maurer-Cartan equation of $a$. 
Note that the compatibility assumption tells us that the two previous sums have finite support. 
In this case, one may also twist the morphism $f_{\bullet} : A \rightarrow B$ to a new morphism $f_{\bullet}^{a}$ of augmented $A_{\infty}$-algebras from $(A,m_{\bullet}^{A,a})$ to $(B, m_{\bullet}^{B,b})$ 
given by 
\[     f_{n}^{a} = \sum_{k \in \NN_{0}} (-1)^{\frac{k(k+1)}{2} + k n} f_{n+k} \circ p_{k,n}^{a},     \]
for $n \in \NN$. 
The local finiteness hypothesis on $\{ f_{k+n} \circ p_{k,n}^{a} \}_{k \in \NN_{0}}$ for each $n \in \NN$ tells us that the previous expression is well-defined. 
Furthermore, the previous collection of maps defines a morphism of augmented $A_{\infty}$-algebras, since $\mathrm{MI}(n)^{m_{\bullet}^{a}}_{l}$ coincides with 
\[      \sum_{k \in \NN_{0}} (-1)^{\frac{k(k+1)}{2} + k n} \mathrm{MI}(k+n)^{m_{\bullet}^{A}}_{l} \circ p_{k,n}^{a},     \]
and $\mathrm{MI}(n)^{m_{\bullet}^{a}}_{r}$ coincides with 
\[      \sum_{k \in \NN_{0}} (-1)^{\frac{k(k+1)}{2} + k n} \mathrm{MI}(k+n)^{m_{\bullet}^{A}}_{r} \circ p_{k,n}^{a}.     \]
Again, the compatibility assumption implies that the two previous sums have finite support. 

Let $M$ be an $A_{\infty}$-bimodule over an augmented $A_{\infty}$-algebra $A$, and $a \in A^{(1,0_{G'})}$ an admissible element satisfying the Maurer-Cartan equation on $A$. 
Given $k',k'' \in \NN_{0}$ and $n', n'' \in \NN_{0}$, let us define 
\[       p_{k',n',k''n,''}^{a} : A^{\otimes n'} \otimes M \otimes A^{\otimes n''} \rightarrow A^{\otimes (k'+n')} \otimes M \otimes A^{\otimes (k''+n'')}     \]
as $p_{k',n'}^{a} \otimes \mathrm{id}_{M} \otimes p_{k'',n''}^{a}$, where the first and last tensor factors morphisms were introduced at the beginning of this subsection. 
We say that the element $a$ is \emph{admissible on $M$} if for each $n', n'' \in \NN_{0}$, the family of maps $\{ m_{n'+k',n''+k''} \circ p_{k',n',k'',n''}^{a} \}_{k',k'' \in \NN_{0}}$ is locally finite. 
In this case, we can twist the augmented $A_{\infty}$-bimodule structure on $M$ by defining 
\[     m_{n',n''}^{a} = \sum_{k', k'' \in \NN_{0}} (-1)^{\frac{(k'+k'')(k'+k''+1)}{2} + (k'+k'') (n'+n''+1)} m_{n'+k',n''+k''} \circ p_{k',n',k'',n''}^{a},     \] 
for $n', n'' \in \NN_{0}$. 
The admissibility assumption on $M$ implies that the previous multiplication morphisms are well-defined, and the proof that these maps 
provide an $A_{\infty}$-bimodule structure on $M$ over the twisted augmented $A_{\infty}$-algebra $(A, m^{a}_{\bullet})$, which we shall denote by $(M,m_{\bullet,\bullet}^{a})$, 
is completely analogous to the case of $A_{\infty}$-algebras. 
It will be usually called the \emph{twisted $A_{\infty}$-bimodule of $(M,m_{\bullet,\bullet})$ by $a$}. 
Notice that these definitions coincide with the corresponding ones for bimodules over augmented dg algebras, in the case that $m_{i',i''}$ vanishes for $i' + i'' \in \NN$.

We may generalize the previous situation when the admissibility and compatibility conditions are not satisfied (\textit{cf.} \cite{LH}, Ch. 6, Section 2). 
We will first need some definitions, which are in some sense standard. 
A \emph{topology} on a graded $k$-module $M$ is a decreasing sequence $\{ M_{i} \}_{i \in \NN_{0}}$ of graded $k$-sumodules of $M$, \textit{i.e.} we have 
\[     M_{0} \supseteq M_{1} \supseteq \dots \supseteq M_{i} \supseteq \dots      \]
(\textit{cf.} \cite{Bour2}, Chap. III, \S 2, n$^{\circ}$ 5). 
We understand each graded $k$-submodule $M_{i}$ as a \emph{neighbourhood} of the zero element of $M$, and we say that $M$ is a topological graded $k$-module. 
We say that the topology is \emph{Hausdorff} if $\cap_{i \in \NN_{0}} M_{i} = 0$. 
On the other hand, we regard $k$ with the discrete topology, \textit{i.e.} the one that comes from a filtration for which there exists $i \in \NN_{0}$ such that $k_{i} = \{ 0 \}$. 
A \emph{morphism of topological graded $k$-modules} $M \rightarrow N$ is a morphism of graded $k$-modules (of certain degree) which is continuous for the corresponding topologies, 
\textit{i.e.} if for each $i \in \NN_{0}$ there exists $j \in \NN_{0}$ such that $f(M_{j}) \subseteq N_{i}$. 
We further say that it is \emph{contracting} if $f(M_{i}) \subseteq N_{i}$, for all $i \in \NN_{0}$. 
Given two topological graded $k$-modules $M$ and $N$, the tensor product $M \otimes N$ has a topology given by 
\[     (M \otimes N)_{i} = \sum_{i_{1} + i_{2} \geq i} M_{i_{1}} \otimes N_{i_{2}},     \]
for $i \in \NN_{0}$. 

A \emph{topological augmented $A_{\infty}$-algebra} is an augmented $A_{\infty}$-algebra such that the underlying graded $k$-module has a Hausdorff topology satisfying that all morphisms $m_{i}$, for $i \in \NN$, 
$\eta_{A}$ and $\epsilon_{A}$ are contracting morphisms of topological graded $k$-modules. 
A \emph{morphism of topological augmented $A_{\infty}$-algebras} $f_{\bullet} : A \rightarrow B$ is a morphism of augmented $A_{\infty}$-algebras such that $f_{i}$ 
is a contracting morphisms of topological graded $k$-modules for all $i \in \NN$.  
An element $a \in A^{(1,0_{G'})}$ satisfies the topological Maurer-Cartan equation if it lies in the neighbourhood $A_{1}$ and the sum 
\begin{equation}
\label{eq:mc1topo}
     \sum_{n \in \NN} (-1)^{\frac{n(n+1)}{2}} m_{n}(a, \dots, a)      
\end{equation} 
converges to zero. 
In this case, we may also twist the topological augmented $A_{\infty}$-algebra structure on $A$ by defining 
\[     m_{n}^{a} = \sum_{k \in \NN_{0}} (-1)^{\frac{k(k+1)}{2} + k n} m_{n+k} \circ p_{k,n}^{a},     \] 
for $n \in \NN$, provided that the previous sums are convergent. 
This new structure of topological augmented $A_{\infty}$-algebras is called the \emph{twist of $A$}.  

Let $f_{\bullet} : A \rightarrow B$ be a morphism of topological augmented $A_{\infty}$-algebras and $a \in A^{(1,0_{G'})}$ an element which satisfies the topological Maurer-Cartan equation. 
We say that $a$ is \emph{compatible} with $f_{\bullet}$ if the following sum
\begin{equation}
\label{eq:mc2topo}
     \sum_{n \in \NN} (-1)^{\frac{n(n+1)}{2} + 1} f_{n}(a, \dots, a),      
\end{equation} 
which we are going to denote by $b$, converges in $B$, and it satisfies that the sum defined by the left member of the Maurer-Cartan equation converges. 
We shall call $b$ \emph{the image of $a$ under $f_{\bullet}$}. 
As in the previous case, it is simple to show that the Maurer-Cartan equation is satisfied by $b$. 
In this case, one may also twist the morphism $f_{\bullet} : A \rightarrow B$ of topological augmented $A_{\infty}$-algebra to a new morphism $f_{\bullet}^{a}$ of topological augmented $A_{\infty}$-algebras from $(A,m_{\bullet}^{A,a})$ to $(B, m_{\bullet}^{B,b})$ 
given by 
\[     f_{n}^{a} = \sum_{k \in \NN_{0}} (-1)^{\frac{k(k+1)}{2} + k n} f_{n+k} \circ p_{k,n}^{a},     \]
for $n \in \NN$, provided the sum converges.  

Analogously, a \emph{topological $A_{\infty}$-bimodule} is an  $A_{\infty}$-bimodule such that the underlying graded $k$-module has a Hausdorff topology satisfying that all morphisms $m_{i',i''}$, for $i', i'' \in \NN_{0}$, 
are contracting morphisms of topological graded $k$-modules. 
It is trivial to extend the topological definitions of twisted structure for the case of $A_{\infty}$-bimodules. 


The most typical example of the previous setting 
is the following. 
Let $A$ be an augmented dg algebra and $C$ an Adams connected coaugmented $A_{\infty}$-coalgebra. 
If $J_{C}$ is concentrated in Adams positive degrees, let us define $C_{n}$ the subspace of $C$ given by elements of Adams degree less than or equal to $n$, for $n \in \NN_{0}$. 
If  $J_{C}$ is concentrated in Adams negative degrees, let us define $C_{n}$ the subspace of $C$ given by elements of Adams greater than or equal to $-n$, for $n \in \NN_{0}$. 
Set $C_{-1} = 0$.
By Adams degree considerations we see that the comultiplication $\Delta_{i}$ sends $C_{n}$ to $C_{n}^{\otimes i}$, for all $i \in \NN$ and $n \in \NN$. 
This implies in particular that $C_{n}$ is a coaugmented $A_{\infty}$-coalgebra such that the canonical inclusions $C_{n} \rightarrow C_{m}$ for $n \leq m$ and $C_{n} \rightarrow C$ for $n \in \NN_{0}$ 
are strict morphisms of coaugmented $A_{\infty}$-coalgebras.  
We consider the augmented $A_{\infty}$-algebra structure on $\mathcal{H}om(C,A)$. 
We define a topology on it by defining $\mathcal{H}om(C,A)_{i}$ to be the subset of $\mathcal{H}om(C,A)$ given by sums of maps which vanish on $C_{i-1}$. 
It is easy to see that this defines a Hausdorff decreasing filtration of graded $k$-submodules of $\mathcal{H}om(C,A)$. 
In this case, a \emph{(generalized or homotopical) twisting cochain} from $C$ to $A$ is an element $\tau \in \mathcal{H}om(C,A)^{(1,0_{G'})}_{1}$ satisfying the topological Maurer-Cartan equation such that $\tau \circ \eta_{C} $ and $\epsilon_{A} \circ \tau$ vanish. 
Note that the fact that $\tau$ satisfies the topological Maurer-Cartan equation equation means that 
\[       d_{A} \circ \tau + \sum_{i \in \NN} (-1)^{i(i+1)/2+1} \mu_{A}^{(i)} \circ \tau^{\otimes i} \circ \Delta_{i} = 0,     \]
where the convergence recalled four paragraphs above is a consequence of the fact that the family of morphisms appearing in the previous sum is locally finite.  
The same occurs for the sums defining the twisted multiplications $m_{i}^{\tau}$. 

We recall in this case that, given a dg $A$-bimodule over $A$, the tensor product $M \otimes C$ becomes an $A_{\infty}$-bimodule. 
We consider a topology on it by defining the filtration $\{ (M \otimes C)_{i} \}_{i \in \NN_{0}}$ as follows. 
If $J_{C}$ is concentrated in Adams positive (resp., negative) degrees, we set $(M \otimes C)_{i}$ to be the subspace spanned by tensors $m \otimes c$ where $c$ is homogeneous of 
Adams degree greater (resp., less) than or equal to $i$ (resp., $-i$), for $i \in \NN_{0}$. 

The augmented $A_{\infty}$-algebra structure on $\mathcal{H}om(C,A)$ twisted by a topological twisting cochain $\tau$ will be denoted by $\mathcal{H}om^{\tau}(C,A)$. 
Note that, if $\operatorname{Tw}(C,A)$ denotes the set of topological twisting cochains from $C$ to $A$, 
we have a canonical map 
\begin{equation}
\label{eq:isostw2}
     \operatorname{Hom}_{\text{aug-dg-alg}} (\Omega^{+}(C),A) \rightarrow \operatorname{Tw}(C,A),     
\end{equation}
given by $g \mapsto g \circ \tau^{C}$, where $\tau^{C} : C \rightarrow \Omega^{+}(C)$ is given by the composition of the canonical projection $C \rightarrow C/\operatorname{Im}(\eta_{C})$, $s_{C/\operatorname{Im}(\eta_{C})[-1]}^{-1}$ and the canonical inclusion 
of $C/\operatorname{Im}(\eta_{C})[-1]$ inside $\Omega^{+}(C)$. 
It is clear that the map \eqref{eq:isostw2} is a bijection (see \cite{Prou}, Lemme 3.17). 
Furthermore, by means of the previous morphism we can define the \emph{composition (topological) twisting cochain} of a morphism of coaugmented $A_{\infty}$-coalgebras $f_{\bullet} : C' \rightarrow C$ with 
a twisting cochain $\tau$ from $C$ to $A$. 
Indeed, if $F_{\tau} \in \operatorname{Hom}_{\text{aug-dg-alg}} (\Omega^{+}(C),A)$ is the morphism such that $F_{\tau} \circ \tau^{C} = \tau$, and $\Omega^{+}(f_{\bullet})$ is the morphism of augmented dg algebras 
from $\Omega^{+}(C')$ to $\Omega^{+}(C)$, the composition twisting cochain $\tau \circ f_{\bullet}$ is defined as $F_{\tau} \circ \Omega^{+}(f_{\bullet}) \circ \tau^{C'}$. 
We see in fact that the twisting cochain defined by this composition is exactly the same as the expression \eqref{eq:mc2}. 
This means in particular that the sum given there converges.  
Moreover, the convergence of the sums appearing in the definition of twisting of the morphism $f_{\bullet}$ of coaugmented $A_{\infty}$-coalgebras follows from the locally finite assumption on the latter.  
On the other hand, for $M$ a dg bimodule over $A$, the twisted $A_{\infty}$-bimodule structure over $\mathcal{H}om^{\tau}(C,A)$ on the tensor product $M \otimes C$ is usually denoted by $M \otimes_{\tau} C$, 
where the convergence assumptions hold trivially. 

\subsection{{\texorpdfstring{Hochschild cohomology of nonnegatively graded connected algebras}{Hochschild cohomology of nonnegatively graded connected algebras}}}
\label{subsec:hhainf}

We recall the following theorem, which must be well-known to the experts. 
It was announced by B. Keller at the X ICRA of Toronto, Canada, in 2002.
\begin{theorem}
\label{theorem:keller}
Let $C$ be a minimal coaugmented $A_{\infty}$-coalgebra and $A$ be a nonnegatively (Adams) graded connected algebra, which we regard in zero (co)homological degree. 
Then, the following are equivalent:
\begin{itemize}
\item[(i)] There is a quasi-isomorphism of augmented minimal $A_{\infty}$-algebras 
\[     \mathcal{E}xt_{A}^{\bullet}(k,k) \rightarrow C^{\#}.     \]
\item[(ii)] There is a twisting cochain $\tau : C \rightarrow A$ such that the twisted tensor product ${}_{\epsilon_{A}}A \otimes_{\tau} C$ is a minimal projective resolution of the trivial left $A$-module $k$, 
where ${}_{\epsilon}A$ denotes the $A$-bimodule structure on $A$ with the action induced by the augmentation $\epsilon_{A}$ of $A$ on the left and with the standard action on the right.
\item[(iii)] There is a twisting cochain $\tau : C \rightarrow A$ such that the twisted tensor product $A^{e} \otimes_{\tau} C$ is a minimal projective resolution of the regular $A$-bimodule $A$.
\end{itemize}
\end{theorem}
\begin{proof}
A short proof of the equivalence between (i) and (ii) was given in \cite{Hers}, Thm. 4.7. 
Moreover, the condition (iii) implies (ii) by a standard argument. 
Indeed, since $A^{e} \otimes_{\tau} C$ is a minimal projective resolution of $A$, the cone 
of the quasi-isomorphism $A^{e} \otimes_{\tau} C \rightarrow A$ is exact. 
Since an exact complex of projective left $A$-modules is homotopically trivial, there exists a 
contracting homotopy $s$ for it, which is $A$-linear.
Then $\mathrm{id}_{k} \otimes_{A} s$ is a contracting homotopy of the cone of 
$k \otimes_{A} (A^{e} \otimes_{\tau} C) \simeq {}_{\epsilon_{A}}A \otimes_{\tau} C \rightarrow k \otimes_{A} A \simeq k$. 
To prove that condition (ii) implies (iii), we only have to show that $A^{e} \otimes_{\tau} C$ is exact in positive homological degrees. 
This follow from \cite{BM}, Prop. 4.1.
\end{proof}

Note that, following our conventions, the differential of ${}_{\epsilon_{A}}A \otimes_{\tau} C$ is the map $m_{0,0}^{{}_{\epsilon_{A}}A \otimes_{\tau} C}$ given by the twisting of the multiplication $m_{0,0}^{{}_{\epsilon_{A}}A \otimes C}$, and its explicit expression is 
\[     \mathrm{id}_{A} \otimes \Delta_{1} + \sum_{i \in \NN_{\geq 2}} (-1)^{\frac{i(i+1)}{2}+1} (\mu_{A}^{(i)} \otimes \mathrm{id}_{C}) \circ (\mathrm{id}_{A} \otimes \tau^{\otimes (i-1)} \otimes \mathrm{id}_{C}) 
\circ (\mathrm{id}_{A} \otimes \Delta_{i}).     \]

Let us suppose that $A$ be a nonnegatively graded connected algebra and that $C$ is a minimal Adams connected coaugmented $A_{\infty}$-coalgebra quasi-isomorphic to $B^{+}(A)$, 
so we have a morphism $f_{\bullet} : C \rightarrow B^{+}(A)$ of coaugmented $A_{\infty}$-coalgebras. 
Under our assumptions this is tantamount to satisfying item (i) of the previous theorem. 
Note that $B^{+}(A)$ is also Adams connected, for $A$ is nonnegatively graded and connected. 
By the comments in the thirteenth paragraph of the previous subsection we get a quasi-isomorphism $\mathcal{H}om (B^{+}(A),A) \rightarrow \mathcal{H}om (C,A)$ 
of topological augmented $A_{\infty}$-algebras. 
Moreover, taking a look at the proof of the previous theorem given in the reference, it is easy to check that the previous quasi-isomorphism sends $\tau_{A}$ to the twisting cochain $\tau$ mentioned 
in the item (ii) of the same theorem. 
The comments in the last three paragraphs of the previous subsection tell us that we have a morphism 
$\mathcal{H}om^{\tau_{A}} (B^{+}(A),A) \rightarrow \mathcal{H}om^{\tau} (C,A)$ of augmented $A_{\infty}$-algebras. 
Moreover, its first component coincides with the map given by applying the functor $\mathcal{H}om_{A^{e}} (\place , A)$ to the comparison morphism $A^{e} \otimes_{\tau} C \rightarrow A^{e} \otimes_{\tau_{A}} B^{+}(A)$ of resolutions 
of the $A$-bimodule $A$ induced by $f_{\bullet}$ and explained in the last paragraph of Subsection \ref{subsec:gena}. 
The comparison map is indeed a quasi-isomorphism by the equivalence between conditions (i) and (iii) in the previous theorem. 
We conclude that the first component of the morphism 
$\mathcal{H}om^{\tau_{A}} (B^{+}(A),A) \rightarrow \mathcal{H}om^{\tau} (C,A)$ of augmented $A_{\infty}$-algebras is in fact a quasi-isomorphism. 
In particular, we may compute the graded algebra structure on the Hochschild cohomology $HH^{\bullet}(A)$ by means of the graded algebra structure 
$H^{\bullet}(\mathcal{H}om^{\tau} (C,A))$ induced by the multiplication $m_{2}$ of the latter. 
The case of Hochschild homology is treated in a similar fashion. 

The comments in the previous paragraph can be summarized in the following result.
\begin{theorem}
\label{theorem:final}
Let $A$ nonnegatively (Adams) graded connected algebra, which we regard in zero (co)homological degree, and let $C$ be a minimal coaugmented $A_{\infty}$-coalgebra such that there is a quasi-isomorphism of augmented minimal $A_{\infty}$-algebras 
\[     \mathcal{E}xt_{A}^{\bullet}(k,k) \rightarrow C^{\#}.     \]
Then, there exists a quasi-isomorphism of $A_{\infty}$-algebras between 
$\mathcal{H}om^{\tau} (C,A)$ and the cochain complex computing the Hochschild 
cohomology of $A$, which in particular induces an isomorphism of graded algebras 
between $H^{\bullet}(\mathcal{H}om^{\tau} (C,A))$ and $HH^{\bullet}(A)$. 
Moreover, using the previous quasi-isomorphism of $A_{\infty}$-algebras, 
there exists a quasi-isomorphism of $A_{\infty}$-bimodules between 
$A \otimes^{\tau} C$ and the cochain complex computing the Hochschild 
cohomology of $A$, which in particular induces an isomorphism of graded bimodules 
between $H_{\bullet}(A \otimes^{\tau} C)$ and $HH_{\bullet}(A)$. 
\end{theorem}

The first part of the previous theorem gives a different proof of the main result of \cite{XX}, Thm. 3.2, just taking into account that the $A_{\infty}$-coalgebra $C$ is just the dual of the one given in \cite{HeL}, Thm. 6.5, 
and the twisting cochain is given by minus the composition of the canonical projection $C \rightarrow V$ and the canonical inclusion $V \rightarrow A$.  
However, our result may be applied to compute the graded algebra structure of Hochschild cohomology of nonnegatively graded connected algebras under much more general situation, 
\textit{e.g.} if $A$ is multi-Koszul in the sense of \cite{Hers}, by applying Thm. 4.8 of that article (taking into account however that the definition of Maurer-Cartan equation has a different sign convention). 

\bibliographystyle{model1-num-names}

\begin{bibdiv}
\begin{biblist}

\bib{AH}{article}{
   author={Avramov, Luchezar},
   author={Halperin, Stephen},
   title={Through the looking glass: a dictionary between rational homotopy
   theory and local algebra},
   conference={
      title={Algebra, algebraic topology and their interactions (Stockholm,
      1983)},
   },
   book={
      series={Lecture Notes in Math.},
      volume={1183},
      publisher={Springer},
      place={Berlin},
   },
   date={1986},
   pages={1--27},
}

\bib{AFH}{article}{
  title={Differential graded homological algebra},
  author={Avramov, Luchezar},
  author={Foxby, Hans-Bjorn},
  author={Halperin, Stephen},
  note={Preprint},
}

\bib{BM}{article}{
   author={Berger, Roland},
   author={Marconnet, Nicolas},
   title={Koszul and Gorenstein properties for homogeneous algebras},
   journal={Algebr. Represent. Theory},
   volume={9},
   date={2006},
   number={1},
   pages={67--97},
}

\bib{Bour}{book}{
   author={Bourbaki, N.},
   title={\'El\'ements de math\'ematique. Alg\`ebre. Chapitre 10. Alg\`ebre
   homologique},
   language={French},
   note={Reprint of the 1980 original [Masson, Paris; MR0610795]},
   publisher={Springer-Verlag},
   place={Berlin},
   date={2007},
   pages={viii+216},
}

\bib{Bour2}{book}{
   author={Bourbaki, N.},
   title={\'El\'ements de math\'ematique. Fascicule XXVIII. Alg\`ebre
   commutative. Chapitre 3: Graduations, filtra- tions et topologies.
   Chapitre 4: Id\'eaux premiers associ\'es et d\'ecomposition primaire},
   language={French},
   series={Actualit\'es Scientifiques et Industrielles, No. 1293},
   publisher={Hermann, Paris},
   date={1961},
   pages={183},
}

\bib{BGSS}{article}{
   author={Buchweitz, Ragnar-Olaf},
   author={Green, Edward L.},
   author={Snashall, Nicole},
   author={Solberg, {\O}yvind},
   title={Multiplicative structures for Koszul algebras},
   journal={Q. J. Math.},
   volume={59},
   date={2008},
   number={4},
   pages={441--454},
}

\bib{CE}{book}{
   author={Cartan, Henri},
   author={Eilenberg, Samuel},
   title={Homological algebra},
   publisher={Princeton University Press},
   place={Princeton, N. J.},
   date={1956},
   pages={xv+390},
}

\bib{FT}{article}{
   author={Fe{\u\i}gin, B. L.},
   author={Tsygan, B. L.},
   title={Cyclic homology of algebras with quadratic relations, universal
   enveloping algebras and group algebras},
   conference={
      title={$K$-theory, arithmetic and geometry},
      address={Moscow},
      date={1984--1986},
   },
   book={
      series={Lecture Notes in Math.},
      volume={1289},
      publisher={Springer},
      place={Berlin},
   },
   date={1987},
   pages={210--239},
}

\bib{FHT}{article}{
   author={F{\'e}lix, Yves},
   author={Halperin, Steve},
   author={Thomas, Jean-Claude},
   title={Differential graded algebras in topology},
   conference={
      title={Handbook of algebraic topology},
   },
   book={
      publisher={North-Holland},
      place={Amsterdam},
   },
   date={1995},
   pages={829--865},
}

\bib{FHT01}{book}{
   author={F{\'e}lix, Yves},
   author={Halperin, Stephen},
   author={Thomas, Jean-Claude},
   title={Rational homotopy theory},
   series={Graduate Texts in Mathematics},
   volume={205},
   publisher={Springer-Verlag},
   place={New York},
   date={2001},
   pages={xxxiv+535},
}

\bib{FMT}{article}{
   author={F{\'e}lix, Yves},
   author={Menichi, Luc},
   author={Thomas, Jean-Claude},
   title={Gerstenhaber duality in Hochschild cohomology},
   journal={J. Pure Appl. Algebra},
   volume={199},
   date={2005},
   number={1-3},
   pages={43--59},
}

\bib{FTVP}{article}{
   author={Felix, Yves},
   author={Thomas, Jean-Claude},
   author={Vigu{\'e}-Poirrier, Micheline},
   title={The Hochschild cohomology of a closed manifold},
   journal={Publ. Math. Inst. Hautes \'Etudes Sci.},
   number={99},
   date={2004},
   pages={235--252},
}

\bib{Fuk}{article}{
   author={Fukaya, Kenji},
   title={Deformation theory, homological algebra and mirror symmetry},
   conference={
      title={Geometry and physics of branes},
      address={Como},
      date={2001},
   },
   book={
      series={Ser. High Energy Phys. Cosmol. Gravit.},
      publisher={IOP, Bristol},
   },
   date={2003},
   pages={121--209},
}

\bib{FOOO1}{book}{
   author={Fukaya, Kenji},
   author={Oh, Yong-Geun},
   author={Ohta, Hiroshi},
   author={Ono, Kaoru},
   title={Lagrangian intersection Floer theory: anomaly and obstruction.
   Part I},
   series={AMS/IP Studies in Advanced Mathematics},
   volume={46},
   publisher={American Mathematical Society, Providence, RI; International
   Press, Somerville, MA},
   date={2009},
   pages={xii+396},
}


\bib{Ger}{article}{
   author={Gerstenhaber, Murray},
   title={The cohomology structure of an associative ring},
   journal={Ann. of Math. (2)},
   volume={78},
   date={1963},
   pages={267--288},
}

\bib{Ge}{article}{
   author={Getzler, Ezra},
   title={Lie theory for nilpotent $L_\infty$-algebras},
   journal={Ann. of Math. (2)},
   volume={170},
   date={2009},
   number={1},
   pages={271--301},
}


\bib{Gin}{article}{
   author={Ginzburg, Victor},
   title={Lectures on noncommutative geometry},
   type={Preprint},
   date={2005},
   eprint={arXiv:math/0506603v1},
}

\bib{HeL}{article}{
   author={He, Ji-Wei},
   author={Lu, Di-Ming},
   title={Higher Koszul algebras and $A$-infinity algebras},
   journal={J. Algebra},
   volume={293},
   date={2005},
   number={2},
   pages={335--362},
}

\bib{Hers}{article}{
   author={Herscovich, Estanislao},
   title={On the multi-Koszul property for connected algebras},
   journal={Doc. Math.},
   volume={18},
   date={2013},
   pages={1301--1347},
}


\bib{HMS}{article}{
   author={Husemoller, Dale},
   author={Moore, John C.},
   author={Stasheff, James},
   title={Differential homological algebra and homogeneous spaces},
   journal={J. Pure Appl. Algebra},
   volume={5},
   date={1974},
   pages={113--185},
}



\bib{Ke}{article}{
   author={Keller, Bernhard},
   title={Deriving DG categories},
   journal={Ann. Sci. \'Ecole Norm. Sup. (4)},
   volume={27},
   date={1994},
   number={1},
   pages={63--102},
}

\bib{Kel}{article}{
   author={Keller, Bernhard},
   title={Derived invariance of higher structures on the Hochschild complex},
   date={2003},
   eprint={http://www.math.jussieu.fr/~keller/publ/dih.dvi},
   note={Preprint},
}


\bib{LH}{thesis}{
   author={Lef\`evre-Hasegawa, Kenji},
   title={Sur les $A_{\infty}$-cat\'egories},
   language={French},
   type={Ph.D. Thesis},
   place={Paris},
   date={2003},
   note={Corrections at \url{http://www.math.jussieu.fr/\~keller/lefevre/TheseFinale/corrainf.pdf}},
}

\bib{Lo}{article}{
   author={L\"ofwall, Clas},
   title={Hochschild homology},
   language={Swedish},
   type={Lecture Notes},
   date={1995},
   note={Preprint},
}

\bib{LPWZ04}{article}{
   author={Lu, D. M.},
   author={Palmieri, J. H.},
   author={Wu, Q. S.},
   author={Zhang, J. J.},
   title={$A_\infty$-algebras for ring theorists},
   booktitle={Proceedings of the International Conference on Algebra},
   journal={Algebra Colloq.},
   volume={11},
   date={2004},
   number={1},
   pages={91--128},
}

\bib{LPWZ08}{article}{
   author={Lu, Di Ming},
   author={Palmieri, John H.},
   author={Wu, Quan Shui},
   author={Zhang, James J.},
   title={Koszul equivalences in $A_\infty$-algebras},
   journal={New York J. Math.},
   volume={14},
   date={2008},
   pages={325--378},
}

\bib{LPWZ09}{article}{
   author={Lu, D.-M.},
   author={Palmieri, J. H.},
   author={Wu, Q.-S.},
   author={Zhang, J. J.},
   title={$A$-infinity structure on Ext-algebras},
   journal={J. Pure Appl. Algebra},
   volume={213},
   date={2009},
   number={11},
   pages={2017--2037},
}

\bib{Me}{article}{
   author={Menichi, Luc},
   title={Batalin-Vilkovisky algebra structures on Hochschild cohomology},
   date={2007},
   type={Preprint},
   eprint={arXiv.org/pdf/0711.1946v2.pdf},
}

\bib{NT}{article}{
   author={Nest, Ryszard},
   author={Tsygan, Boris},
   title={On the cohomology ring of an algebra},
   conference={
      title={Advances in geometry},
   },
   book={
      series={Progr. Math.},
      volume={172},
      publisher={Birkh\"auser Boston},
      place={Boston, MA},
   },
   date={1999},
   pages={337--370},
}

\bib{PP}{book}{
   author={Polishchuk, Alexander},
   author={Positselski, Leonid},
   title={Quadratic algebras},
   series={University Lecture Series},
   volume={37},
   publisher={American Mathematical Society},
   place={Providence, RI},
   date={2005},
   pages={xii+159},
}

\bib{Prou}{article}{
   author={Prout{\'e}, Alain},
   title={$A_\infty$-structures. Mod\`eles minimaux de Baues-Lemaire et
   Kadeishvili et homologie des fibrations},
   language={French},
   note={Reprint of the 1986 original;
   With a preface to the reprint by Jean-Louis Loday},
   journal={Repr. Theory Appl. Categ.},
   number={21},
   date={2011},
   pages={1--99},
}

\bib{TT}{article}{
   author={Tamarkin, Dmitri},
   author={Tsygan, Boris},
   title={The ring of differential operators on forms in noncommutative
   calculus},
   conference={
      title={Graphs and patterns in mathematics and theoretical physics},
   },
   book={
      series={Proc. Sympos. Pure Math.},
      volume={73},
      publisher={Amer. Math. Soc.},
      place={Providence, RI},
   },
   date={2005},
   pages={105--131},
}

\bib{T}{article}{
   author={Tsygan, Boris},
   title={Cyclic homology},
   conference={
      title={Cyclic homology in non-commutative geometry},
   },
   book={
      series={Encyclopaedia Math. Sci.},
      volume={121},
      publisher={Springer},
      place={Berlin},
   },
   date={2004},
   pages={73--113},
}

\bib{VP}{article}{
   author={Vigu{\'e}-Poirrier, Micheline},
   title={Homologie de Hochschild et homologie cyclique des alg\`ebres
   diff\'erentielles gradu\'ees},
   language={French},
   note={International Conference on Homotopy Theory (Marseille-Luminy,
   1988)},
   journal={Ast\'erisque},
   number={191},
   date={1990},
   pages={7, 255--267},
}

\bib{W}{book}{
   author={Weibel, Charles A.},
   title={An introduction to homological algebra},
   series={Cambridge Studies in Advanced Mathematics},
   volume={38},
   publisher={Cambridge University Press},
   place={Cambridge},
   date={1994},
   pages={xiv+450},
}

\bib{XX}{article}{
   author={Xu, Yunge},
   author={Xiang, Huali},
   title={Hochschild cohomology rings of $d$-Koszul algebras},
   journal={J. Pure Appl. Algebra},
   volume={215},
   date={2011},
   number={1},
   pages={1--12},
}

\end{biblist}
\end{bibdiv}
\addcontentsline{toc}{section}{References}



\end{document}